\documentclass[11pt]{article}
\usepackage{latexsym,amsfonts,amssymb,amsmath,amsthm}
\usepackage{graphicx}

\usepackage[usenames,dvipsnames]{color}
\usepackage{ulem}
\usepackage{xcolor}
\usepackage{subfigure}
\usepackage{float}

\usepackage[mathlines]{lineno}

\begin{document}
\setlength{\baselineskip}{16pt}

\parindent 0.5cm
\evensidemargin 0cm \oddsidemargin 0cm \topmargin 0cm \textheight 22.5cm \textwidth 16cm \footskip 2cm \headsep
0cm

\newtheorem{theorem}{Theorem}[section]
\newtheorem{lemma}{Lemma}[section]
\newtheorem{proposition}{Proposition}[section]
\newtheorem{definition}{Definition}[section]
\newtheorem{example}{Example}[section]
\newtheorem{corollary}{Corollary}[section]

\newtheorem{remark}{Remark}[section]

\numberwithin{equation}{section}

\def\p{\partial}
\def\I{\textit}
\def\R{\mathbb R}
\def\C{\mathbb C}
\def\u{\underline}
\def\l{\lambda}
\def\a{\alpha}
\def\O{\Omega}
\def\e{\epsilon}
\def\ls{\lambda^*}
\def\D{\displaystyle}
\def\wyx{ \frac{w(y,t)}{w(x,t)}}
\def\imp{\Rightarrow}
\def\tE{\tilde E}
\def\tX{\tilde X}
\def\tH{\tilde H}
\def\tu{\tilde u}
\def\d{\mathcal D}
\def\aa{\mathcal A}
\def\DH{\mathcal D(\tH)}
\def\bE{\bar E}
\def\bH{\bar H}
\def\M{\mathcal M}
\def\E{\mathcal{E}}
\renewcommand{\labelenumi}{(\arabic{enumi})}

\def\disp{\displaystyle}
\def\undertex#1{$\underline{\hbox{#1}}$}
\def\card{\mathop{\hbox{card}}}
\def\sgn{\mathop{\hbox{sgn}}}
\def\exp{\mathop{\hbox{exp}}}
\def\OFP{(\Omega,{\cal F},\PP)}
\newcommand\JM{Mierczy\'nski}
\newcommand\RR{\ensuremath{\mathbb{R}}}
\newcommand\CC{\ensuremath{\mathbb{C}}}
\newcommand\QQ{\ensuremath{\mathbb{Q}}}
\newcommand\ZZ{\ensuremath{\mathbb{Z}}}
\newcommand\NN{\ensuremath{\mathbb{N}}}
\newcommand\PP{\ensuremath{\mathbb{P}}}
\newcommand\abs[1]{\ensuremath{\lvert#1\rvert}}

\newcommand\normf[1]{\ensuremath{\lVert#1\rVert_{f}}}
\newcommand\normfRb[1]{\ensuremath{\lVert#1\rVert_{f,R_b}}}
\newcommand\normfRbone[1]{\ensuremath{\lVert#1\rVert_{f, R_{b_1}}}}
\newcommand\normfRbtwo[1]{\ensuremath{\lVert#1\rVert_{f,R_{b_2}}}}
\newcommand\normtwo[1]{\ensuremath{\lVert#1\rVert_{2}}}
\newcommand\norminfty[1]{\ensuremath{\lVert#1\rVert_{\infty}}}
\newcommand{\ds}{\displaystyle}

\title{Stability, bifurcation and spikes of  stationary solutions in a chemotaxis system with  singular sensitivity and logistic source}

\author{Halil Ibrahim  Kurt\\
Department of Mathematics and Statistics\\
Auburn University\\
Auburn, AL 36849, USA\\
\\
   Wenxian Shen\\
   Department of Mathematics and Statistics\\
Auburn University\\
Auburn, AL 36849, USA\\
\\
and\\
Shuwen Xue\\
Department of Mathematical Sciences\\
Northern Illinois University\\
DeKalb, IL 60115, USA
 }

\date{}

 \maketitle

\begin{abstract}

In the current paper, we study  stability, bifurcation, and spikes of positive stationary solutions of {the} following parabolic-elliptic  chemotaxis system with singular sensitivity and logistic source,
\begin{equation}
\label{abstract-eq}
\begin{cases}
u_t=u_{xx}-\chi (\frac{u}{v} v_x)_x+u(a-b u), &  0<x<L, \, t>0,\cr
0=v_{xx}- \mu v+ \nu u, &  0<x<L, \, t>0 \cr
u_x(t,0)=u_x(t,L)=v_x(t,0)=v_x(t,L)=0, & t>0, \cr
\end{cases}
\end{equation}
where $\chi$, $a$, $b$, $\mu$, $\nu$ are positive constants.
Among others, we prove there are $\chi^*>0$ and $\{\chi_k^*\}\subset [\chi^*,\infty)$
($\chi^*\in\{\chi_k^*\}$)
such that the constant solution $(\frac{a}{b},\frac{\nu}{\mu}\frac{a}{b})$ of \eqref{abstract-eq} is locally stable when $0<\chi<\chi^*$ and is unstable when $\chi>\chi^*$, and  under some generic condition,  for each $k\ge 1$, a (local) branch of non-constant stationary solutions of \eqref{abstract-eq} bifurcates   from $(\frac{a}{b},\frac{\nu}{\mu}\frac{a}{b})$ when $\chi$ passes through $\chi_k^*$, and global extension of the local bifurcation branch is obtained. We also prove that any sequence of non-constant {positive} stationary solutions
$\{(u(\cdot;\chi_n),v(\cdot;\chi_n))\}$ of \eqref{abstract-eq} with $\chi=\chi_n(\to \infty)$ develops spikes at any $x^*$ satisfying $\liminf_{n\to\infty} u(x^*;\chi_n)>\frac{a}{b}$.
Some numerical analysis is carried out. It is observed numerically that the local bifurcation branch bifurcating from  $(\frac{a}{b},\frac{\nu}{\mu}\frac{a}{b})$ when
$\chi$ passes through $\chi^*$ can be extended to $\chi=\infty$ and the stationary solutions on this  global bifurcation extension are locally stable when $\chi\gg 1$ and develop
spikes as $\chi\to\infty$.
\end{abstract}

\noindent {\bf Keywords}:  Chemotaxis, singular sensitivity,  logistic source,  stationary solutions, stability, local bifurcation,
global bifurcation,  spikes.

\medskip

\noindent {\bf Mathematics Subject Classification.}   35B20, 35B32, 35B40, 35Q92, 92C17, 92D25.

\section{Introduction}

The current paper is devoted to the study of stability, bifurcation, and spikes of positive stationary solutions of {the} following parabolic-elliptic  chemotaxis system with singular sensitivity and logistic source {on a bounded interval $[0,L]$   complemented with Neumann boundary condition},
\begin{equation}
\label{main-eq0}
\begin{cases}
u_t=u_{xx}-\chi (\frac{u}{v} v_x)_x+u(a-b u), &  0<x<L\cr
0=v_{xx}- \mu v+ \nu u, &  0<x<L \cr
u_x(t,0)=u_x(t,L)=v_x(t,0)=v_x(t,L)=0,
\end{cases}
\end{equation}
{where $u(t,x)$ and $v(t,x)$ represent the cellular density and chemical concentration at time $t$ and location $x$, $\chi>0$ is the chemotaxis sensitivity coefficient, $a>0$ is the growth rate of the cell population, $b>0$ is the self-limitation rate of the cell population,  $\mu  > 0$ represents the degradation rate of the chemical signal substance and $\nu >0$ is the production rate of the chemical signal substance by the cell population.}

Chemotaxis models are used to describe the movements of cells or living organisms in response to gradients of some chemical  substances.   Various chemotaxis systems, also known as Keller-Segel  systems, have been
widely studied since the pioneering works   \cite{Keller-0, Keller-00} by Keller and Segel at the beginning of 1970s
on the mathematical modeling of the aggregation process of Dictyostelium discoideum.

Consider the following parabolic-elliptic  chemotaxis system with singular sensitivity and logistic source in general dimensional setting,
\begin{equation}
\label{main-eq-intro}
\begin{cases}
u_t=\Delta u-\chi\nabla\cdot (\frac{u}{v} \nabla v)+ u(a-bu),\quad &x\in \Omega\cr
0=\Delta  v- \mu v+ \nu u,\quad &x\in \Omega \quad \cr
\frac{\p u}{\p n}=\frac{\p v}{\p n}=0,\quad &x\in\p\Omega,
\end{cases}
\end{equation}
where  $\Omega\subset\R^N$ is a bounded  smooth domain, and $\chi, \mu$ and $\nu$ are positive constants, and $a,b$ are nonnegative constants.
A considerable amount of research has been carried out on the global existence and boundedness of positive classical solutions  of \eqref{main-eq-intro}.

For example, in the case that $a=b=0$ and $\mu=\nu=1$,  Biler in \cite{Bil}  proved  the global existence of positive  solutions of \eqref{main-eq-intro}   when $\chi\le 1$ and $N=2$, or
 $\chi<2/N$ and $N\ge 2$.    Fujie,  Winkler, and  Yokota  in \cite{FuWiYo} proved the boundedness of globally defined
positive  solutions  of \eqref{main-eq-intro}   when $\chi<\frac{2}{N}$ and $N\ge 2$. Fujie and Senba in \cite{FuSe1} proved
 the global existence and  boundedness of
classical positive  solutions of \eqref{main-eq-intro}   for the case of $N=2$ for any $\chi>0$.

In the case that $a$ and $b$ are positive constants, when $N=2,$ Fujie, Winkler, and
Yokota in \cite{FuWiYo1} proved that finite-time blow-up does not occur in \eqref{main-eq-intro}, and moreover, if
\begin{equation}
\label{intro-assumption}
    a>
\begin{cases}
\frac{\mu \chi^2}{4} &\text{if $0< \chi \leq 2$}\\
\mu (\chi-1) &\text{if $\chi>2$},
\end{cases}
\end{equation}
then any globally defined positive solution of  \eqref{main-eq-intro}  is bounded.
Furthermore, Cao et al. in \cite{CaWaYu} proved that,    if $$a>2(\sqrt{\chi+1}-1)^2+\frac{\chi^2}{16\eta|\Omega|}$$ with the constant $\eta$ depending  on $\Omega$,  then  for any nonnegative initial data $u_0 \in C(\bar \Omega)$,  {where
\begin{equation}\label{Cfunction}
C(\bar \Omega):=\{u\,|\, u {\rm \,\ is \,\ continuous \,\ on\,\ \bar \Omega }\},
\end{equation}}
and $u_0 \not\equiv 0$ satisfying that $\int_{\Omega} u_0^{-1}<16 \eta b |\Omega|^2/\chi^2$,  the globally defined positive solution of \eqref{main-eq-intro}  converges to $\big( \frac{a}{b},\frac{\nu}{\mu} \frac{a}{b} \big)$ as $t\to\infty$ exponentially (see \cite[Theorem 1]{CaWaYu} for details).

{Let \begin{equation}\label{Cplus}
C^+(\bar \Omega):=\{u\in C(\bar\Omega)\,|\, u(x)\ge 0,\,\, \int_\Omega u(x)dx>0\}.
\end{equation}
Recently,  Kurt and Shen proved in \cite{HKWS} that  classical solutions of \eqref{main-eq-intro} with initial functions $u_0\in C^+(\bar\Omega)$ exist globally and stay bounded as time evolves, provided that $u_0$ is not too small and $a$ is large relative to $\chi$ (see \cite[Theorem 1.2(3)]{HKWS}), which is an interesting biological phenomenon.
 Moreover, some qualitative properties of \eqref{main-eq-intro} have  been obtained in \cite{HKWS2}. For example, assuming that $a$ is large relative to $\chi$,
 it is  shown in \cite{HKWS2}  that any globally defined positive
solution of \eqref{main-eq-intro}  is bounded above and below eventually by some positive constants independent of its initial functions provided that they are not too small.
}

There are still many interesting dynamical issues to be studied for  \eqref{main-eq-intro}. For example, whether any globally defined positive solution of  \eqref{main-eq-intro}  is bounded
without {the assumption that $a$ is large relative to $\chi$}; whether chemotaxis induces non-constant positive stationary solutions and if so, what about the stability of non-constant positive stationary solutions and whether non-constant stationary solutions develop spikes as $\chi\to \infty$, etc. {Some of those issues are strongly related to the stability and bifurcation of
stationary solutions of \eqref{main-eq-intro}.}

{There are many  works on bifurcation and spikes of stationary solutions of various chemotaxis models. For example, local and global bifurcation of constant positive stationary solutions and existence of spiky steady states for chemotaxis models with certain
regular sensitivity and with or without logistic source are studied in
\cite{KaKoWa, KoWeAl, KoWa, KoWeXu, KoWeXu1, WaYaGa, WaXu}, etc.  In particular,
consider
the following chemotaxis system with regular sensitivity and logistic source,
\begin{equation}
\label{regular-eq1}
\begin{cases}
u_t=u_{xx}-\chi (u v_x)_x+u(a-b u), &  0<x<L\cr
0=v_{xx}- \mu v+ \nu u, &  0<x<L \cr
u_x(t,0)=u_x(t,L)=v_x(t,0)=v_x(t,L)=0.
\end{cases}
\end{equation}
The authors of \cite{KoWeXu} proved the existence of single boundary spikes to \eqref{regular-eq1}  when $\chi$ is sufficiently large via the standard Lyapunov–Schmidt
reduction.  Among others, the authors of \cite{WaYaGa}  carried
out bifurcation analysis on
\eqref{regular-eq1} and  obtained the explicit formulas of bifurcation values and small amplitude nonconstant
positive solutions.}

{
There are  also some studies on bifurcation and spikes of stationary solutions of chemotaxis models with singular sensitivity but without logistic source. For example,  in \cite{Li} and \cite{Xu},  local and global bifurcation of positive constant solutions and  existence of spiky steady states are studied
for parabolic-parabolic  chemotaxis models with singular sensitivity and without logistic source.
It is shown in \cite{Li}  that positive monotone steady states exist as long as $\chi$  is larger than
the first bifurcation value $\chi^*$, and  that  the cell density function forms a spike.
The results in \cite{Li} apply to \eqref{main-eq0} with $a=b=0$.  The paper \cite{CaLiWa}
investigated the existence and nonlinear stability of boundary spike-layer solutions of a chemotaxis-consumption type system with logarithmic singular sensitivity in the half space, where the physical zero-flux
and Dirichlet boundary conditions are prescribed.}

{
However,  there is little study on bifurcations and spiky solutions of \eqref{main-eq0} with $a,b>0$.  Bifurcation analysis of \eqref{main-eq0} with $a,b>0$ is of great importance  in several aspects.  For example,   it  will show what types of bifurcations may occur in \eqref{main-eq0} as $\chi$ changes and what types of  spiky patterns may be developed in \eqref{main-eq0} as $\chi\to\infty$. It will discover some  intrinsic  similarity and/or difference between the singular and regular chemotaxis sensitivities. Note that, considering \eqref{main-eq-intro}, for a positive classical solution $(u(t,x),v(t,x))$ to exist globally and stay bounded, it requires that $u(t,x)$ does not become arbitrarily large and $v(t,x)$ (equivalently, $u(t,x)$) does not become arbitrarily small on any finite time interval. Consider the following chemotaxis model with regular sensitivity,
\begin{equation}
\label{main-eq-intro1}
\begin{cases}
u_t=\Delta u-\chi\nabla\cdot (u \nabla v)+ u(a-bu),\quad &x\in \Omega\cr
0=\Delta  v- \mu v+ \nu u,\quad &x\in \Omega \quad \cr
\frac{\p u}{\p n}=\frac{\p v}{\p n}=0,\quad &x\in\p\Omega.
\end{cases}
\end{equation}
For a positive classical solution $(u(t,x),v(t,x))$ of \eqref{main-eq-intro1} to exist globally and stay bounded, it only requires that $u(t,x)$ does not become arbitrarily large on any finite time interval.
It is of great interest to explore intrinsic similarity and difference between singular and regular chemotaxis sensitivity via various analysis.}

In this paper,  {we carry out some bifurcation analysis of \eqref{main-eq0} and some relevant study on stationary solutions of \eqref{main-eq0}}. In particular, we study the local stability and instability of the positive constant solution $(\frac{a}{b},\frac{\nu}{\mu}\frac{a}{b})$ of \eqref{main-eq0};   bifurcation solutions of \eqref{main-eq0}  from  the positive constant solution $(\frac{a}{b},\frac{\nu}{\mu}\frac{a}{b})$;
properties of non-constant positive stationary solutions; spikes developed by non-constant positive stationary solutions.
 Among others, we obtain the following results:

\begin{itemize}

\item[(i)] {\it There is $\chi^*:=\chi^*(a,\mu)>0$ depending only on $a$ and $\mu$ such that when
$0<\chi<\chi^*$, the positive constant solution  $(\frac{a}{b},\frac{\nu}{\mu}\frac{a}{b})$ of \eqref{main-eq0} is locally exponentially  stable with respect to small perturbations in $C([0,L])$, and when $\chi>\chi^*$,  $(\frac{a}{b},\frac{\nu}{\mu}\frac{a}{b})$  is unstable } (see Theorem \ref{local-stable-thm}).

\item[(ii)] {\it There are $\chi_k^*:=\chi_k^*(a,\mu)\ge \chi^*$ ($k=1,2,\cdots$) depending only on $a$ and $\mu$
such that $\inf_{k\ge 1}\chi_k^*=\chi_{k^{*}}^{*}=\chi^*$ for some ${k^*}\ge 1$, $\chi_k^*$ is strictly decreasing in $k\le {k^*}$ and strictly increasing in $k\ge {k^*}$, and under some generic condition,  a branch of non-constant stationary solutions of \eqref{main-eq0}, denoted by
${\tilde \Gamma_k(\chi)}$ , bifurcates   from $(\frac{a}{b},\frac{\nu}{\mu}\frac{a}{b})$
  when $\chi$ passes through $\chi_k^*$ and $\chi\in I_k$, where
 $0<\chi-\chi_k^*\ll 1$ or $0<\chi_k^*-\chi\ll 1$ for $\chi\in I_k$ }
(see Theorem  \ref{local-bif-thm}).

\item[(iii)]  {\it The local bifurcation branches ${\tilde \Gamma_k}=\{(\chi,{\tilde \Gamma_k(\chi)})\,|\, \chi\in I_k\}$
can be extended to global branches ${\tilde {\mathcal{C}}_k}$}   (see Theorem \ref{global-bif-thm} for detail).

\item[(iv)] {\it The $u$-components of non-constant positive stationary solutions {of} \eqref{main-eq0}  stay bounded in
$L^2(0,L)$ and the $v$-components of non-constant positive stationary solutions of \eqref{main-eq0}  stay bounded
in $H^2(0,L)$ as $\chi\to\infty$}  (see Theorem {\ref{property-bif-solu-thm1}}).

\item[(v)] {\it The weak limit of the $u$-components of any sequence of  non-constant positive
stationary solutions  $\{(u(\cdot;\chi_n),v(\cdot;\chi_n))\}$ of \eqref{main-eq0} with $\chi=\chi_n (\to\infty)$   is a constant lying in $[0,\frac{a}{b}]$}  (see Theorem \ref{property-bif-solu-thm}).

\item[(vi)] {\it A notion of spikes is introduced} (see Definition \ref{spike-solution-def})
{\it and it is proved that  any given sequence $\{(u(x;\chi_n),v(x;\chi_n))\}$ of  non-constant positive stationary solutions of \eqref{main-eq0} with $\chi=\chi_n (\to\infty)$ develops spikes as $n\to\infty$  at
those  $x^*$ satisfying  that $\liminf_{n\to\infty} u(x^*;\chi_n)>\frac{a}{b}$} (see Theorem \ref{spiky-solu-thm}) (hence a sequence $\{(u(x;\chi_n),v(x;\chi_n))\}$ of  non-constant positive stationary solutions of \eqref{main-eq0} with $\chi=\chi_n (\to\infty)$
can develop multiple spikes as $n\to\infty$).

\item[(vii)]  {\it Several numerical simulations are carried out and it is observed numerically that
the local bifurcation branch ${\tilde \Gamma_{k^*}(\chi)}$ in (ii) extends to $\chi=\infty$ and the solutions on ${\tilde{\mathcal{C}}_{k^*}}$ are locally stable for $\chi\gg 1$,  stay bounded,  and  develop spikes as $\chi\to\infty$}  (see Numerical Experiment 2 in section 6 for the case $a=b=\mu=\nu=L=1$,  locally stable boundary spiky solutions are observed in this case,
 and
Numerical Experiment 4 in section 6 for the case $a=b=\mu=\nu=1$, $L=6$, locally stable spiky solutions with multiple spikes are observed in this case).
\end{itemize}

The main methods/techniques used and/or developed in this paper  for  the study of stability, bifurcation, and spikes  of positive stationary solutions of \eqref{main-eq0} can be described  as follows. We study local stability/instability of positive stationary solutions of \eqref{main-eq0} by general perturbation theory
for dynamical systems/differential equations (see \cite{Hen}). The investigation of  local bifurcation of the constant solution $(\frac{a}{b},\frac{\nu}{\mu}\frac{a}{b})$ of \eqref{main-eq0} is carried out via invariant manifold theory, in particular,  center manifold theory,  for dynamical systems/differential equations (see \cite{Hen}).    We study the global extension of local bifurcation branches of $(\frac{a}{b},\frac{\nu}{\mu}\frac{a}{b})$ by the global bifurcation theory developed by Shi and Wang
(see \cite{ShWa}).
We investigate the development of spikes of non-constant positive stationary solutions of \eqref{main-eq0} via some important properties of positive stationary solutions of \eqref{main-eq0} obtained in this paper, for example, the property (v) stated in the above.  This is a novel approach. The  above mentioned methods/techniques would be useful for the study of bifurcation and spikes of stationary solutions in other chemotaxis models.

 {We would like to point out that, many people study the existence of nonconstant positive stationary solutions by applying the local bifurcation theory of Crandall-Rabinowtz (see \cite{Crandall-Rabinowitz}),  which can also be used to study \eqref{main-eq0}.}  {In  this paper, we employed the center manifold theory as the theoretical framework. The computations associated with this approach are elementary, and by using this method, we obtain several information simultaneously, such as the bifurcation direction and the explicit dependence of the bifurcating solutions on the parameters.}

\smallskip

{
We make the following remarks  on the implications of the results obtained in this paper,
and some related problems, which are worthy to be studied further.
\begin{itemize}
\item[(a)]   Some essential difference is observed  between the effects of the parameters on the dynamics of \eqref{main-eq0} and \eqref{regular-eq1}. For example, consider \eqref{main-eq0} and \eqref{regular-eq1}, and treat $\chi$ as a bifurcation parameter. Let $\chi_k^*$ be as in \eqref{chi-k-star}, that is,
\begin{equation*}
\chi_k^*=\frac{\mu L^2+k^2\pi^2}{k^2\pi^2L^2}\cdot\frac{k^2\pi^2+a L^2}{\mu},\quad k=1,2,\cdots,
\end{equation*}
and $\bar\chi_k^*$ be as in \eqref{tilde-chi-k-star}, that is,
\begin{equation*}
{\bar \chi_k^*}= \frac{\mu L^2+k^2\pi^2}{k^2\pi^2L^2}\cdot\frac{b\big(k^2\pi^2+a L^2\big)}{a\nu},\quad k=1,2,\cdots.
\end{equation*}
Let $\chi^*$  and $\bar\chi^*$ be as in \eqref{chi-star} and  \eqref{tilde-chi-star}, respectively, that is,
$$
\chi^*=\inf_{k\ge 1} \chi_k^*,\quad \bar\chi^*=\inf_{k\ge 1}\bar\chi_k^*.
$$
 Then  the constant solution $(\frac{a}{b},\frac{\nu}{\mu}\frac{a}{b})$  of \eqref{main-eq0}
(resp. \eqref{regular-eq1}) is locally stable when $0<\chi<\chi^*$ (resp. $0<\chi<\bar\chi^*$)  and unstable when $\chi>\chi^*$ (resp. $\chi>\bar\chi^*$), and  pitchfork bifurcation occurs in \eqref{main-eq0} (resp. \eqref{regular-eq1}) when $\chi$ passes through $\chi_k^*$ (resp. $\bar\chi_k^*$) (see (i), (ii) described in the above and \cite{WaYaGa}). Hence the larger $\chi^*$ (resp. $\bar\chi^*$), the less influence of the chemotaxis on the dynamics of \eqref{main-eq0} (resp. \eqref{regular-eq1}).

 \item[(b)] Observe that
$\chi_k^*$ and $\chi^*$ depend only on the parameters $a$ and $\mu$ in \eqref{main-eq0}, and
$\bar\chi_k^*$ and $\bar\chi^*$ depend on all the parameters $a,b,\mu,\nu$ in \eqref{regular-eq1}.   Hence there is some essential difference between the influence of the chemotaxis on the dynamics of  \eqref{main-eq0} and \eqref{regular-eq1}.
Mathematically,  this difference  is due to the following two factors: first,  $b$ is cancelled
in $\frac{u}{v}$ when linearizing \eqref{main-eq0} at the constant solution $(\frac{a}{b},\frac{\nu}{\mu}\frac{a}{b})$;
second,  $\nu$ can be assumed to be $1$ by changing $u$ to $\nu u$ (after such change of variable,
$b$ becomes $\frac{b}{\nu}$). Note that
$$
 \lim_{a\to\infty} \chi_k^*=\infty, \quad \lim_{\mu\to 0^{+}} \chi_ k^*=\infty,
$$
and
$$
\lim_{a\to 0^{+}}\bar\chi_k^*=\infty,\quad \lim_{b\to\infty}\bar \chi_k^*=\infty,\quad \lim_{\nu\to 0^{+}} \bar \chi_k^*=\infty.
$$
 Biologically,  the difference between $\chi^*$ and $\bar\chi^*$ can be interpreted as follows. In the chemotaxis model \eqref{main-eq0} with singular sensitivity, the chemotaxis has less influence on the dynamics of \eqref{main-eq0} if $a$ is large or $\mu$ is small, which is natural since large $a$ and small $\mu$  prevent $v$ from becoming too small as time evolves. In the chemotaxis model
\eqref{regular-eq1} with regular sensitivity, the chemotaxis has less influence on the dynamics of \eqref{regular-eq1} if $a$ is small, $b$ is large, or $\nu$ is small, which is also natural
since small $a$, large $b$, and small $\nu$  prevent $v$ from becoming too large as time evolves.
The above intrinsic difference between singular and regular chemotaxis sensitivities resulting from the local bifurcation analysis matches in certain sense the difference between the following sufficient conditions for
the global existence of positive classical solutions of \eqref{main-eq-intro} and \eqref{main-eq-intro1}:
a positive classical solution $(u(t,x),v(t,x))$ of \eqref{main-eq-intro} exists globally if $a$ is large relative to $\chi$ and the initial distribution $u(0,x)$ is not too small (see \cite[Theorem 1.2(3)]{HKWS}), and a positive classical solution $(u(t,x),v(t,x))$ of \eqref{main-eq-intro1} exists globally if $b$ is large relative to $\chi$ (see \cite[Theorem 2.5]{TeWi1}).

\item[(c)] Besides rigorous bifurcation analysis near the constant solution
$(\frac{a}{b},\frac{\nu}{\mu}\frac{a}{b})$,  in this paper, we
also  perform some  theoretical as well as numerical analysis of non-constant stationary solutions of \eqref{main-eq0}  for $\chi\gg 1$
and obtain several interesting results (see (iii)-(vii) described in the above). For example, it is observed that,  up to a subsequence,
as $\chi_n\to\infty$,  a sequence $ \{(u(\cdot;\chi_n),v(\cdot;\chi_n))\}$ of nonconstant stationary solutions of \eqref{main-eq0}
develops spikes {{at any $x^*$ satisfying $\liminf_{n\to\infty} u(x^*;\chi_n)>\frac{a}{b}$}}.  It is interesting to further study spiky solutions of \eqref{main-eq0}  with $\chi\gg 1$, for example, to derive
asymptotic expansions of spiky patterns with $\chi\gg 1$,  to  perform some bifurcation analysis around
spiky patterns, etc.  We leave such questions  for future investigation. It should be pointed out that
 spiky patterns of logistic
Keller–Segel models with regular sensitivity has been studied  in \cite{KoWeAl}, \cite{KoWeXu} and \cite{KoWeXu1}.

\item[(d)]
Consider the following two parabolic-parabolic chemotaxis models
\begin{equation}
\label{para-para-regular}
\begin{cases}
u_t=u_{xx}-\chi (u v_x)_x+u(a-b u), &  0<x<L\cr
v_t=v_{xx}- \mu v+ \nu u, &  0<x<L \cr
u_x(t,0)=u_x(t,L)=v_x(t,0)=v_x(t,L)=0,
\end{cases}
\end{equation}

\begin{equation}
\label{para-para-singular}
\begin{cases}
u_t=u_{xx}-\chi (\frac{u}{v} v_x)_x+u(a-b u), &  0<x<L\cr
v_t=v_{xx}- \mu v+ \nu u, &  0<x<L \cr
u_x(t,0)=u_x(t,L)=v_x(t,0)=v_x(t,L)=0.
\end{cases}
\end{equation}

We want to mention that the system \eqref{para-para-regular} has very rich spatial-temporal dynamics as demonstrated by the numerical studies in \cite{PainterHillen}. For example, the authors in \cite{PainterHillen} numerically demonstrated that the long time dynamics of solutions of \eqref{para-para-regular} undergo homogeneous stationary solution, stationary spatial patterns in which multiple-peak patterns develop, spatial-temporal periodic solutions and spatial-temporal irregular solutions which describes a form of spatial-temporal chaos as $\chi$ is increased steadily (see section 5 of \cite{PainterHillen} for details).

Based on some numerical simulations we performed, complicated dynamics is not observed
in  \eqref{para-para-singular} for various parameter sets.
For example,    following the same simulations performed in section 5 of \cite{PainterHillen} and using the same parameter set, we did some numerical simulations for the system  \eqref{para-para-singular}. We did not observe  such rich dynamics.  Indeed,  we only observed homogeneous stationary solution and different peak patterns depending on randomised initial conditions as $\chi$ is increased steadily. Spatial-temporal periodic solutions and spatial-temporal irregular solutions are not observed in \eqref{para-para-singular}. We also used the parameter sets we used in Section 6.1 and Section 6.2 to simulate the existence of nonconstant solutions for the system \eqref{para-para-singular}. For the parameter set used in section 6.1, we observed the same phenomenon as demonstrated in section 6.1. For the parameter set in section 6.2, in addition to the phenomenon we observed in section 6.2, we also see the boundary perks. Time-dependent solutions are also not observed. The above numerical simulations indicate that the dynamics of system \eqref{para-para-singular} is simpler than that of \eqref{para-para-regular}. It is certainly important to further explore whether \eqref{para-para-singular} exhibits complicated  dynamics
or not.
 We also  leave this question  for future investigation.

\end{itemize}
}

The rest of this paper is organized as follows.  Section 2 is devoted to the study {of} the local stability and instability of the constant solution $(\frac{a}{b},\frac{\nu}{\mu}\frac{a}{b})$ of \eqref{main-eq0}. In section 3, we study local bifurcation of  $(\frac{a}{b},\frac{\nu}{\mu}\frac{a}{b})$ and stability of the local bifurcation solutions. The global extension of
the local bifurcation is investigated in section 4. Several important properties of non-constant positive stationary solutions of \eqref{main-eq0} are also obtained in section 4.
Section 5 is devoted to the study of spiky stationary solutions of \eqref{main-eq0}.
In the last section, we provide some numerical analysis on the bifurcation and spiky solutions of \eqref{main-eq0}.

\section{Local stability and instability of the positive constant solution}

In this section, we analysis the  local stability and instability of the positive constant solution
$(\frac{a}{b},\frac{\nu}{\mu}\frac{a}{b})$ of \eqref{main-eq0} by  general perturbation theory.

The following is the main theorem in this section.

\begin{theorem}
\label{local-stable-thm}
There is $\chi^*=\chi^*(a,\mu)$ depending only on $a$ and $\mu$ such that
when $0<\chi<\chi^*$, the constant solution $(\frac{a}{b},\frac{\nu}{\mu}\frac{a}{b})$
of  \eqref{main-eq0} is locally {asymptotically} stable  with respect to perturbations in $C^+([0,L])$ and when $\chi>\chi^*$,  $(\frac{a}{b},\frac{\nu}{\mu}\frac{a}{b})$ is unstable with respect to perturbations in $C^+([0,L])$.
\end{theorem}

\begin{proof}
{Recall that for any $u_0\in C^+([0,L)]$, \eqref{main-eq0} has a unique globally defined classical solution $(u(t,x;u_0),v(t,x;u_0))$ satisfying that
$$
\lim_{t\to 0^{+}} \|u(t,\cdot;u_0)-u_0\|_{{C([0,L])}}=0
$$
and $u(t,\cdot;u_0)$ is continuous in $(t,u_0)\in [0,\infty)\times C^+([0,L])$ {(see \cite[Theorem 1.2]{HKWS})}.}
To prove the local stability of the constant solution $(\frac{a}{b},\frac{\nu}{\mu}\frac{a}{b})$, consider the linearized equation of \eqref{main-eq0} at $(\frac{a}{b},\frac{\nu}{\mu}\frac{a}{b})$ in $C([0,L])$,
\begin{equation}
\label{main-linear-eq1}
\begin{cases}
u_t=u_{xx}+(\chi\mu -a)u-\frac{\chi \mu^2}{\nu} v, &  0<x<L\cr
0=v_{xx}-\mu v+\nu u, &  0<x<L\cr
u_x(t,0)=u_x(t,L)=v_x(t,0)=v_x(t,L)=0.
\end{cases}
\end{equation}
Consider the eigenvalue problem associated with \eqref{main-linear-eq1} in the space {$C([0,L])$},
\begin{equation}
\label{main-ev-eq1}
\begin{cases}
u_{xx}+(\chi\mu -a)u-\frac{\chi \mu^2}{\nu} v=\lambda u, &  0<x<L\cr
0=v_{xx}-\mu v+\nu u, & 0<x<L\cr
{u_x(0)=u_x(L)=v_x(0)=v_x(L)=0.}
\end{cases}
\end{equation}

Let $A: \{u\in C^2([0,L])\,|\, u_x(0)=u_x(L)=0\}\to C([0,L])$ be defined by
$$
A u=u_{xx}+(\chi\mu-a)u-\frac{\chi\mu^2}{\nu} v,
$$
where $v$ is the unique solution of
\begin{equation}
\label{new-new-v-eq}
\begin{cases}
0=v_{xx}-\mu v+\nu u,\quad 0<x<L\cr
v_x(0)=v_x(L)=0.
\end{cases}
\end{equation}
Note that the spectrum of the operator $A$, denoted by $\sigma(A)$,  consists of eigenvalues of \eqref{main-ev-eq1}.

Suppose that $\lambda$ is an eigenvalue of \eqref{main-ev-eq1} and $u$ is a corresponding eigenfunction.
Let
$$
u=\sum_{k=0}^\infty u_k \cos(\frac{k \pi x}{L}),\quad v=\sum_{k=0}^\infty v_k \cos(\frac{k\pi x}{L}).
$$
Then we have
\begin{equation}
\label{main-ev-eq2}
-\frac{k^2\pi^2}{L^2}u_k+(\chi\mu -a)u_k-\frac{\chi\mu^2}{\frac{k^2\pi^2}{L^2}+\mu}
u_k=\lambda u_k,\quad  k=0,1,2,\cdots.
\end{equation}
This implies that $\lambda=\lambda_k$ and $u=u_k\cos\frac{k\pi x}{L}$ for some $k\in\mathbb{Z}^+$ and $u_k\in\mathbb{R}\setminus\{0\}$, where
\begin{equation}
\label{lambda-k-eq}
\lambda_k(\chi,a,\mu):=-\frac{k^2\pi^2+aL^2}{L^2}+\frac{\chi\mu k^2\pi^2}{\mu L^{2}+k^2\pi^2}, \quad { k=0,1,2,\cdots}.
\end{equation}
It then follows that
$$
\sigma(A)=\{\lambda_k(\chi,a,\mu)\,|\, k=0,1,2,\cdots\}
$$
and $u=\phi_k=\cos\frac{k\pi x}{L}$ is an eigenfunction of $A$ associated to the eigenvalue $\lambda_k(\chi,a,\mu)$.

It is clear that
$$
\lambda_0=-a<0.
$$
Let
\begin{equation}
\label{chi-k-star}
\chi_k^*=\frac{\mu L^2+k^2\pi^2}{k^2\pi^2L^2}\cdot\frac{k^2\pi^2+a L^2}{\mu},\quad k=1,2,\cdots.
\end{equation}
Then
$$
\lambda_k(\chi,a,\mu)\begin{cases} <0& {\rm for}\,\,\chi<\chi_k^*\cr
=0 &{\rm for}\,\, \chi=\chi_k^*\cr
>0 & {\rm for}\,\, \chi>\chi_k^*
\end{cases}
$$
for $k=1,2,\cdots.$
Let
\begin{equation}
\label{chi-star}
\chi^*=\inf_{k\ge 1}\chi_k^*.
\end{equation}
Then for $\chi<\chi^*$, all the eigenvalue of \eqref{main-ev-eq1} are negative. Hence
$(\frac{a}{b},\frac{\nu}{\mu}\frac{a}{b})$ is linearly exponentially stable with respect to perturbations in $C([0,L])$, where $C([0,L])$ is defined as in \eqref{Cfunction}.
By the general perturbation theory (see \cite[Theorem 5.1.1]{Hen}), $(\frac{a}{b},\frac{\nu}{\mu}\frac{a}{b})$ is a locally {exponentially} stable solution of \eqref{main-eq0} with respect to {small} perturbations in {$C([0,L])$}.
If $\chi>\chi^*$,  then $(\frac{a}{b},\frac{\nu}{\mu}\frac{a}{b})$ is linearly unstable with respect to perturbations in $C([0,L])$.
By the general perturbation theory again (see \cite[Theorem 5.1.1]{Hen}), $(\frac{a}{b},\frac{\nu}{\mu}\frac{a}{b})$ is an unstable solution of \eqref{main-eq0} with respect to {small} perturbations in {$C([0,L])$}.
\end{proof}

\begin{remark}
\label{chi-star-rk}
\begin{itemize}
\item[(1)]
{$\lambda_k$  depends only  on $\chi,a,\mu$
and $\chi_k^*$  depends only  on $a,\mu$.}
Let
 $$h(k)=\frac{\mu L^2+k^2\pi^2}{k^2\pi^2\mu}\cdot\frac{k^2\pi^2+a L^2}{L^2}.
$$
We have that $h(k)$ is monotone decreasing in $k$ for $0<k\le \frac{L}{\pi}(a\mu)^{1/4}$ and is monotone increasing in $k$ for $k\ge \frac{L}{\pi}(a\mu)^{1/4}$.  Let
$$
k_*=\Big\lfloor {\frac{L}{\pi}(a\mu)^{1/4}}\Big\rfloor
$$
be  the greatest integer less than or equal to $\frac{L}{\pi}(a\mu)^{1/4}$.
{Then} we have the following explicit formula for $\chi^*$,
\begin{equation*}
\chi^*=\begin{cases}
\frac{\mu L^2+\pi^2}{\pi^2\mu}\cdot\frac{\pi^2+a L^2}{L^2}\quad &{\rm if}\,\, k_*=0\cr\cr
\frac{\mu L^2+k_*^2\pi^2}{k_*^2\pi^2\mu}\cdot\frac{k_*^2\pi^2+a L^2}{L^2} \quad &{\rm if}\,\, k_*\ge 1,\,\, k_*=\frac{L}{\pi}(a\mu)^{1/4}\cr\cr
\min\{ \frac{\mu L^2+k_*^2\pi^2}{k_*^2\pi^2\mu}\cdot\frac{k_*^2\pi^2+a L^2}{L^2}, \frac{\mu L^2+(k_*+1)^2\pi^2}{(k_*+1)^2\pi^2\mu}\cdot\frac{(k_*+1)^2\pi^2+a L^2}{L^2}\}\quad &{\rm if}\quad k_*\ge 1,\,\, k_*\not =\frac{L}{\pi}(a\mu)^{1/4}.
\end{cases}
\end{equation*}
{Note that
$$
 h\Big(\frac{L}{\pi}(a\mu)^{1/4}\Big)\le \chi^*\le h(1).
$$
 Hence
\begin{equation*}
    \left(1+\sqrt{\frac{a}{\mu}}\right)^2 \leq \chi^* \leq \left(1+\sqrt{\frac{a}{\mu}}\right)^2+\frac{ \pi^2}{\mu L^2}+\frac{aL^2}{\pi^2}.
\end{equation*}}

\item[(2)]
$k_*$ can be any nonnegative integer by properly choosing $a,\mu$ and $L$.
For example,  {fix $b=\nu=1$,} if $a=\mu=L=1$, then $k_*=0$ and
$\chi^*=\chi_1^*\approx 11.9709$;  if $a=\mu=1$ and $L=6$,
then $k_*=1$ and $\chi^{*}=\chi_2^*\approx 4.0085${; if}
$a=2$, $\mu=3$,  and $L=6$,  {then $k_{*}=2$ and} $\chi^*=\chi_3^*\approx 3.2997$.
\end{itemize}
\end{remark}

\medskip

\begin{remark}
\label{stability-constant-solu-rk}
\begin{itemize}
\item[(1)]
It can be proved that when $0<\chi\ll1$,  the constant solution $(\frac{a}{b},\frac{\nu}{\mu}\frac{a}{b})$
of  \eqref{main-eq0} is globally asymptotically  stable (see \cite[Theorem 1]{CaWaYu} for details).
$(\frac{a}{b},\frac{\nu}{\mu}\frac{a}{b})$ may not be globally asymptotically stable
when $\chi<\chi^*$ and $0<\chi^*-\chi\ll 1$ (see (2) and (3) in the following).

\item[(2)] For the case $a=b=\mu=\nu=L=1$, \eqref{main-eq0} experiences super-critical pitchfork bifurcation when $\chi$ passes through $\chi^*$ (see Remark \ref{remark-on-bifurcation} (1)).
It is observed numerically that when $0<\chi<\chi^*$,
the constant solution $(\frac{a}{b},\frac{\nu}{\mu}\frac{a}{b})$
of  \eqref{main-eq0} is globally stable (see {Numerical Experiment 1}).  But it remains open whether the constant solution $(\frac{a}{b},\frac{\nu}{\mu}\frac{a}{b})$
of  \eqref{main-eq0} is truly globally stable for any $0<\chi<\chi^*$.

\item[(3)] For the case $a=b=\mu=\nu=1$ and $L=6$, \eqref{main-eq0} experiences sub-critical pitchfork bifurcation when $\chi$ passes through $\chi^*$ (see  Remark \ref{remark-on-bifurcation} (2)). It is observed numerically that  when $0<\chi<\chi^*$ and $0<\chi^*-\chi\ll 1$, there are some locally stable nonconstant positive steady-states of \eqref{main-eq0} and hence
the constant solution $(\frac{a}{b},\frac{\nu}{\mu}\frac{a}{b})$
of  \eqref{main-eq0} is not globally stable (see {Numerical Experiment 3}).

\item[(4)] For the case $a=2$, $b=\nu=1$, $\mu=3$ and $L=6$, \eqref{main-eq0} also experiences sub-critical pitchfork bifurcation when $\chi$ passes through $\chi^*$ (see  Remark \ref{remark-on-bifurcation} (3)). It is also observed numerically that  when $0<\chi<\chi^*$ and $0<\chi^*-\chi\ll 1$, there are some locally stable nonconstant positive steady-states of \eqref{main-eq0} and hence
the constant solution $(\frac{a}{b},\frac{\nu}{\mu}\frac{a}{b})$
of  \eqref{main-eq0} is not globally stable {(see section 6.4).}

\end{itemize}
\end{remark}

\begin{remark}
\label{regular-sensitivity-rk1}
Consider \eqref{regular-eq1}.
The linearized equation of \eqref{regular-eq1} reads as follows
\begin{equation}
\label{regular-linear-eq1}
\begin{cases}
u_t=u_{xx}+(\frac{\chi\nu a}{b} -a)u-\frac{\chi \mu a}{b} v, &  0<x<L\cr
0=v_{xx}-\mu v+\nu u, &  0<x<L\cr
u_x(t,0)=u_x(t,L)=v_x(t,0)=v_x(t,L)=0.
\end{cases}
\end{equation}
The associated  eigenvalue problem to \eqref{regular-linear-eq1}  {reads} as
\begin{equation}
\label{regular-ev-eq1}
\begin{cases}
u_{xx}+(\frac{\chi\nu a}{b} -a)u-\frac{\chi \mu a}{b} v=\lambda u, &  0<x<L\cr
0=v_{xx}-\mu v+\nu u, & 0<x<L\cr
u_x(t,0)=u_x(t,L)=v_x(t,0)=v_x(t,L)=0.
\end{cases}
\end{equation}
We have
\begin{align*}
\bar \lambda_k&=-\frac{\pi^2 k^2}{L^2}+\big(\frac{\chi \nu a}{b}-a\big)-
\frac{\chi \mu a}{b}\frac{\nu}{\mu +\frac{\pi^2 k^2}{L^2}}\\
&=-\frac{\pi^2 k^2+aL^2}{L^2}+\frac{\chi \nu a \pi^2  k^2}{ b(L^2 \mu +\pi^2 k^2)}, \quad { k=0,1,2,\cdots}
\end{align*}
 is an algebraic simple eigenvalue
with $\cos\frac{k \pi x}{L}$ being a corresponding eigenfunction. Note that $\bar \lambda_0=-a$.
Let
\begin{equation}
\label{tilde-chi-k-star}
{\bar \chi_k^*}= \frac{\mu L^2+k^2\pi^2}{k^2\pi^2L^2}\cdot\frac{b\big(k^2\pi^2+a L^2\big)}{a\nu},\quad k=1,2,\cdots
\end{equation}
and
\begin{equation}
\label{tilde-chi-star}
\bar \chi^*=\inf_{k\ge 1}{\bar \chi_k^*}.
\end{equation}
Then the constant solution $(\frac{a}{b},\frac{\nu}{\mu}\frac{a}{b})$ is locally asymptotically stable when $0<\chi<\bar\chi^*$ and unstable when $\chi>\bar \chi^*$. {Detailed bifurcation analysis of \eqref{regular-eq1} can be seen in \cite{WaYaGa} over a one-dimensional bounded region and \cite{JinWangZhang} over a multidimensional bounded domain.}
It is seen that {$\bar \chi_k^{*}$} depends {on} all the parameters in \eqref{regular-eq1} and then
$\bar \chi^*$ depends {on} all the parameters in \eqref{regular-eq1}, {while $\chi^*$ defined in \eqref{chi-star} depends only on $a$ and $\mu$ (see Remark \ref{chi-star-rk} (1)).}
\end{remark}

\section{Local bifurcation and stability of bifurcating solutions}

In this section, we investigate the local bifurcation of \eqref{main-eq0} when $\chi$ passes through $\chi_k^*$, where $\chi_k^*$ is given in \eqref{chi-k-star}.

To state our main results on local bifurcation solutions from the constant solution $(\frac{a}{b},\frac{\nu}{\mu}\frac{a}{b})$, we first introduce some notions.  For given $k_0\ge 1$, let
\begin{equation}
\label{a-0-k-eq}
a_{0}=-\frac{b}{2a}, \quad
a_{2k_0}=\frac{1}{\lambda_{2k_0}(\chi_{k_0}^*,a,\mu)}\Big(\frac{b}{2}-\frac{\chi_{k_0}^* \mu b  k_0^4 \pi^4}{a(\mu L^2+k_0^2\pi^2)^2}\Big).
\end{equation}
{Let
\begin{equation}\label{alpha-k0-eq}
\alpha_{k_0}=\frac{\mu k_0^2\pi^2}{\mu L^2 +k_0^2 \pi^2}
\end{equation}
and
\begin{equation}
\label{beta-k-0-eq}
\beta_{k_0}=b\Big( 2a_{0}+a_{2k_0}+\frac{\chi_{k_0}^* \mu k_0^4\pi^4}{\nu a (\mu L^2+k_0^2\pi^2)(\mu L^2+4k_0^2\pi^2)} a_{2k_0}
+\frac{\chi_{k_0}^* \mu^2 b k_0^4 \pi^4 L^2}{4 a^2 (\mu L^2+k_0^2\pi^2)^3}\Big).
\end{equation} }

We now state the main results of this section.

\begin{theorem}
\label{local-bif-thm}
\begin{itemize}
\item[(1)] (Pitchfork bifurcation) For  given $k_0\ge 1$,   if $\beta_{k_0}\not =0$ and $\lambda_k(\chi_{k_0}^*,a,\mu)\not =0$ for any $k\not =k_0$, {where $\beta_{k_0}$ and $\lambda_k(\chi_{k_0}^*,a,\mu)$ are given in \eqref{beta-k-0-eq} and \eqref{lambda-k-eq}, respectively,} then  pitchfork bifurcation occurs  in \eqref{main-eq0} near $(\frac{a}{b},\frac{\nu}{\mu}\frac{a}{b})$ when $\chi$ passes through $\chi_{k_0}^*$.  Moreover,  let $\Gamma_{k_0}(\chi)$ be the  $u$-components of  bifurcation solutions  of \eqref{main-eq0} for
$\chi\approx \chi_{k_0}^*$, we have
\begin{equation}
\label{two-branches-eq}
\Gamma_{k_0}(\chi)=\Gamma_{k_0}^+(\chi)\cup \Gamma_{k_0}^-(\chi),
\end{equation}
where, for any $u\in \Gamma_{k_0}^+(\chi)$, $u$  is of the form
\begin{equation}
\label{upper-branch-eq}
u(x)=\begin{cases}
\frac{a}{b}+\sqrt{\frac{\alpha_{k_0}(\chi- \chi_{k_0}^*)}{\beta_{k_0}}}\cos\frac{k_0\pi x}{L}+O(\chi-\chi_{k_0}^*)\,\,{\rm for}\,\,  0<\chi-\chi_{k_0}^*\ll 1\,\,  {\rm if}\,\, \beta_{k_0}>0\cr\cr
\frac{a}{b}+\sqrt{\frac{\alpha_{k_0} (\chi-\chi_{k_0}^*)}{\beta_{k_0}}}\cos\frac{k_0\pi x}{L}+O(\chi-\chi_{k_0}^*)\,\,{\rm for}\,\, 0<\chi_{k_0}^*-\chi\ll 1\,\, {\rm if}\,\, \beta_{k_0}<0,
\end{cases}
\end{equation}
{where $\alpha_{k_0}$ is given in \eqref{alpha-k0-eq},} and for any $u\in \Gamma_{k_0}^-(\chi)$, $u$ is of the form
\begin{equation}
\label{lower-branch-eq}
u(x)=\begin{cases}
\frac{a}{b}-\sqrt{\frac{\alpha_{k_0}(\chi- \chi_{k_0}^*)}{\beta_{k_0}}}\cos\frac{k_0\pi x}{L}+O(\chi-\chi_{k_0}^*)\,\,{\rm for}\,\,  0<\chi-\chi_{k_0}^*\ll 1\,\,  {\rm if}\,\, \beta_{k_0}>0\cr\cr
\frac{a}{b}-\sqrt{\frac{\alpha_{k_0} (\chi-\chi_{k_0}^*)}{\beta_{k_0}}}\cos\frac{k_0\pi x}{L}+O(\chi-\chi_{k_0}^*)\,\,{\rm for}\,\, 0<\chi_{k_0}^*-\chi\ll 1\,\, {\rm if}\,\, \beta_{k_0}<0.
\end{cases}
\end{equation}

\item[(2)] (Stability/instability of bifurcation solutions) Assume  the conditions in (1). If  $\chi_{k_0}^*>\chi^*$, then the bifurcation solutions  with $u\in \Gamma_{k_0}(\chi)$  are unstable.

\item[(3)] (Stability/instability of bifurcation solutions)
If  $\chi_{k_0}^*=\chi^*$  and  $\beta_{k_0}>0$, then super-critical pitchfork bifurcation occurs when $\chi$ passes through $\chi^*$ and the bifurcation solutions  with $u\in \Gamma_{k_0}{(\chi)}$ are locally stable.
If  $\chi_{k_0}^*=\chi^*$  and  $\beta_{k_0}<0$, then sub-critical pitchfork  bifurcation occurs when $\chi$ passes through $\chi^*$ and the bifurcation solutions  with $u\in \Gamma_{k_0}{(\chi)}$ are unstable.
\end{itemize}
\end{theorem}

\begin{remark}
\label{remark-on-bifurcation}
\begin{itemize}
\item[(1)]
For the case $a=b=\mu=\nu=L=1$, we have
$$
 \chi^*=\chi_1^*\approx 11.9709,\quad \chi_k^*>\chi^*\quad \forall\, k\ge 2,
\quad {\rm and}\,\,\,
\beta_1\approx {0.4144} >0. \quad 
$$
Hence supercritical pitchfork bifurcation occurs  when $\chi$ passes through $\chi^*$.

\item[(2)] For the case $a=b=\mu=\nu=1$ and $L=6$, we have
$$\chi^*=\chi_2^*\approx 4.0085, \quad \chi_k^*>\chi^*\quad \forall\, k\ge 1,\,\, k\not =2,
\quad {\rm and}\,\,\,
\beta_2\approx {-0.3109} <0. \quad 
$$
Hence subcritical pitchfork bifurcation occurs for $\chi<\chi^*$, $0<\chi^*-\chi\ll 1$.

\item[(3)]
For the case
$a=2$, $b=\nu=1$, $\mu=3$,  and $L=6$, we have
$$\chi^*=\chi_3^*\approx 3.2997,\quad {\chi_k^*>\chi_3^*\quad \forall\, k\ge 1,\,\, k\not = 3,}
\quad {\rm and}\,\,\,
\beta_3\approx { -0.1344 <0}.
$$
Hence subcritical pitchfork bifurcation occurs for $\chi<\chi^*$, $0<\chi^*-\chi\ll 1$.
\end{itemize}
\end{remark}

Before proving  Theorem \ref{local-bif-thm}, we first make some variable changes.
First, let
$$
u=\frac{a}{b}+\tilde u,\quad v=\frac{\nu}{\mu}\frac{a}{b}+\tilde v.
$$
Then \eqref{main-eq0} becomes
\begin{equation*}
\begin{cases}
\tilde u_t= \tilde u_{xx}-\frac{\chi\mu}{\nu} \Big(\frac{(1+\frac{b}{a}\tilde u)}{(1+\frac{\mu}{\nu}\frac{b}{a}\tilde v)}\tilde v_x\Big)_x-a \tilde u -b\tilde u^2,\quad &0<x<L\cr
0=\tilde v_{xx}-\mu \tilde v+\nu \tilde u,\quad &   0<x<L\cr
\tilde u_x(t,0)=\tilde u_x(t,L)=\tilde v_x(t,0)=\tilde v_x(t,L)=0.
\end{cases}
\end{equation*}
This implies that
\begin{equation*}
\begin{cases}
\tilde u_t= \tilde u_{xx}-\frac{\chi\mu}{\nu} \Big(\big(1+\frac{b}{a}\tilde u\big)
\big(1- \frac{\mu}{\nu}\frac{b}{a}\tilde v+\frac{\mu^2}{\nu^2}\frac{b^2}{a^2}\tilde v^2-\frac{\mu^3}{\nu^3}\frac{b^3}{a^3}\tilde v^3+O({\tilde v^3})\big) \tilde v_x\Big)_x-a\tilde u-b\tilde u^2,& 0<x<L\cr
0=\tilde v_{xx}-\mu \tilde v+\nu \tilde u,&  0<x<L\cr
\tilde u_x(t,0)=\tilde u_x(t,L)=\tilde v_x(t,0)=\tilde v_x(t,L)=0,
\end{cases}
\end{equation*}
that is,
\begin{equation}
\label{main-eq2}
\begin{cases}
\tilde u_t= \tilde u_{xx}-\frac{\chi\mu}{\nu} \Big(\big(1+\frac{b}{a}\tilde u -\frac{\mu}{\nu}\frac{b}{a}\tilde v  -\frac{\mu}{\nu}\frac{b^2}{a^2}\tilde u\tilde v +\frac{\mu^2}{\nu^2}\frac{b^2}{a^2}\tilde v^2 + O(|\tilde u|^3+|\tilde v|^3)\big)\tilde v_x  \Big)_x -a\tilde u-b\tilde u^2,& 0<x<L\cr
0=\tilde v_{xx}-\mu \tilde v+\nu \tilde u,& 0<x<L\cr
\tilde u_x(t,0)=\tilde u_x(t,L)=\tilde v_x(t,0)=\tilde v_x(t,L)=0.
\end{cases}
\end{equation}

Next, let
\begin{equation}
\label{change-variable-eq1}
\begin{cases}
\tilde u=\displaystyle \sum_{i=0}^\infty u_i(t)\cos(\frac{i \pi x}{L})\cr
\tilde v=\displaystyle \sum_{i=0}^\infty v_i(t)\cos(\frac{i\pi x}{L}).
\end{cases}
\end{equation}
{Note that for any $t>0$, $\tilde u(t,\cdot),\tilde v(t,\cdot)\in C^2([0,L])$.  Hence we have
$$
\tilde u_{x}(t,x)=-\sum_{i=0}^{\infty} \frac{i\pi }{L} u_i(t)\sin(\frac{i\pi x}{L}),\quad  \tilde v_{x}(t,x)=-\sum_{i=0}^{\infty} \frac{i\pi }{L} v_i(t)\sin(\frac{i\pi x}{L})
$$
and
$$
\tilde u_{xx}(t,x)=-\sum_{i=0}^{\infty} \frac{i^2\pi^2 }{L^2} u_i(t)\cos(\frac{i\pi x}{L}),\quad  \tilde v_{xx}(t,x)=-\sum_{i=0}^{\infty} \frac{i^2\pi^2 }{L^2} v_i(t)\cos(\frac{i\pi x}{L}).
$$ }
Then the first equation in \eqref{main-eq2} becomes
\begin{align}
\label{main-eq3-u}
{\Big(\sum_{i=0}^\infty u_i(t) \cos(\frac{i\pi x}{L})\Big)_t}=& -\sum_{i=0}^\infty \frac{i^2\pi^2}{L^2}u_i \cos \frac{i \pi x}{L}-a\sum_{i=0}^\infty u_i \cos \frac{i \pi x}{L}-b \Big( \sum_{i=0}^\infty u_i \cos \frac{i \pi x}{L}\Big)^2\nonumber\\
&+\frac{\chi\mu}{\nu} \Big( \big(1+\frac{b}{a}\sum_{i=0}^\infty u_i\cos\frac{i\pi x}{L}-\frac{\mu}{\nu}\frac{b}{a}\sum_{i=0}^\infty v_i \cos \frac{i\pi x}{L}\nonumber\\
& -
\frac{\mu }{\nu}\frac{b^2}{a^2}\Big(\sum_{i=0}^\infty u_i\cos \frac{i \pi x}{L}\Big)\Big(\sum_{i=0}^\infty v_i\cos\frac{i \pi x}{L}\Big)
+\frac{\mu^2}{\nu^2}\frac{b^2}{a^2}\Big(\sum_{i=0}^\infty v_i\cos \frac{i \pi x}{L}\Big)^2 \nonumber\\
&+O( \sum_{i=0}^\infty (|u_i|^3+|v_i|^3))\big) (\sum_{i=0}^\infty
\frac{i \pi}{L} v_i \sin\frac {i \pi x}{L})\Big)_x,\quad 0<x<L.
\end{align}
By the second equation in \eqref{main-eq2}, we have
\begin{equation}
\label{main-eq3-v}
v_i(t)=\frac{\nu}{\mu +\frac{i^2 \pi^2}{L^2}} u_i(t),\quad i=0,1,2,\cdots.
\end{equation}

By \eqref{main-eq3-u} and \eqref{main-eq3-v}, we obtain a system of ODEs for $u_k(t)$, $k=0,1,2,\cdots$. To be more precise, we first observe that
 \begin{align}
\label{ode-eq0}
 \int_0^L \Big(\sum _{i=0}^\infty u_i(t)\cos(\frac{i\pi x}{L})\Big)_t\cos(\frac{k\pi x}{L})dx
&=\Big(\int_0 ^L \big(\sum_{i=0}^\infty u_i(t)\cos(\frac{i\pi x}{L})\big)\cos(\frac{k\pi x}{L})dx\Big)_t\nonumber\\
&=\begin{cases} L  \frac{d u_k(t)}{dt}\quad &{\rm if}\quad k=0\cr
\frac{L}{2} \frac{d u_k}{dt}\quad &{\rm if}\quad k\ge 1,
\end{cases}
\end{align}
and
\begin{align}
\label{ode-eq1}
\int_0^L (\sum_{i=0}^\infty u_i \cos \frac{i\pi x}{L})^2 dx
&=\Big(u_0^2+\frac{1}{2} \big(u_1^2+u_2^2+\cdots\big)\Big)L.
\end{align}
We also observe that for $k\ge 1$,
\begin{align}
\label{ode-eq2}
\int_0^L \Big(\sum_{i=0}^\infty \frac{i^2\pi^2}{L^2}u_i\cos\frac{i\pi x}{L}\Big)\cos\frac{k\pi x}{L}dx={\frac{k^2\pi^2}{2L}} u_k,
\end{align}
\begin{align}
\label{ode-eq3}
\int_0^L \Big(\sum_{i=0}^\infty u_i \cos\frac{i \pi x}{L}\Big)\cos\frac{k \pi x}{L} dx=\frac{L}{2} u_k,
\end{align}
\begin{align}
&\int_0^L \Big(\sum_{i=0}^\infty u_i \cos\frac{i\pi x}{L} \Big)^2 \cos \frac{k \pi x}{L}dx
=\frac{L}{4} \Big (\sum_{i-j=k} u_i u_j + \sum_{i-j=-k}  u_i u_j+\sum_{i+j=k} u_iu_j\Big),
\end{align}
\begin{align}
\int_0^L \Big(\sum_{i=0}^\infty \frac{i \pi}{L}v_i \sin\frac{i\pi x}{L}\Big)_x \cos\frac{k \pi x}{L}dx
&=\frac{k^2 \pi^2}{2L} v_k,
\end{align}
\begin{align}
\label{ode-eq5}
&\int_0^L \Big((\sum_{i=0}^\infty u_i \cos\frac{i\pi x}{L})(\sum_{i=0}^\infty \frac{i \pi}{L} v_i \sin \frac{{ i} \pi x}{L})\Big)_x \cos\frac{ k\pi x}{L} dx \nonumber\\
&={\frac{k \pi}{4}}\Big(\sum _{i-j=-k}^\infty \frac{ j\pi}{L} u_i v_j
+\sum_{i+j=k} \frac{j\pi}{L} u_i v_{j} -\sum_{i-j=k}^\infty  \frac{j\pi}{L} u_{i} v_j\Big),
\end{align}
\begin{align}
\label{ode-eq6}
&\int_0^L \Big((\sum_{i=0}^\infty v_i \cos\frac{i\pi x}{L})(\sum_{i=0}^\infty \frac{i \pi}{L} v_i \sin \frac{{ i} \pi x}{L})\Big)_x \cos\frac{ k\pi x}{L} dx \nonumber\\
&={\frac{k \pi}{4}} \big(\sum _{i-j=-k}^\infty \frac{ j\pi}{L} v_i v_j
+\sum_{i+j=k} \frac{j\pi}{L} v_i v_{j} -\sum_{i-j=k}^\infty  \frac{j\pi}{L} v_{i} v_j\big),
\end{align}
\begin{align}
\label{ode-eq7}
&\int_0^L \Big( (\sum_{i=0}^\infty u_i\cos\frac{i\pi x}{L})(\sum_{i=0}^\infty v_i\cos\frac{i\pi x}{L})(\sum_{i=0}^\infty \frac{i\pi x}{L} v_i \sin\frac{i\pi x}{L})\Big)_x \cos\frac{k\pi x}{L}dx\nonumber\\
&= {\frac{k \pi}{8}} \Big[\sum_{i-j+l=k}\frac{l\pi}{L} u_iv_j v_l+\sum_{-i-j+l=k}\frac{l\pi}{L} u_iv_j v_l +\sum_{i+j+l=k}\frac{l\pi}{L} u_iv_j v_l+\sum_{-i+j+l=k}\frac{l\pi}{L} u_iv_j v_l\nonumber\\
&\qquad -\sum_{-i+j-l=k}\frac{l\pi}{L} u_i v_j v_l-\sum_{i+j-l=k}\frac{l\pi}{L}u_i v_j v_l
-\sum_{i-j-l=k}\frac{l\pi}{L} u_i v_j v_l\Big],
\end{align}
and
\begin{align}
\label{ode-eq8}
&\int_0^L \Big((\sum_{i=0}^\infty v_i \cos\frac{i \pi x}{L})^2 (\sum_{i=0}^\infty { \frac{i \pi }{L}v_{i}\sin \frac{i\pi x}{L}})\Big)_x\cos\frac{k\pi x}{L}dx\nonumber\\
&= {\frac{k \pi}{8}} \Big[\sum_{i-j+l=k}\frac{l\pi}{L} v_iv_j v_l+\sum_{-i-j+l=k}\frac{l\pi}{L} v_iv_j v_l +\sum_{i+j+l=k}\frac{l\pi}{L} v_iv_j v_l+\sum_{-i+j+l=k}\frac{l\pi}{L} v_iv_j v_l\nonumber\\
&\qquad -\sum_{-i+j-l=k}\frac{l\pi}{L} v_i v_j v_l-\sum_{i+j-l=k}\frac{l\pi}{L}v_i v_j v_l
-\sum_{i-j-l=k}\frac{l\pi}{L} v_i v_j v_l\Big].
\end{align}

By \eqref{ode-eq0}-\eqref{ode-eq8}, integrating \eqref{main-eq3-u} over $[0,L]$ and then dividing by $L$, we get
\begin{equation}
\label{ode-eq-u0}
\frac{d u_0}{dt}=-a u_0-b u_0^2 -\frac{1}{2} b\sum_{i=1}^\infty u_i^2.
\end{equation}
By \eqref{main-eq3-v} and \eqref{ode-eq0}-\eqref{ode-eq8}, for each $k\ge 1$, multiplying \eqref{main-eq3-u} by $\frac{2}{L}\cos\frac{k\pi x}{L}$ and then integrating over $[0,L]$, we get
\begin{align}
\label{ode-eq-uk}
\frac{du_k}{dt}=& \lambda_k(\chi,a,\mu)u_k -  \frac{b}{2} \Big (\sum_{i-j=k} u_i u_j + \sum_{i-j=-k}  u_i u_j+\sum_{i+j=k} u_iu_j\Big)\nonumber\\
&+ \frac{\chi\mu}{\nu} g_k(u_0,u_1,\cdots ),\quad k=1,2,\cdots,
\end{align}
where {$\lambda_k(\chi,a,\mu)$ is as in \eqref{lambda-k-eq}},
and
\begin{align}
\label{gk-eq1}
&g_{k}(u_0,u_1,u_2,\cdots)\nonumber\\
&=\frac{b}{a} \frac{k \pi}{L} \frac{1}{2} \Big(\sum _{i-j=-k}^\infty \frac{ j\pi}{L} u_i v_j
+\sum_{i+j=k} \frac{j\pi}{L} u_i v_{j} -\sum_{i-j=k}^\infty  \frac{j\pi}{L} u_{i} v_j\Big)\nonumber\\
&\quad -\frac{\mu}{\nu}\frac{b}{a} \frac{k \pi}{L} \frac{1}{2} \Big(\sum _{i-j=-k}^\infty \frac{ j\pi}{L} v_i v_j
+\sum_{i+j=k} \frac{j\pi}{L} v_i v_{j} -\sum_{i-j=k}^\infty  \frac{j\pi}{L} v_{i} v_j\Big)\nonumber\\
&\quad -\frac{\mu}{\nu}\frac{b^2}{a^2}  \frac{k\pi}{L}\frac{1}{4} \Big[\sum_{i-j+l=k}\frac{l\pi}{L} u_iv_j v_l+\sum_{-i-j+l=k}\frac{l\pi}{L} u_iv_j v_l +\sum_{i+j+l=k}\frac{l\pi}{L} u_iv_j v_l+\sum_{-i+j+l=k}\frac{l\pi}{L} u_iv_j v_l\nonumber\\
&\qquad -\sum_{-i+j-l=k}\frac{l\pi}{L} u_i v_j v_l-\sum_{i+j-l=k}\frac{l\pi}{L}u_i v_j v_l
-\sum_{i-j-l=k}\frac{l\pi}{L} u_i v_j v_l\Big]\nonumber\\
&\quad +\frac{\mu^2}{\nu^2}\frac{b^2}{a^2}\frac{k\pi}{L}\frac{1}{4} \Big[\sum_{i-j+l=k}\frac{l\pi}{L} v_iv_j v_l+\sum_{-i-j+l=k}\frac{l\pi}{L} v_iv_j v_l +\sum_{i+j+l=k}\frac{l\pi}{L} v_iv_j v_l+\sum_{-i+j+l=k}\frac{l\pi}{L} v_iv_j v_l\nonumber\\
&\qquad -\sum_{-i+j-l=k}\frac{l\pi}{L} v_i v_j v_l-\sum_{i+j-l=k}\frac{l\pi}{L}v_i v_j v_l
-\sum_{i-j-l=k}\frac{l\pi}{L} v_i v_j v_l\Big]{+O(\sum_{i=0}^{\infty}(u_{i}^{4}+v_{i}^{4}))}
\end{align}
($v_i$ is given by \eqref{main-eq3-v}).
Combing \eqref{ode-eq-u0} and \eqref{ode-eq-uk},
 we have the following system of ODEs for $u_0,u_1,u_2,\cdots$ with $\{u_i\}\in l^2$,
\begin{equation}
\label{ode-system-eq1}
\begin{cases}
\frac{du_0}{dt}=-au_0-bu_0^2-\frac{1}{2} b \sum_{i=1}^\infty u_i^2\cr
\frac{d u_k}{dt}=\lambda_k(\chi,a,\mu) u_k -  \frac{b}{2} \Big (\sum_{i-j=k} u_i u_j + \sum_{i-j=-k}  u_i u_j+\sum_{i+j=k} u_iu_j\Big)\cr
\qquad\quad +\frac{\chi\mu}{\nu} g_k(u_0,u_1,u_2,\cdots),\quad k=1,2,\cdots.
\end{cases}
\end{equation}

To study the bifurcation of \eqref{main-eq0} near $(\frac{a}{b},\frac{\nu}{\mu}\frac{a}{b})$, it then reduces  to study
the bifurcation  of \eqref{ode-system-eq1} from the trivial solutions $u_k=0$ for $k=0,1,2\cdots$.

We now  prove Theorem \ref{local-bif-thm}.

\begin{proof}[Proof of Theorem \ref{local-bif-thm}]
(1) Fix $k_0\ge 1$.   Recall that $\lambda_{k_0}(\chi_{k_0}^*,a,\mu)=0$.
We will use center manifold theory to investigate the bifurcation solutions of \eqref{ode-system-eq1} near the zero solution when $\chi$ passes through $\chi_{k_0}^*$. To this end,
set $\chi=\chi_{k_0}^*+\tilde \chi$.    We can then write \eqref{ode-system-eq1} as
\begin{equation}
\label{ode-system-eq2}
\begin{cases}
\frac{du_0}{dt}=-au_0-bu_0^2-\frac{1}{2} b \sum_{i=1}^\infty u_i^2\cr
\frac{d u_{k_0}}{dt}=\frac{\tilde \chi\mu k_0^2\pi^2}{\mu L^2 +k_0^2 \pi^2}u_{k_0}  -  \frac{b}{2} \Big (\sum_{i-j=k_0} u_i u_j + \sum_{i-j=-k_0}  u_i u_j+\sum_{i+j=k_0} u_iu_j\Big)\cr
\qquad\quad +
\frac{(\chi_{k_0}^*+\tilde\chi)\mu}{\nu} g_{k_0}(u_0,u_1,u_2,\cdots)\cr
\frac{d u_k}{dt}=\Big(\lambda_k(\chi_{k_0}^*,a,\mu) +\frac{\tilde\chi\mu k^2\pi^2}{\mu L^2+\kappa^2\pi^2}\Big) u_k  -  \frac{b}{2} \Big (\sum_{i-j=k} u_i u_j + \sum_{i-j=-k}  u_i u_j+\sum_{i+j=k} u_iu_j\Big)\cr
\qquad\quad +\frac{( \chi_{k_0}^*+\tilde \chi )\mu}{\nu} g_k(u_0,u_1,u_2,\cdots),\quad k\not =k_0,  k=1,2,\cdots\cr
\frac{d\tilde \chi}{dt}=0.
\end{cases}
\end{equation}
By the assumption in (1),  $\lambda_k(\chi_{k_0}^*,a,\mu)\not =0$ for all $k\not =k_0$. Note that
$$
g_k(u_0,u_1,\cdots)=O(u_0^2+u_1^2+u_2^2+\cdots),
\quad k=1,2,\cdots.
$$
Then by center manifold theory, for $\tilde\chi=o(1)$ and ${u_{k_{0}}}=o(1)$, there are
$h_{k}(\tilde \chi,u_{k_0})$, {and constants $a_{k,1}, a_{k,2}, a_{k,3}$} ($k{{\not =}}k_0$) such that
\begin{equation}
\label{center-manifold-eq1}
h_k(\tilde \chi,u_{k_0})=a_{k,1} u_{k_0}^2 +a_{k,2} \tilde\chi u_{k_0} +a_{k,3} \tilde \chi^2 + O(\tilde\chi^3+u_{k_0}^3),\quad k\not = k_0,\,\, k=0,1,2\cdots,
\end{equation}
and
\begin{equation}
\label{center-manifold-eq2}
\mathcal{M}=\{\big((u_0,u_1,\cdots,u_{k_0},\cdots),\tilde\chi\big)|
\tilde \chi=o(1),\,\, u_{k_0}=o(1),\,\, u_k=h_k(\tilde\chi,u_{k_0}) \,\, {\rm for}\,\, k\not =k_0\}
\end{equation}
is locally invariant under \eqref{ode-system-eq2}.  $\mathcal{M}$ is referred to the center manifold at $0$ for \eqref{ode-system-eq2}.

 In the following, we find the reduced ODE on the center manifold and then study the bifurcation solutions of the reduced ODE.
To this end, first, differentiating  $u_0=a_{0,1}u_{k_0}^2 +a_{0,2} \tilde \chi u_{k_0}+a_{0,3}\tilde \chi^2+ O(\tilde\chi^3+u_{k_0}^3)$
with respect to $t$ and using  \eqref{ode-system-eq2}, we get  on $\mathcal{M}$,
\begin{align}
\label{ode-eq-u0-1}
\frac{d u_0}{dt}&=2 a_{0,1} u_{k_0} \frac{d u_{k_0}}{dt}+a_{0,2}\tilde \chi \frac{du_{k_0}}{dt}+O(\tilde\chi^3+u_{k_0}^3)\nonumber\\
&=\Big(2 a_{0,1}u_{k_0} +a_{0,2}\tilde\chi\Big)   \Big(\frac{\tilde \chi\mu { k_{0}^2}\pi^2}{\mu L^2 +{ k_{0}^2} \pi^2}u_{k_0} -  \frac{b}{2} \big (\sum_{i-j=k_0} u_i u_j + \sum_{i-j=-k_0}  u_i u_j+\sum_{i+j=k_0} u_iu_j\big)\nonumber\\
&\quad +
\frac{(\chi_{k_0}^*+\tilde\chi)\mu}{\nu}g_{k_0}(u_0,u_1,u_2,\cdots)\Big)\nonumber+O(\tilde\chi^3+u_{k_0}^3)\nonumber\\
&=O(\tilde \chi^3+u_{k_0}^3)
\end{align}
for $\tilde \chi=o(1)$ and $u_{k_0}=o(1)$.
On the other hand, we have
\begin{align}
\label{ode-eq-u0-2}
\frac{du_0}{dt}&=-a u_0-b u_0^2-\frac{1}{2}b\sum_{i=1}^\infty u_i^2\nonumber\\
&=-a \Big(a_{0,1} u_{k_0}^2+a_{0,2}\tilde\chi u_{k_0}+a_{0,3}\tilde \chi^2 \Big)-\frac{1}{2}b u_{k_0}^2+O(\tilde\chi^3+u_{k_0}^3)
\end{align}
for $\tilde \chi=o(1)$ and $u_{k_0}=o(1)$. By \eqref{ode-eq-u0-1} and \eqref{ode-eq-u0-2}, we must have
\begin{equation}
\label{coefficients-eq1}
a_{0,1}=-\frac{b}{2a},\,\, a_{0,2}=a_{0,3}=0.
\end{equation}

 Next, for $k\ge 1$, $k\not =k_0$, differentiating  $u_k=a_{k,1}u_{k_0}^2 +a_{k,2} \tilde \chi u_{k_0}+a_{k,3}\tilde \chi^2+ O(\tilde\chi^3+u_{k_0}^3)$
with respect to $t$ and using  \eqref{ode-system-eq2}, we get  on $\mathcal{M}$,
\begin{align}
\label{ode-eq-uk-1}
\frac{du_k}{dt}&=(2a_{k,1} u_{k_0}+a_{k,2}\tilde \chi)   \Big(\frac{\tilde \chi\mu k_0^2\pi^2}{\mu L^2 +k_0^2 \pi^2}u_{k_0} -  \frac{b}{2} \big (\sum_{i-j=k_0} u_i u_j + \sum_{i-j=-k_0}  u_i u_j+\sum_{i+j=k_0} u_iu_j\big)\nonumber\\
&\qquad +
\frac{(\chi_{k_0}^*+\tilde\chi)\mu}{\nu}g_{k_0}(u_0,u_1,u_2,\cdots)\Big) +O(\tilde\chi^3+u_{k_0}^3)\nonumber\\
&=O(\tilde\chi^3+u_{k_0}^3)
\end{align}
for $\tilde \chi=o(1)$ and $u_{k_0}=o(1)$.
On the other hand, we have
\begin{align}
\label{ode-eq-uk-2}
\frac{d u_k}{dt}&=\Big(\lambda_k(\chi_{k_0}^*,a,\mu)
+\frac{\tilde\chi\mu k^2\pi^2}{\mu L^2+\kappa^2\pi^2}\Big) u_k   -  \frac{b}{2} \Big (\sum_{i-j=k} u_i u_j + \sum_{i-j=-k}  u_i u_j+\sum_{i+j=k} u_iu_j\Big)\nonumber\\
&\quad +\frac{( \chi_{k_0}^*+\tilde \chi )\mu}{\nu} g_k(u_0,u_1,u_2,\cdots)\nonumber\\
&=\lambda_k(\chi_{k_0}^*,a,\mu)  (a_{k,1} u_{k_0}^2+a_{k,2}\tilde\chi u_{k_0} +a_{k,3}\tilde\chi^2) -  \frac{b}{2} \Big (\sum_{i-j=k} u_i u_j + \sum_{i-j=-k}  u_i u_j+\sum_{i+j=k} u_iu_j\Big)\nonumber\\
&\quad +\frac{\chi_{k_0}^* \mu}{\nu} g_k(u_0,u_1,\cdots)+O(\tilde\chi^3+u_{k_0}^3)
\end{align}
for $\tilde \chi=o(1)$ and $u_{k_0}=o(1)$.
By \eqref{gk-eq1}, when $k=2k_0$, we have
\begin{align}
\label{ode-eq-uk-3}
 & -  \frac{b}{2} \Big (\sum_{i-j=2k_0} u_i u_j + \sum_{i-j=-2k_0}  u_i u_j+\sum_{i+j=2k_0} u_iu_j\Big) +\frac{\chi_{k_0}^* \mu}{\nu} g_{2k_0}(u_0,u_1,\cdots)\nonumber\\
&=-\frac{b}{2} u_{k_0}^2+\frac{\chi_{k_0}^*\mu}{\nu} \Big[\frac{b}{a} \frac{k_0^2 \pi^2}{L^2}  u_{k_0} v_{k_0} -\frac{\mu}{\nu}\frac{b}{a} \frac{k_0^2 \pi^2}{L^2}   v_{k_0} v_{k_0}\Big]+O(\tilde\chi^3+u_{k_0}^3)\nonumber\\
&=\Big(-\frac{b}{2}+\chi_{k_0}^*\mu \frac{b}{a}\frac{k_0^4\pi^4}{L^4}\frac{1}{\big(\mu+\frac{k_0^2\pi^2}{L^2}\big)^2}\Big)u_{k_0}^2+O(\tilde\chi^3+u_{k_0}^3).
\end{align}
When $k\not =2k_0$, we have
\begin{align}
\label{ode-eq-uk-4}
&  -  \frac{b}{2} \Big (\sum_{i-j=k} u_i u_j + \sum_{i-j=-k}  u_i u_j+\sum_{i+j=k} u_iu_j\Big) +\frac{\chi_{k_0}^* \mu}{\nu} g_k(u_0,u_1,\cdots)\nonumber\\
&=O(\tilde\chi^3+u_{k_0}^3).
\end{align}
By  \eqref{ode-eq-uk-1}-\eqref{ode-eq-uk-4}, we obtain that
\begin{equation}
\label{coefficients-eq2}
a_{2k_0,1}=\frac{1}{\lambda_{2k_0}(\chi_{k_0}^*,a,\mu)}\Big(\frac{b}{2}-\frac{\chi_{k_0}^* \mu b  k_0^4 \pi^4}{a(\mu L^2+k_0^2\pi^2)^2}\Big),\,\,  a_{2k_0,2}=a_{2k_0,3}=0,
\end{equation}
and
\begin{equation}
\label{coefficients-eq3}
a_{k,1}=a_{k,2}=a_{k,3}=0,\quad k\not = k_0,2k_0,\,\, k=1,2,\cdots.
\end{equation}

Now, by \eqref{coefficients-eq1}-\eqref{coefficients-eq3}, we have  on $\mathcal{M}$,
\begin{align*}
&g_{k_0}(u_0,u_1,u_2,\cdots)\nonumber\\
&=\frac{b}{a}\frac{k_0^2\pi^2}{L^2}\frac{1}{2}\big[ 2 u_0 v_{k_0}+2 u_{k_0} v_{2k_0}-u_{2k_0}v_{k_0}\big] -\frac{\mu}{\nu}\frac{b}{a}\frac{k_0^2\pi^2}{L^2}\frac{1}{2}\big[2 v_0 v_{k_0}+v_{k_0}v_{2k_0}\big]\nonumber\\
&\quad -\frac{\mu}{\nu}\frac{b^2}{a^2}\frac{k_0^2\pi^2}{L^2}\frac{1}{4} u_{k_0} v_{k_0}^2 +\frac{\mu^2}{\nu^2}\frac{b^2}{a^2}\frac{k_0^2\pi^2}{L^2}\frac{1}{4} v_{k_0}^3+
O(u_{k_0}^4)\nonumber\\
&=\frac{b k_0^2 \pi^2}{2a}\Big(\frac{2\nu}{\mu L^2+k_0^2\pi^2} u_0 u_{k_0}+\frac{\nu \mu L^2-2{ \nu}k_0^2\pi^2}{(\mu L^2+k_0^2\pi^2)(\mu L^2+4 k_0^2\pi^2)}u_{k_0}u_{2k_0}\Big)\nonumber\\
&\quad -\frac{b k_0^2\pi^2}{2a}\Big(\frac{2\nu}{\mu L^2+k_0^2\pi^2}u_0 u_{k_0}+
\frac{\mu \nu L^2}{(\mu L^2+k_0^2\pi^2)(\mu L^2+4 k_0^2\pi^2)} u_{k_0}u_{2k_0}\Big)\nonumber\\
&\quad -\frac{\mu b^2 k_0^2 \pi^2}{4 a^2}\frac{\nu L^2}{(\mu L^2+k_0^2\pi^2)^2}u_{k_0}^3
+\frac{\mu^2b^2 k_0^2\pi^2}{4 a^2}\frac{\nu L^4}{(\mu L^2+k_0^2\pi^2)^3} u_{k_0}^3+
O(\tilde \chi^4+u_{k_0}^4)\nonumber\\
&=-\frac{b k_0^4\pi^4}{a (\mu L^2+k_0^2\pi^2)(\mu L^2+4k_0^2\pi^2)} u_{k_0}u_{2k_0}
-\frac{\mu b^2 k_0^4 \pi^4 \nu L^2}{4 a^2 (\mu L^2+k_0^2\pi^2)^3} u_{k_0}^3+O(\tilde\chi^4+u_{k_0}^4)\nonumber\\
&=\Big[-\frac{b k_0^4\pi^4}{a (\mu L^2+k_0^2\pi^2)(\mu L^2+4k_0^2\pi^2)} a_{2k_0,1}
-\frac{\mu b^2 k_0^4 \pi^4 \nu L^2}{4 a^2 (\mu L^2+k_0^2\pi^2)^3}\Big] u_{k_0}^3+O(\tilde \chi^4+u_{k_0}^4)
\end{align*}
Therefore, on the center manifold $\mathcal{M}$, the dynamics is determined by the following ODE,
\begin{align}
\label{ode-on-center-manifold-1}
\frac{d u_{k_0}}{dt}&= \frac{\tilde \chi\mu k_0^2\pi^2}{\mu L^2 +k_0^2 \pi^2}u_{k_0}  -  \frac{b}{2} \Big (\sum_{i-j=k_0} u_i u_j + \sum_{i-j=-k_0}  u_i u_j+\sum_{i+j=k_0} u_iu_j\Big)\nonumber\\
&\quad
\qquad\quad +
\frac{(\chi_{k_0}^*+\tilde\chi)\mu}{\nu} g_{k_0}(u_0,u_1,u_2,\cdots)\nonumber\\
&= \frac{\tilde \chi\mu k_0^2\pi^2}{\mu L^2 +k_0^2 \pi^2}u_{k_0}  -b\Big( 2a_{0,1}+a_{2k_0,1}+\frac{\chi_{k_0}^* \mu k_0^4\pi^4}{\nu a (\mu L^2+k_0^2\pi^2)(\mu L^2+4k_0^2\pi^2)} a_{2k_0,1}\nonumber\\
&\qquad
+\frac{\chi_{k_0}^* \mu^2 b k_0^4 \pi^4 L^2}{4 a^2 (\mu L^2+k_0^2\pi^2)^3}\Big) u_{k_0}^3 +O(\tilde\chi^4+u_{k_0}^4).
\end{align}
By \eqref{a-0-k-eq}, \eqref{coefficients-eq1}, and \eqref{coefficients-eq2}, equation \eqref{ode-on-center-manifold-1}
can be written as
\begin{equation}
\label{ode-on-center-manifold-2}
\frac{du_{k_0}}{dt}=\alpha_{k_0}\tilde \chi u_{k_0}-\beta_{k_0} u_{k_0}^3 +O(\tilde\chi^4+u_{k_0}^4),
\end{equation}
where
$$
\alpha_{k_0}=\frac{\mu k_0^2\pi^2}{\mu L^2 +k_0^2 \pi^2}.
$$
Pitchfork bifurcation then occurs in \eqref{ode-on-center-manifold-2} near $0$
provided that $\beta_{k_0}\not =0$. This proves (1).

Moreover, by the above arguments,  it is clear that \eqref{two-branches-eq},
\eqref{upper-branch-eq}, and \eqref{lower-branch-eq} hold.

\smallskip

(2)  Let $k^*\ge 1$ be such that
$\chi_{k^*}^*=\chi^*$.  If $\chi_{k_0}^*>\chi^*$, then  $\lambda_{k^*}(\chi_{k_0}^*,a,\mu)>0$
and the positive constant solution $(\frac{a}{b},\frac{\nu}{\mu}\frac{a}{b})$ is linearly unstable. By the general  perturbation theory (see \cite[Theorem 5.1.3]{Hen}), any solution {with the $u$-component belonging to $\Gamma_{k_0}(\chi)\setminus\frac{a}{b}$} is linearly unstable for
$\chi\approx\chi_{k_0}^*$.

\smallskip

(3) If $\chi_{k_0}^*=\chi^*$ and $\beta_{k_0}>0$,  then \eqref{ode-on-center-manifold-2}
experiences super-pitchfork bifurcation near $0$ when $\tilde\chi$ passes through $0$.
This implies that  any solution {with the $u$-component belonging to $\Gamma_{k_0}(\chi)\setminus\frac{a}{b}$}  is linearly  stable for
$\chi\approx\chi_{k_0}^*$.

If $\chi_{k_0}^*=\chi^*$ and $\beta_{k_0}<0$,  then \eqref{ode-on-center-manifold-2}
experiences  sub-pitchfork bifurcation near $0$ when $\tilde\chi$ passes through $0$.
This implies that  any solution {with the $u$-component belonging to $\Gamma_{k_0}(\chi)\setminus\frac{a}{b}$}  is linearly  unstable for
$\chi\approx\chi_{k_0}^*$.
The theorem is thus proved.
\end{proof}

\section{Global bifurcation and  properties of non-constant stationary solutions}

In this section, we study the global extension of the local bifurcation branches $\Gamma_{k}(\chi)$ for $\chi\approx\chi_k^*$ and properties of non-constant stationary  solutions.

We first introduce some notations. Let
$$
X=\{u\in H^2((0,L))\,|\, u_x(0)=u_x(L)=0, u(x)>0\,\, {\rm for}\,\, x\in [0,L]\},\quad Y=L^2((0,L))
$$
Let $F:(0,\infty)\times X\to Y$ be defined by
$$
F(\chi,u)=u_{xx}-\chi (\frac{u}{v} v_x)_x+u(a-b u),
$$
where $v$ is the solution of \eqref{new-new-v-eq}.
Let
 $$
S_k=\{(\chi,u)\,|\, 0< \chi<\chi_{k}^*+1, \,\, u\in X, \,\, u\not =\frac{a}{b},\,\, F(\chi,u)=0\}.
$$

For given $k\ge 1$,  if $\lambda_{l}(\chi_k^*,a,\mu)\not =0$ for all $l\not =k$, then there is
 $\epsilon_k>0$  such that $\chi_k^*-\epsilon_k>0$ and
$\lambda_l(\chi,a,\mu)\not =0$ for $\chi\in (\chi_k^*-\epsilon_k,\chi_k^*+\epsilon_k)$. For such $k$, let
$$
S_{0,k}=\{(\chi,\frac{a}{b})\,|\, \chi_k^*-\epsilon_k<\chi<\chi_k^*+\epsilon_k\},
$$
 and $\mathcal{C}_k$ be the connected component of
$\bar S_k\cup S_{0,k}$ containing $S_{0,k}$.
In addition, if  $\beta_{k}\not =0$, then by Theorem \ref{local-bif-thm}, pitchfork-type bifurcation occurs in \eqref{main-eq0} near $(\frac{a}{b},\frac{\nu}{\mu}\frac{a}{b})$ when $\chi$ passes through $\chi_k^*$. In such case,
let
$$
\Gamma_{k}^\pm=\begin{cases}
\{(\chi,\Gamma_{k}^\pm(\chi)\,|\, 0<\chi-\chi_k^*\ll 1\}\quad {\rm if}\; \beta_k>0\cr\cr
\{(\chi,\Gamma_{k}^\pm(\chi)\,|\, 0<\chi_k^*-\chi\ll 1\}\quad {\rm if}\; \beta_k<0.
\end{cases}
$$

Next, we state the main results of this section.   The first theorem is on the global bifurcation.

\begin{theorem} [Global bifurcation]
\label{global-bif-thm}
For given $k\ge 1$, assume that  $\beta_{k}\not =0$ and $\lambda_l(\chi_{k}^*,a,\mu)\not =0$ for any $l\not =k$.
Let $\mathcal{C}_k^\pm $ be the connected component of
$\mathcal{C}_k\setminus \Gamma_k^\mp$ which contains $\Gamma_k^\pm$. Then each of the sets $\mathcal{C}_k^+$ and $\mathcal{C}_k^-$ satisfies one of the following:
(i) it is  not compact; (ii) it contains a point $(\chi,\frac{a}{b})$ with  $\chi\not =\chi_k^*$;
or (iii) it contains a point $(\chi,\frac{a}{b}+w)$, where $w\not =0$ and $w\in W$, where $W$ complements ${\rm span}\{\cos\frac{k\pi x}{L}\}$. Moreover, if $\mathcal{C}_k^+$ (resp.
$\mathcal{C}_k^-$) is not compact, then $\mathcal{C}_k^+$ (resp. $\mathcal{C}_k^-$) extends to infinity in the positive direction of $\chi$.
\end{theorem}

\begin{remark}
\label{bifurcation-rk1}
\begin{itemize}

\item[(1)]
For any {$u \in \Gamma_1^+(\chi)$}  (resp. $\Gamma_1^-(\chi)$) with $\chi\approx \chi_1^*$, $u_x(x)<0$  (resp. $u_x(x)>0$) for $x\in (0,L)$, and $u(0)>\frac{a}{b}$
(resp. $u(L)>\frac{a}{b}$).

\item[(2)]
For given $k\ge 1$,  if $k={k^*}$ and  super-critical pitchfork bifurcation  occurs when $\chi$ passes through $\chi^*$, it is observed numerically that  the global bifurcation diagram is of the form in Figure \ref{bifurcation}(a) and the bifurcation solutions are locally stable (see numerical simulations in subsection {\ref{l-1-sec}}), where the red part consists of {the u-components of} local bifurcation solutions and the blue parts are the global extension of the {u-components of} local bifurcation solutions.

\item[(3)] For given $k\ge 1$, if $k={k^*}$ and   sub-critical pitchfork occurs when $\chi$ passes through $\chi^*$, it is observed numerically that  the global bifurcation diagram is of the form in Figure \ref{bifurcation}(b) and the bifurcation solutions {with u-components} lying in the red and green parts are unstable and the bifurcation solutions {with u-components} lying  in the blue parts  are locally stable (see numerical simulations in subsections {\ref{l-6-sec}} and {\ref{a-2-l-6-sec}}), where the red part consists of  {the u-components of} local bifurcation solutions and the green and blue parts are the global extension of the {u-components of} local bifurcation solutions.
\end{itemize}

\begin{figure}[!ht]
\begin{center}
\subfigure[]{
\resizebox*{0.36\linewidth}{!}{\includegraphics{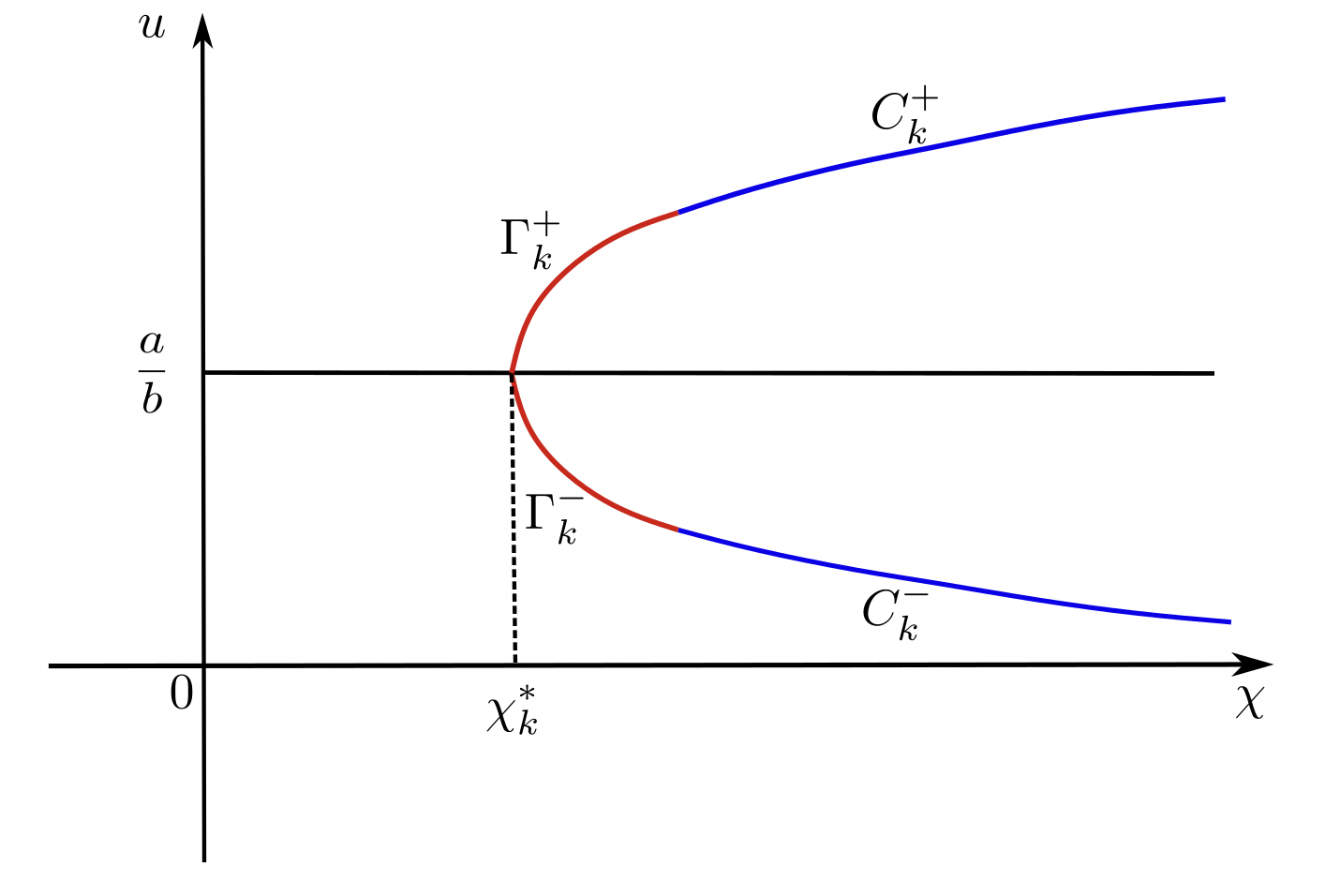}}
}
\subfigure[]{\resizebox*{0.36\linewidth}{!}{\includegraphics{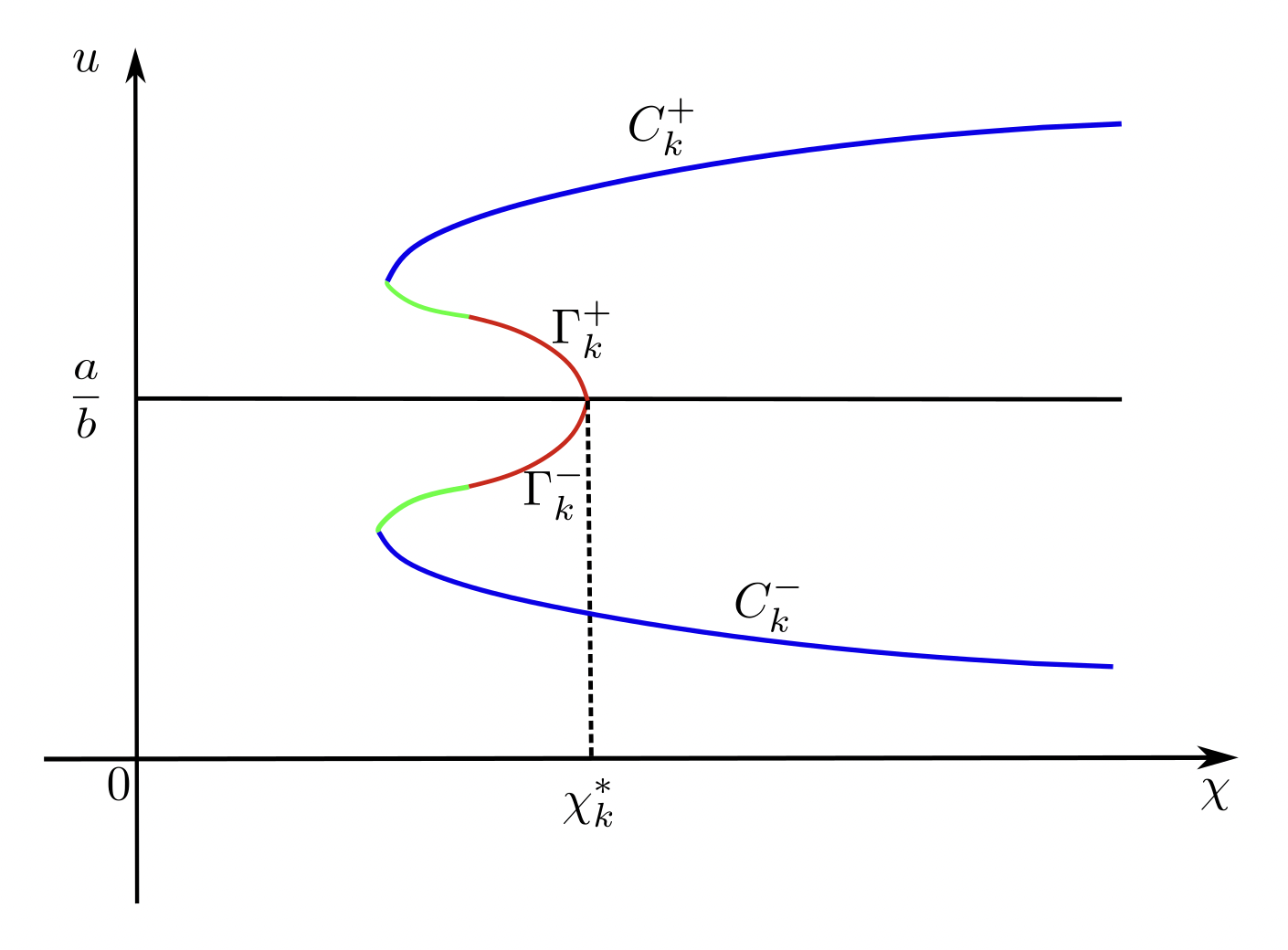}}}
\caption{(a) super-critical pitchfork bifurcation (b)  sub-critical pitchfork bifurcation 
}
\label{bifurcation}
\end{center}
\end{figure}

\end{remark}

The next two  theorems are on the properties of non-constant stationary solutions.

\begin{theorem} [Properties of non-constant stationary solutions]
\label{property-bif-solu-thm}
Let $(u^*(x),v^*(x))$ be a non-constant positive  stationary solution of \eqref{main-eq0}.
Then the following hold.
\begin{itemize}

\item[(1)]
\begin{equation}
\label{bound-v-eq2}
\Big|\frac{v_x^*(x)}{v^*(x)}\Big|\le \sqrt \mu \quad \forall\, x\in [0,L].
\end{equation}

\item[(2)]  There is a positive constant $\delta^*>0$ independent of $u^*$ such that
\begin{equation}
\label{lower-bound-v-eq}
v^*(x)\ge \delta^* \int_0^L u^*(x)dx \quad \forall\, x\in [0,L].
\end{equation}

\item[(3)]
\begin{equation}
\label{bound-u-eq1}
0<\inf_{x\in [0,L]} u^*(x)< \frac{a}{b}< \sup_{x\in [0,L]}u^*(x),
\end{equation}
and
\begin{equation}
\label{bound-u-eq2}
\int_0^L u^*(x)dx< \frac{aL}{b},\quad \int_0^L (u^*)^2(x)dx< \frac{a^2 L}{b^2}.
\end{equation}

\item[(4)]  There is $C_1(\chi,a,b,\mu)>0$ such that
\begin{equation}
\label{bound-u-eq3}
|u^*(x)|, \,\, |u_x^*(x)|\le C_1(\chi, a,b,\mu)\quad \forall\, x\in [0,L].
\end{equation}

\item[(5)] There is $C_2(a,b,\mu,\nu)>0$ such that
\begin{equation}
\label{bound-v-eq1}
\|v^*\|_{H^2(0,L)}\le C_2(a,b,\mu,\nu).
\end{equation}

\item[(6)]
If $u^*(x)$ is monotone increasing (respectively, decreasing) on $(0,L)$, then
  $v^*(x)$ is strictly increasing (respectively, decreasing) on $(0,L)$.
\end{itemize}
\end{theorem}

\begin{remark}
\label{property-bif-solu-rk}
By Theorem \ref{property-bif-solu-thm}(3), the set of $u$-components of  non-constant positive stationary solutions is bounded in $L^2(0,L)$ as $\chi\to\infty$. By Theorem \ref{property-bif-solu-thm}(5),  the set of $v$-components of non-constant positive stationary solutions is bounded in $H^2(0,L)$  as $\chi\to\infty$ and hence is bounded in $C^{1,\gamma}([0,L])$ for some ${\gamma}>0$. It remains open whether the set of $u$-components of  non-constant positive stationary solutions is bounded in $C([0,L])$ as $\chi\to\infty$.
\end{remark}

\begin{theorem} [Properties of non-constant stationary solutions]
\label{property-bif-solu-thm1}
Let $\{(u(\cdot;\chi_n),v(\cdot;\chi_n))\}_{n\ge 1}$ be a sequence of non-constant positive stationary solutions of \eqref{main-eq0} with $\chi=\chi_n$ and $\chi_n\to\infty$.
\begin{itemize}
\item[(1)] If $v^*\in H^2(0,L)$ is such that
$$
\lim_{n\to\infty} v(x;\chi_n)= v^*(x)\quad \text{in}\,\, C^1([0,L])
\quad {\rm and}\,\,\,
\lim_{n\to\infty} v(x;\chi_n)=v^*(x)\quad \text{weakly in}\,\, H^2(0,L),
$$
then there is a constant $0\le \bar v\le \frac{\nu}{\mu}\frac{a}{b}$ such that
$$
v^*(x)=\bar v\quad a.e.\,\, x\in [0,L].
$$

\item[(2)] If $u^*\in L^2(0,L)$ is such that
$$
\lim_{n\to\infty} u(x;\chi_n)=u^*(x)\quad \text{weakly in}\,\, L^2(0,L)
\quad {\rm or}\,\,\,
\lim_{n\to\infty}u(x;\chi_n)=u^*(x)\quad a.e.\,\, x\in [0,L],
$$
then there is a constant $0\le \bar u\le \frac{a}{b}$ such that
$$
u^*(x)=\bar u\quad a.e.\,\, x\in [0,L].
$$
\end{itemize}
\end{theorem}

\begin{remark}
\label{property-bif-solu-rk1}
Let $\{(u(\cdot;\chi_n),v(\cdot;\chi_n))\}_{n\ge 1}$ be a sequence of non-constant positive stationary solutions of \eqref{main-eq0} with $\chi=\chi_n$ and $\chi_n\to\infty$.

\begin{itemize}
\item[(1)] By Theorem \ref{property-bif-solu-thm}(3) and (5), there are $n_k\to\infty$,
$u^*\in L^2(0,L)$, and $v^*\in H^2(0,L)$ such that
$$
\lim_{n_k\to\infty} u(x;\chi_{n_k})=u^*(x)\quad \text{weakly in}\,\, L^2(0,L),
$$
and
$$
\lim_{n_k\to\infty} v(x;\chi_{n_k})=v^*(x)\quad \text{in}\,\, C^1([0,L]),
\quad
\lim_{n_k\to\infty}v(x;\chi_{n_k})=v^*(x)\quad \text{weakly in}\,\, H^2(0,L).
$$

\item[(2)]
If there are $u^*\in L^2(0,L)$ and $x^*\in [0,L]$ such that
$$
\lim_{n\to\infty} u(x;\chi_n)=u^*(x)\quad  a.e. \,\, x\in (0,L) ,
\quad {\rm and}\,\,\,
\liminf_{n\to\infty}u(x^*;\chi_n)>\frac{a}{b},
$$
then  by Theorem \ref{property-bif-solu-thm1} (2), there is $\bar u\in [0,\frac{a}{b}]$ such that
$$
\lim_{n\to\infty} u(x;\chi_n)=\bar u \le \frac{a}{b}\quad a.e.\,\,  x\in [0,x^*)\cup (x^*,L],
$$
which shows in certain sense that $\{(u(\cdot;\chi_n),v(\cdot;\chi_n))\}_{n\ge 1}$ develops ``spikes''  at $x^*$ as $n\to\infty$. We will give a definition of spikes and study spiky solutions in next section.

\end{itemize}
\end{remark}

In the following, we prove  Theorems \ref{global-bif-thm}-\ref{property-bif-solu-thm1}.
{We first prove Theorem \ref{property-bif-solu-thm} since we need to employ Theorem \ref{property-bif-solu-thm} in the proof of Theorem \ref{global-bif-thm}. Next, we prove Theorem \ref{property-bif-solu-thm1}. We end this section by proving Theorem \ref{global-bif-thm}.}

\medskip

We first prove Theorem \ref{property-bif-solu-thm}.

\begin{proof} [Proof of Theorem \ref{property-bif-solu-thm}]

(1) We first  extend $u^*(x)$ to the following $\tilde u^*(x)$,
\begin{equation*}
\tilde u^*(x)=\begin{cases} u^*(x),\quad & 0\le x\le L\cr
u^*(2kL-x),\quad &(2k-1)L\le x\le 2kL,\,\, k=1,2,\cdots\cr
u^*(x-2kL),\quad &2kL\le x\le (2k+1)L,\,\, k=1,2,\cdots
\end{cases}
\end{equation*}
and
\begin{equation*}
\tilde u^*(x)=\tilde u^*(-x)\quad \forall\,\, x\le 0.
\end{equation*}
It is clear that
\begin{equation}
\label{u-extension-eq3}
\tilde u^*(x)=\tilde u^*(-x),\quad \tilde u^*(x+L)=\tilde u^*(-x+L),\quad  \tilde u^*(x+2L)=\tilde u^*(x).
\end{equation}
This implies that
\begin{equation}
\label{u-extension-eq}
\tilde u^*_x(kL)=0\quad {\rm for}\,\, k=0,\pm1,\pm 2,\cdots.
\end{equation}

Next, let
$\tilde v^*(x)$ be the solution of
\begin{equation}
\label{v-extension-eq1}
v_{xx}-\mu v+\tilde u^*(x)=0,\quad x\in\mathbb{R}.
\end{equation}
Then we also have
\begin{equation}
\label{v-extension-eq2}
\tilde v^*(x)=\tilde v^*(-x),\quad \tilde v^*(x+L)=\tilde v^*(-x+L),\quad  \tilde v^*(x+2L)=\tilde v^*(x)
\end{equation}
and
\begin{equation}
\label{v-extension-eq3}
\tilde v^*_x(kL)=0\quad {\rm for}\,\, k=0,\pm1,\pm 2,\cdots.
\end{equation}
This implies that
\begin{equation}
\label{v-extension-eq4}
{v^*(x)}=\tilde v^*(x)\quad {\rm for}\quad 0\le x\le L.
\end{equation}

Now,
by {\cite[Lemma 2.2]{SaShXu},}
$$
|\frac{\tilde v_x^*(x)}{\tilde v^*(x)}|\le \sqrt \mu\quad \forall\, x\in\mathbb{R}.
$$
This together with \eqref{v-extension-eq4} implies that \eqref{bound-v-eq2} holds.

(2) It follows from \cite[Lemma 2.1]{FuWiYo}.

(3)
First, observe that
\begin{equation}
\label{stationary-u-eq1}
\begin{cases}
u^*_{xx}-\chi(\frac{u^*}{v^*} v^*_x)_x+u^*(a-bu^*)=0,\quad 0<x<L\cr
u^*_x(0)=u^*_x(L)=0
\end{cases}
\end{equation}
and $v_x^*(0)=v_x^*(L)=0$. It then follows that
$$
0=a\int_0^L u^*(x)dx-b\int_0^L  (u^*)^2(x)dx\le a\int_0^L u^*(x)dx-\frac{b}{L}\big(\int_0^L u^*(x)dx\big)^2.
$$
This implies  that
$$
\inf_{x\in [0,L]}(a-b u^*(x))<0< \sup_{x\in[0,L]}(a-b u^*(x)) \quad {\rm (since\,\,}
u^*(x)\not\equiv \,{\rm constant)},
$$
$$
\int_0^L u^*(x)dx\le \frac{a L}{b},
$$
and
$$
\int_0^L (u^*)^2(x)dx=\frac{a}{b}\int_0^L u^*(x)dx.
$$
\eqref{bound-u-eq1} and \eqref{bound-u-eq2} then follow.

(4)
Integrating the first equation in \eqref{stationary-u-eq1} from $0$ to $x$ and noting that $u^*_x(0)=v^*_x(0)=0$, we get
\begin{equation}
\label{u-x-eq}
u^*_x(x)=\chi \frac{v^*_x(x)}{v^*(x)} u^*(x)+ f^*(x)\quad \forall\, x\in [0,L],
\end{equation}
where
$$
f^*(x)=-\int_0^ x  u^*(z)(a-b u^*(z)) dz\quad {\rm for}\quad x\in [0, L].
$$
By \eqref{bound-u-eq1}, there is $x_0\in [0,L]$ such that $u^*(x_0)=\frac{a}{b}$.
We then have
\begin{equation}
\label{u-eq}
u^*(x)=e^{\chi\int_{x_0}^x \frac{v_x^*(z)}{v^*(z)}dz} \frac{a}{b}+
\int_{x_0} ^ x  e^{\chi\int_{z}^x \frac{v_x^*(y)}{v^*(y)}dy } f^*(z) dz\quad \forall\, x\in [0,L].
\end{equation}
By   \eqref{bound-v-eq2}, \eqref{bound-u-eq2}, \eqref{u-x-eq}, and \eqref{u-eq},
there is $C(\chi,\mu,a,b)>0$ such that  \eqref{bound-u-eq3} holds.

(5) Observe that
$$
\begin{cases}
v^*_{xx}-\mu v^*(x)+\nu u^*(x)=0,\quad x\in (0,L)\cr
v^*_x(0)=v^*_x(L)=0.
\end{cases}
$$
By \eqref{bound-u-eq1}, \eqref{bound-u-eq2}, and a priori estimates for elliptic equations, there is
$C_2(a,b,\mu,\nu)>0$ such that  \eqref{bound-v-eq1} {holds}.

(6)
Without loss of generality, we assume that $u^*(x)$ is monotone increasing on $(0,L)$.
Then $u^*_x(x)\ge 0$ on $(0,L)$.
Let $w^*=v^*_x$. Then $w^*(x)$ is the solution of
\begin{equation}
\label{w-star-eq1}
\begin{cases}
w^*_{xx}-\mu w^*(x)+\nu u^*_x(x)=0,\quad 0<x<L\cr
w^*(0)=w^*(L)=0.
\end{cases}
\end{equation}
Since $u^*(x)\not\equiv$ constant, we have $v^*(x)\not \equiv$ constant and $w^*(x)\not \equiv 0$.
By  maximum principle for elliptic equations, we have
$$
w^*(x)>0 ,\quad 0<x<L.
$$
This implies that $v^*_x(x)>0$ for every $x\in (0,L)$.
\end{proof}

Next, we prove  Theorem \ref{property-bif-solu-thm1}.

\begin{proof} [Proof of Theorem \ref{property-bif-solu-thm1}]
For simplicity in notation, put
$$
(u_n(x),v_n(x))=(u(x;\chi_n),v(x;\chi_n)).
$$

(1) First,
By Theorem \ref{property-bif-solu-thm},
$$
\limsup_{n\to\infty}\int_0^L u_n^2(x)dx<\infty.
$$
Without loss of generality, we may then assume that there is $u^*\in L^2(0,L)$ such that
$$
\lim_{n\to\infty} u_n(x)=u^*(x)\quad \text{weakly in }\,\, L^2(0,L).
$$
Note that
\begin{equation}
\label{case2-v-eq1}
v_{n xx}-\mu v_{n}+\nu u_{n}=0\quad \text{on}\,\, (0,L).
\end{equation}
We then have   $v_{n xx}$ converges {to} $v^*_{xx}=\mu v^*-\nu u^*$ weakly in $L^2(0,L)$,
and $v^*$ is a weak solution of
\begin{equation}
\label{case2-v-eq2}
\begin{cases}
v^*_{xx}-\mu v^*+\nu u^*=0,\quad 0<x<L\cr
v^*_x(0)=v^*_x(L)=0.
\end{cases}
\end{equation}
If $u^*(x)=0$ for a.e. $x\in [0,L]$, we have $v^*(x)=0$ for all $x\in [0,L]$.
(1)  then follows.

In the following, we assume that $u^*(x)\not =0$ for a.e. $x\in [0,L]$. Then
$$
\int_0^L u^*(x)dx>0.
$$
This implies that
$$
\lim_{n\to\infty}\int_0^L u_n(x)dx>0.
$$
This together with Theorem \ref{property-bif-solu-thm}(2) implies that there is $\sigma>0$ such that
\begin{equation}
\label{spike-lm-eq3}
v_n(x)\ge \sigma\quad \forall\, n\gg 1,\,\, x\in [0,L].
\end{equation}
Note   that
\begin{equation}
\label{case2-eq4}
 \frac{v_{nx}(x)}{v_n(x)} u_n(x)=\frac{1}{\chi_n}\Big(u_{nx}(x)- f_n(x)\Big)\quad \forall\, x\in (0,L),
\end{equation}
where
$$
f_n(x)={\int_{x}^{L} } u_n(z)(a-b u_n(z)) dz\quad {\rm for}\quad x\in (0, L).
$$
By Theorem \ref{property-bif-solu-thm},
$\{f_n\}$ is a bounded sequence on $C([0,L])$.
By \eqref{spike-lm-eq3},
$$
v^*(x)\ge {\sigma}\quad \forall\, x\in [0,L].
$$
This together with     \eqref{case2-eq4} implies that,  for any {$x_1,x_2\in [0,L]$},
\begin{align}
\label{spike-lm-eq4}
\int_{x_1} ^{x_2}\frac{v_x^*(x)}{v^*(x)}u^*(x)dx&=\lim_{n\to\infty}
\int_{x_1} ^{x_2} \frac{v_{nx}}{v_{n}}u_{n}dx\nonumber\\
&=\lim_{n\to\infty}\big(\frac{1}{\chi_{n}}\int_{x_1} ^{x_2} u_{nx}dx-\frac{1}{\chi_{n}}\int_{x_1} ^{x_2} f_{n}(x)dx\big)\nonumber\\
&=\lim_{n\to\infty}\frac{1}{\chi_{n}}\big(u_n(x_2)-u_n(x_1)\big).
\end{align}

By Fatou's Lemma, we have
$$
0\le\int_0^L \liminf_{n\to\infty}u_n(x) dx\le \liminf_{n\to\infty}\int_0^L u_n(x)dx\le \frac{aL}{b}.
$$
Hence  for a.e. $x\in [0,L]$, $\liminf_{n\to\infty}u_n(x)<\infty$.  Therefore, there is $c\in [0,L]$
such that
$$
\liminf_{n\to\infty}u_n(c)<\infty.
$$
Then there is $n_k\to\infty$ such that
$\lim_{n_k\to\infty} u_{n_k}(c)$ exists and
$$
\lim_{n_k\to\infty} u_{n_k}(c)={\liminf_{n\to\infty}}  u_n(c)<\infty.
$$
By Fatou's Lemma again,
$$
0\le \int_0^L \liminf_{n_k\to\infty} u_{n_k}(x)dx\le \liminf_{n_k\to\infty}\int_0^L u_{n_k}(x)dx\le \frac{aL}{b}.
$$
Then for a.e. $x\in [0,L]$, $\liminf_{n_k\to\infty} u_{n_k}(x)dx<\infty$.
Let
$$
E=\{x\in [0,L]\,|\, \liminf_{n_k\to\infty} u_{n_k} (x)<\infty\}.
$$
 For any $d\in E$, there is $\{n_k^{'}\}\subset \{n_k\}$ such that
$$
\lim_{n_k^{'}\to\infty} u_{n_k^{'}}(d)=\liminf_{n_k\to\infty}u_{n_k}(d)<\infty.
$$
Then by \eqref{spike-lm-eq4},
we have
$$
\int_{c}^{d}\frac{v_x^*(x)}{v^*(x)}u^*(x)dx=\lim_{n_k^{'}\to\infty}\frac{1}{\chi_{n_k^{'}}}\big(u_{n_k^{'}}({d})-u_{n_k^{'}}({c})\big)=0.
$$
This implies that for any $d_1,d_2\in E$,
$$
\int_{d_1}^{d_2} \frac{v_x^*(x)}{v^*(x)}u^*(x)dx=-\int_c^{d_1} \frac{v_x^*(x)}{v^*(x)}u^*(x)dx+\int_{c}^{d_2} \frac{v_x^*(x)}{v^*(x)}u^*(x)dx=0.
$$
For any $x_1,x_2\in [0,L]$ and $\epsilon>0$, there are $d_1,d_2\in E$ such that
$$
|\int_{x_1}^{x_2} \frac{v_x^*(x)}{v^*(x)}u^*(x)dx-\int_{d_1}^{d_2} \frac{v_x^*(x)}{v^*(x)}u^*(x)dx|=| \int_{x_1}^{x_2} \frac{v_x^*(x)}{v^*(x)}u^*(x)dx|<\epsilon.
$$
This implies that
$$
\int_{x_1}^{x_2} \frac{v_x^*(x)}{v^*(x)}u^*(x)dx=0\quad \forall\, x_1,x_2\in [0,L]
$$
and then
$$
v_x^*(x)u^*(x)=0\quad {\rm for}\,\, a.e.\,\, x\in [0,L].
$$

{Note that $v^*_x(0)=v^*_x(L)=0$.}
Let ${c^{*},d^*}\in [0,L]$ be such that
$$
v^*({c^*})=\min_{x\in [0,L]} v^*(x),\quad v^*(d^*)=\max_{x\in[0,L]}v^*(x).
$$
Then $v_x^*({c^*})=v_x^*({d^*})=0$. By \eqref{case2-v-eq2}, we have
$$
\int_{c^{*}}^{d^{*}} v^*_{xx}v^*_xdx-\mu \int_{c^{*}}^{d^{*}} v^* v_x^*dx+\nu\int_{c^{*}}^{d^{*}} u^* v_x^*dx=0.
$$
This implies that
$$
v^*(c^*)=v^*(d^*).
$$
Hence  there is $\bar v\ge 0$ such that
\begin{equation}
\label{case2-v-eq3}
v^*(x)=\bar v\quad \forall\, x\in [0,L]
\end{equation}
and then
\begin{equation*}
u^*(x)=\frac{\mu}{\nu} \bar v\quad a.e. \,\, x\in [0,L].
\end{equation*}
Note that
$$
\int_0^L {u^*(x)}dx\le\frac{a}{b}{L}.
$$
Hence $\bar v\le \frac{\nu}{\mu}\frac{a}{b}$.
(1)  is thus proved.

\smallskip

(2) First, suppose that
$$
u_n(x)\to u^*(x)\quad \text{weakly in}\,\, L^2(0,L).
$$
By Theorem \ref{property-bif-solu-thm}(5), without loss of generality, we may assume that
there is $v^*\in H^2(0,L)$ such that
$$
\lim_{n\to\infty} v_n(x)= v^*(x)\quad \text{in}\,\, C^1([0,L])
\quad {\rm and}\,\,\,
\lim_{n\to\infty} v_n(x)=v^*(x)\quad \text{weakly in}\,\, H^2(0,L),
$$
By the arguments in (1),  there is $0\le\bar u\le \frac{a}{b}$ such that
$$
u^*(x)=\bar u\quad a.e.\, x\in [0,L].
$$

Next, suppose that
$$
u_n(x)\to u^*(x)\quad \, a.e.\, x\in [0,L].
$$
By Theorem \ref{property-bif-solu-thm}(3), {$u_n$ is a bounded sequence in $L^2(0,L)$. Hence}
$$
\lim_{n\to\infty}u_n(x)=u^*(x)\quad \text{weakly in}\,\, L^2(0,L).
$$
Then there is $0\le\bar u\le \frac{a}{b}$ such that
$$
u^*(x)=\bar u\quad a.e.\, x\in [0,L].
$$
The theorem is thus proved.
\end{proof}

Finally, we prove Theorem \ref{global-bif-thm}.

\begin{proof}[Proof of Theorem \ref{global-bif-thm}]
It is easy to see that
$$
F(\chi,\frac{a}{b})=0\quad \forall\,\ \chi>0,
$$
$$
D_uF(\chi,\frac{a}{b}) u=u_{xx}+(\chi\mu -a)u-\frac{\chi \mu^2}{\nu} v \quad \forall\, u\in X,
$$
and
$$
D_{\chi}D_u F(\chi,\frac{a}{b})u=\mu u-\frac{\mu^2}{{\nu}} v\quad\forall \, u\in X,
$$
where $v$ is the solution of \eqref{new-new-v-eq}.

When $\chi=\chi_k^{{*}}$, we have
$$
\mathcal{N} (D_u F(\chi_k^{{*}},\frac{a}{b})={\rm span}\{{\phi_k}\}
$$
and
$$
\Big(D_{\chi}D_u F(\chi,\frac{a}{b}) \phi_k\Big) (x)=\frac{{\mu}k^2 \pi^2}{\mu L^2+k^2\pi^2}\cos\frac{k\pi x}{L}\not \in \mathcal{R} (D_u(\chi_k^{{*}},\frac{a}{b})),
$$
{where $\phi_k(x)=\cos\frac{ k\pi x}{L}$.}
It is clear that for any $u\in  X$ and $\chi>0$, ${ {D_uF(\chi,\frac{a}{b})}}$ is a Fredholm operator.
By
\cite[Theorem 4.4]{ShWa},  each of the sets $\mathcal{C}_k^+$ and $\mathcal{C}_k^-$ satisfies one of (i), (ii), (iii) in the statement.

Suppose that  $\mathcal{C}_k^+$  satisfies (i), i.e.,  $\mathcal{C}_k^+$  is not compact.
We claim that $\mathcal{C}_k^+$ extends to infinity in the positive direction of $\chi$.
For otherwise, there are ${(\chi_{n}^{*},u_n^{*})}\in \mathcal{C}_k^+$ such that ${{\chi_n^{*}}}\to \chi^{**}<\infty$, and there is no $u\in X$ such that $({\chi^{**}},u)\in\mathcal{C}_k^+$.
 By Theorem \ref{property-bif-solu-thm}, without loss of generality, we  may assume that there are $(u^*,v^*)$ such that
$$
u_n^*(x)\to u^*(x),\quad v_n^*(x)\to v^*(x)
$$
as $n\to\infty$ uniformly in $x\in [0,L]$.  Then both $u^*(x)$ and $v^*(x)$ are uniformly
continuous in $x\in [0,L]$.  By Theorem \ref{property-bif-solu-thm} again,
 $\sup_{x\in [0,L]}u_n^*(x)\ge \frac{a}{b}$,
we must have $\sup_{x\in [0,L] } u^*(x)\ge \frac{a}{b}$. This implies that
$$
\int_0^L u^*(x)dx>0
$$ and then
$$
\inf _{x\in [0,L]} v^*(x)>0.
$$
We then  have that $(u^*(x),v^*(x))$ is a positive stationary solution of \eqref{main-eq0}
with $\chi=\chi^{**}$.  By the connectness of $\mathcal{C}_k^+$, we have
$(\chi^{**},u^*)\in \mathcal{C}_k^+$, which is a contradiction. Therefore, the claim holds.

Similarly,   if ${\mathcal{C}_k^-}$  is not compact, then it extends to infinity in the positive direction of $\chi$. The theorem is thus proved.
\end{proof}

\section{Spiky stationary solutions}

In this section, we study spiky stationary solutions.  We first give the following definition.

\begin{definition}
\label{spike-solution-def}
Let $\{(u^*(\cdot;\chi_n), v^*(\cdot;\chi_n))\}$ be positive non-constant stationary solutions of \eqref{main-eq0} with $\chi=\chi_n$ and $\chi_n\to\infty$.
We say that  $u^*(\cdot;\chi_n)$ develops spikes at $x^*\in [0,L]$ as $n\to\infty$ if there are  $\sigma^*>0$ and $\delta^*>0$ such that
\begin{equation}
\label{new-spike-eq1}
\liminf_{n\to\infty}\Big(\max_{x\in [x^*-\delta,x^*+\delta]\cap[0,L]}u^*(x;\chi_n)-
\min_{x\in [x^*-\delta,x^*+\delta]\cap[0,L]}u^*(x;\chi_n)\Big)\ge \sigma^* \quad \forall\,\, 0<\delta<\delta^*.
\end{equation}
Such $x^*$ is called a spiky point of $\{(u^*(\cdot;\chi_n), v^*(\cdot;\chi_n))\}$. If $x^*\in\{0,L\}$ (resp. $x^*\in (0,L)$), we say that $u^*(\cdot;\chi_n)$ develops boundary spikes (resp. interior spikes).
\end{definition}

\begin{remark}
\label{spike-rk1}
Let $\{(u^*(\cdot;\chi_n), v^*(\cdot;\chi_n))\}$ be positive non-constant stationary solutions of \eqref{main-eq0} with $\chi=\chi_n$ and $\chi_n\to\infty$.
\begin{itemize}

\item[(1)]
 The condition  \eqref{new-spike-eq1}
 indicates that for any $\delta>0$,  $\{(u^*(\cdot;\chi_n), v^*(\cdot;\chi_n))\}$  is not near a constant function on $[x^*-\delta,x^*+\delta]\cap[0,L]$ for $n\gg 1$.

\item[(2)]   If \eqref{new-spike-eq1} holds,  then  for any $0<\delta\ll 1$,
\begin{equation}
\label{new-spike-eq2}
\liminf_{n\to\infty}\max_{x\in [x^*-\delta,x^*+\delta]\cap [0,L]}|u^*_x(x;\chi_n)|=\infty,
\end{equation}
which   indicates that  $u^*(x;\chi_n)$ changes quickly near $x^*$.
In fact,  assume that \eqref{new-spike-eq2} does not hold.  Then there is $0<\delta_0<\delta^*$ such that
$$
\liminf_{n\to\infty}\max_{x\in [x^*-\delta_0,x^*+\delta_0]\cap [0,L]}|u^*_x(x;\chi_n)|<\infty.
$$
Then there are $M>0$ and $n_k\to\infty$ such that
$$
|u^*_x(x;\chi_{n_k})|\le M\quad\forall\,\, x\in [x^*-\delta_0,x^*+\delta_0]\cap[0,L],\,\, k\ge 1.
$$
By Theorem {\ref{property-bif-solu-thm} (3)},
$$
\int_0^L u^*(x;\chi_n)<\frac{aL}{b}\quad\forall \, n\ge 1.
$$
 This implies that there are $\tilde M>0$ and $x_n\in [x^*-\delta_0,x^*+\delta_0]\cap[0,L]$ such that
$$
u^*(x_n;\chi_n)=\min _{x\in [x^*-\delta_0,x^*+\delta_0]\cap[0,L]}u^*(x;\chi_n)\le \tilde M\quad \forall\, n\ge 1.
$$
Note also that
$$
u^*(x;\chi_{n})=u^*(x_{n};\chi_{n})+\int_{x_{n}}^ x u^*_x(x;\chi_{n})dx
$$
for all $ x\in [0,L]$ and $  n\ge 1$.
This implies that
 \begin{equation}
\label{new-spike-eq3}
\liminf_{n_k\to\infty}\Big(\max_{x\in [x^*-\delta,x^*+\delta]\cap[0,L]}u^*({x};\chi_{n_k})-
\min_{x\in [x^*-\delta,x^*+\delta]\cap[0,L]}u^*(x;\chi_{n_k})\Big)< \sigma^* \quad \forall\, 0<\delta\ll \delta_0,
\end{equation}
which contradicts to \eqref{new-spike-eq1}. Hence \eqref{new-spike-eq2} holds.
\end{itemize}
\end{remark}

\begin{theorem}[Spiky stationary solutions]
\label{spiky-solu-thm}
Let $\{(u(\cdot;\chi_n),v(\cdot;\chi_n))\}_{n\ge 1}$ be a sequence of non-constant positive stationary solutions of \eqref{main-eq0} with $\chi=\chi_n$ and $\chi_n\to\infty$.
\begin{itemize}
\item[(1)]
If   $x^*\in [0,L]$ satisfies  that  there is $m^*>\frac{a}{b}$ such that
\begin{equation}
\label{spike-assumption-eq0}
 \liminf_{n\to\infty} \max_{x\in [x^*-\delta,x^*+\delta]\cap[0,L]}u(x;\chi_n)\ge m^*\quad \forall\,\, 0<\delta\ll 1,
\end{equation}
then  $x^*$ is a spiky point of  the sequence $\{(u(\cdot;\chi_n),v(\cdot;\chi_n))\}_{n\ge 1}$.
In particular, if $x^*\in [0,L]$ satisfies that
\begin{equation}
\label{spike-assumption-eq0-1}
\liminf_{n\to\infty}u(x^*;\chi_n)>\frac{a}{b},
\end{equation}
then $x^*$  is a spiky point of  the sequence $\{(u(\cdot;\chi_n),v(\cdot;\chi_n))\}_{n\ge 1}$.

\item[(2)]  If
\begin{equation}
\label{spike-assumption-eq0-2}
\liminf_{n\to\infty}\max_{x\in [0,L]}u(x;\chi_n)>\frac{a}{b},
\end{equation}
then there are $\{n_k\}$ and  {$x^*\in [0,L]$} such that $x^*$ is a spiky point of  the sub-sequence $\{(u(\cdot;\chi_{n_k}),v(\cdot;\chi_{n_k}))\}_{k\ge 1}$.

\item[(3)]  If $\{u(\cdot;\chi_n)\}$ satisfies that
\begin{equation}
\label{spike-assumption-eq0-3}
\begin{cases}
 u_{x}(x;\chi_n)\le 0\,\, \, ({\rm resp.}\,\,  u_{x}(x;\chi_n)\ge 0) \quad \forall\,\, x\in (0,L),
\cr \cr
   \liminf_{n\to\infty} u(0;\chi_n)>\frac{a}{b}\,\,\, ({\rm resp.}\,\, \liminf_{n\to\infty} u(L;\chi_n)>\frac{a}{b}),
\end{cases}
\end{equation}
  then
$x^*=0$ (resp. $x^*=L$) is a boundary spiky point of $\{(u(\cdot;\chi_n),v(\cdot;\chi_n))\}_{n\ge 1}$. Moreover,  there is ${\bar u}\in [0,\frac{a}{b}]$ such that
\begin{equation}
\label{spiky-thm-eq1}
\lim_{{n}\to\infty} u(x;\chi_n)=\bar u \quad \text{locally uniformly in}\,\, x\in (0,L).
\end{equation}

In addition, if $\limsup_{n\to\infty} u(0;\chi_n)<\infty$ (resp. $\limsup_{n\to\infty} u(L;\chi_n)<\infty$) and   there are $0< c<d< L$ such that
$$
\limsup_{n\to\infty}\max_{x\in [c,d]}u(x;\chi_n)<\frac{a}{b},
$$
then $\bar u=0$.
\end{itemize}
\end{theorem}

\begin{remark}
\label{spike-rk2}
\begin{itemize}

\item[(1)] For the case that $a=b=\mu=\nu=L=1$, it is seen numerically that solutions in $\mathcal{C}_1^\pm $ are stable and for any $(\chi_n,{u(\cdot;\chi_n)})\in \mathcal{C}_1^+$ (resp. $\mathcal{C}_1^-$)  with
$\chi_n\to\infty$, $\liminf_{n\to\infty}{u(0;\chi_n)}>1$ (resp. $\liminf_{n\to\infty} { u(L;\chi_n)}>1$),
 $\{{u(\cdot;\chi_n)}\}$ develops spikes at the
boundary point $x^*=0$ (resp. $x^*=1$),  and $\lim_{n\to\infty} { u(x;\chi_{n})}=0$ locally uniformly in $(0,1]$  (resp. $[0,1)$) (see {Numerical Experiment 2}).

\item[(2)] For the case $a=b=\mu=\nu=1$ and $L=6$, {it is known that solutions {with $u$-components belonging to}
${\Gamma}_2^\pm(\chi)$ with $\chi\approx \chi_2^*$ are unstable.  Numerically, it is observed  that solutions {with $u$-components belonging to} $\mathcal{C}_2^\pm(\chi)$ with $\chi\gg \chi_{2}^{*}$ are locally stable and either a double boundary spike or a single interior spike appears as $\chi\to\infty$  (see {Numerical Experiment 4}).}

\item[(3)] For the case $a=2$, $b=\nu=1$, $\mu=3$ and $L=6$, {it is known that solutions {with $u$-components belonging to}
${\Gamma}_3^\pm(\chi)$ with $\chi\approx \chi_3^*$ are unstable.  Numerically, it is seen that solutions {with $u$-components belonging to} $\mathcal{C}_3^\pm(\chi)$ with $\chi\gg \chi_{3}^{*}$ are locally stable
and develops spikes at some boundary point and some interior point simultaneously  (see numerical simulations in subsection 5.4)}.

\end{itemize}
\end{remark}

Now, we prove Theorem \ref{spiky-solu-thm}.

\begin{proof}[Proof of Theorem \ref{spiky-solu-thm}]
 (1) Assume that  $x^*$ satisfies \eqref{spike-assumption-eq0}.  Put
$$(u_n(x),v_n(x))=(u(x;\chi_n),v(x;\chi_n)).
$$
Note that, to prove that $x^*$ is a spiky point,  it is to prove that there are $\delta^*>0$ and  $\sigma^*>0$ such that
\begin{equation}
\label{new-spike-eq1-0}
\liminf_{n\to\infty}\Big(\max_{x\in[x^*-\delta,x^*+\delta]\cap[0,L]}u_n(x)-
\min_{x\in [x^*-\delta,x^*+\delta]\cap[0,L]}u_n(x)\Big)\ge \sigma^* \quad \forall\,\, 0<\delta<\delta^*.
\end{equation}
We prove this by contradiction. Assume that \eqref{new-spike-eq1-0} does not hold. Then for any $m\ge 1$, there is $0<\delta_m<\frac{1}{m}$ such that
\begin{equation}
\label{new-spike-eq1-00}
\liminf_{n\to\infty}\Big(\max_{x\in[x^*-\delta_m,x^*+\delta_m]\cap[0,L]}u_n(x)-
\min_{x\in [x^*-\delta_m,x^*+\delta_m]\cap[0,L]}u_n(x)\Big)< \frac{1}{m}.
\end{equation}
Observe that, by Theorem \ref{property-bif-solu-thm}(3),  we always have
\begin{equation}
\label{aux-case1-eq}
\limsup_{n\to\infty}\min_{x\in[x^*-\delta,x^*+\delta]\cap[0,L]}u_n(x) <\infty\quad \forall\,\, 0<\delta\ll 1.
\end{equation}

Assume that \eqref{new-spike-eq1-00} holds for any $m\ge 1$.  We claim that for any $m\ge 1$,
$$
\liminf_{n\to\infty}\max_{x\in[x^*-\delta_m, x^*+\delta_m]\cap[0,L]}u_n(x)<\infty.
$$
For otherwise, there is $m_0\geq 1$ such that
$$
\liminf_{n\to\infty}\max_{x\in[x^*-\delta_{m_0}, x^*+\delta_{m_0}]\cap[0,L]}u_n(x)=\infty.
$$
Then, by \eqref{aux-case1-eq}, \eqref{new-spike-eq1-00} does not hold with $m=m_0$, which is a contradiction. Hence  the claim holds.

Fix $m\gg 1$ with $\frac{1}{m}<m^*-\frac{a}{b}$.  Let  $\tilde m^*:=\liminf_{n\to\infty}\max_{x\in[x^*-\delta_m,x^*+\delta_m]\cap[0,L]} u_n(x)$.
By \eqref{spike-assumption-eq0}, $\tilde m^*\ge m^*(>\frac{a}{b})$.
By  \eqref{new-spike-eq1-00},  there is  $n_k\to\infty$   such that
\begin{equation}
\label{new-aux-eq}
\lim_{n_k\to\infty}\Big(\max_{x\in [x^*-\delta_m,x^*+\delta_m]\cap[0,L]}u_{n_k}(x)-\min_{x\in [x^*-\delta_m,x^*+\delta_m]\cap [0,L]}u_{n_k}(x)\Big) =\sigma_m^*\le  \frac{1}{m}.
\end{equation}
By \eqref{aux-case1-eq},
without loss of generality, we may assume that there is $M^*\in [\tilde m^*,\infty)\subset (\frac{a}{b},\infty)$ such that
$$\lim_{n_k\to \infty} \max_{x\in [x^*-\delta_m,x^*+\delta_m]\cap[0,L]}u_{n_k}(x)=M^*
$$
This together with \eqref{new-aux-eq} implies that
\begin{equation}
\label{u-star-eq1}
\lim_{n_k\to\infty}\min_{x\in [x^*-{\delta_m},x^*+{\delta_m}]\cap[0,L]}{u_{n_k}(x)}>\frac{a}{b}.
\end{equation}

By Theorem \ref{property-bif-solu-thm}(3),  without loss of generality, we may assume that there is $u^*\in L^2(0,L)$ such that
$$
\lim_{n_k\to\infty}u_{n_k}(x)=u^*(x)\quad \text{weakly in}\,\, L^2(0,L).
$$
and
$$
\int_0^L u^*(x)dx=\lim_{n_k\to\infty}\int_0^L u_{n_k}(x)dx\le \frac{a}{b}L.
$$
By Theorem \ref{property-bif-solu-thm1}(2),  there is $\bar u\le\frac{a}{b}$ such that
$$
u^*(x)=\bar u\quad a.e.\, x\in [0,L].
$$
By \eqref{u-star-eq1}, we must have
$\bar u>\frac{a}{b}$ and then $\int_0^L u^*(x)dx>\frac{a}{b}L$, which is a contradiction. Therefore, \eqref{new-spike-eq1-00} cannot hold for all $m\ge 1$, and then  $x^*$ is a spiky point.

\smallskip

(2) Let $x_n\in [0,L]$ be such that
$$
u_n(x_n)=\max_{x\in [0,L]}u_n(x).
$$
Then there are $n_k\to\infty$ and $x^*\in [0,L]$
such that
$$
\lim_{n_k\to\infty} x_{n_k}=x^*.
$$
This implies that for any $\delta>0$,
$$
\max_{x\in [x^*-\delta,x^*+\delta]\cap[0,L]}u_{n_k}(x)= u_{n_k}(x_{n_k})\quad \forall\, k\gg 1.
$$
It then follows that for any $\delta>0$,
\begin{align*}
\liminf_{n_k\to\infty} \max_{x\in[x^*-\delta,x^*+\delta]\cap[0,L]}u_{n_k}(x)&=\liminf_{n_k\to\infty} u_{n_k}(x_{n_k})\\
&=\liminf_{n_k\to\infty}\max_{x\in [0,L]}u_{n_k}(x)\\
&\ge \liminf_{n\to\infty}\max_{x\in [0,L]}u_n(x)\\
&>\frac{a}{b}.
\end{align*}
Then by (1), $x^*$ is a spiky point of the subsequence $\{u_{n_k}(x)\}$.

\smallskip

(3) Assume that $u_{nx}(x)\le 0$ for $x\in [0,L]$. The case that $u_{nx}(x)\ge 0$ for $x\in[0,L]$ can be proved similarly.  By  (1),  $x^*$ is a spiky point of
$\{u(x;\chi_n)\}_{n\ge 1}$.  We prove that \eqref{spiky-thm-eq1} holds.

To this end, first, note that
  for any $x_0\in (0,L]$, we must have
\begin{equation}
\label{spike-eq1-000}
\limsup_{n\to\infty} u_n(x_0)<\infty.
\end{equation}
For otherwise, there is $x_0\in (0,L]$ such that $\limsup_{n\to\infty} u_n(x_0)=\infty$.
Without loss of generality, we may assume that
$\lim_{n\to\infty}u_n(x_0)=\infty$. This implies that
$$
\liminf_{n\to\infty}\int_0^L u_n(x)dx\ge \liminf_{n\to\infty}\int_0 ^{x_0} u_n(x)dx
\ge \liminf_{n\to\infty} u_n(x_0) \cdot x_0=\infty,
$$
which is a contradiction.
Hence \eqref{spike-eq1-000} holds for any $x_0\in (0,L]$. Therefore,
$\{u_n\}$ is a bounded sequence of monotone decreasing functions on $[x_0,L]$
for any $x_0\in (0,L]$.

Next, by Helly's theorem, without loss of generality, we may assume that
there is a nonincreasing function {$\hat u^{*}$}  on $(0,L]$ such that
$$
u_n\to \hat u^{*} \quad \text{pointwise in} \,\, (0,L].
$$
By Theorem \ref{property-bif-solu-thm1},
 there is $\bar u\le \frac{a}{b}$ such that
\begin{equation}
\label{u-star-eq0}
\hat u^*(x)=\bar u\quad {\rm a.e.}\,\, x\in (0,L].
\end{equation}
{Moreover, for any ${[x_1,x_2]}\subset (0,L)$, there are $\tilde x_1\in (0,x_1)$ and $\tilde x_2\in (x_2,L)$ such that
$$
\lim_{n\to\infty} u_n(\tilde x_1)=\bar u=\lim_{n\to\infty} {u_n}(\tilde x_2).
$$
Since
$$
u_n(\tilde x_1)\ge u_n(x)\ge u_n(\tilde x_2)\quad \forall\, x\in [x_1,x_2],
$$
we have that
$$
\lim_{n\to\infty} u_n(x)=\bar u\quad {\rm uniformly\,\, in}\,\, [x_1,x_2],
$$
and hence  \eqref{spiky-thm-eq1} holds.}

{In addition, if $\limsup_{n\to\infty} u(0;\chi_n)<\infty$, then $\{u(x;\chi_n)\}$ is a bounded sequence in $C([0,L])$.  By Dominated  Convergence Theorem,}
we have
$$
0=\lim_{n\to\infty}\int_0^L (au_n(x)-bu_n^2(x))dx=
\int_0^L (a\bar u-b \bar u^2)dx.
$$
This implies that $\bar u=0$ or $\bar u=\frac{a}{b}$.  Moreover,  if there are $0< c<d< L$ such that
$$
\limsup_{n\to\infty}\max_{x\in [c,d]}u_n(x)<\frac{a}{b},
$$
it is clear that $\bar u=0$.
Therefore, (3) holds and
the theorem is thus proved.

\end{proof}

\section{Numerical analysis}

In this section, we carry out some numerical analysis about the stability and bifurcation of
stationary solutions of \eqref{main-eq0}. First, in subsection \ref{scheme-sec}, we describe the numerical scheme to be used
in the simulations. We then discuss the simulations we carried out by the scheme  for three parameter settings: $a=b=\mu=\nu=L=1$; $a=b=\mu=\nu=1$ and $L=6$;  and $a=2$,
$b=\nu=1$, $\mu={3}$, $L=6$ in subsections \ref{l-1-sec}, \ref{l-6-sec}, and
\ref{a-2-l-6-sec}, respectively.

\subsection{Numerical scheme}
\label{scheme-sec}


In this subsection, we describe the scheme we use to perform numerical simulations of the solutions of \eqref{main-eq0}.

 For a given initial function $u_0$, to solve \eqref{main-eq0} numerically, we first simulate the solution of the second equation of \eqref{main-eq0} subject to Neumann Boundary condition to obtain the numerical solution of $v(t,x; u_0)$ using Matlab bvp4c command. Then we compute the solution of the first equation of \eqref{main-eq0} subject to the Neumann Boundary condition to get the numerical solution of $u(t,x;u_0)$ by the finite difference method in space and Runge-Kutta method in time. Observe that if $(u(x), v(x)):=\lim_{t\to\infty} (u(t, x;u_0), v(t,x;u_0))$ exists, then $(u(x), v(x))$ is a stationary solution of \eqref{main-eq0}. Hence, we run the simulation until the numerical solution changes very little. Then the numerical solution at the final time can be taken as a stationary solution. In all numerical simulations, the space step size and time step size are respectively taken as $0.01$ and $5\times 10^{-5}$. In the following subsections, we {fix $b=\nu=1$} and choose different values for {$a$, $\mu$,} $\chi$, $L$ and different initial functions to simulate the stability of the positive constant solution and bifurcation solutions from the positive constant solution.


\subsection{Numerical analysis for the case $a=b=\mu=\nu=L=1$}
\label{l-1-sec}

In this subsection, we carry out some numerical analysis about the stability and bifurcation of the positive constant solution   for the case
$a=b=\mu=\nu=L=1$. In this case, it is known that $\chi^*=\chi_1^*\approx 11.9709$;
$(\frac{a}{b},\frac{\nu}{\mu}\frac{a}{b})=(1,1)$ is locally  asymptotically stable when
$0<\chi<\chi^*$; and when $\chi$ passes through $\chi^*$, super-critical pitchfork bifurcation occurs. Throughout this subsection, $a=b=\mu=\nu=L=1$.

\smallskip

\noindent{\bf Numerical Experiment 1.} In this numerical experiment,  we explore  the global stability of the constant solution $(1,1)$. First, we let $\chi=5$ and initial function $u_0=1+0.5\cos(\pi x)$. We observe that as time goes by, the numerical solution of $(u(t,x;u_0), v(t,x;u_0))$ converges to $(1,1)$ (see Figure  \ref{chi-5} (a) for the limit of $u(t,x;u_0)$ and (b) for the evolution of $u(t,x;u_0)$).  Note that $v$ solves
\begin{equation}
\label{v-eqq}
\begin{cases}
0=v_{xx}-\mu v+\nu u,\quad 0<x<L\cr
v_x(t,0)=v_x(t,L)=0.
\end{cases}
\end{equation}
Hence, if  $u(t,x;u_0)$ converges to some function $u_1(x)$ as $t\to\infty$, then $v(t,x;u_0)$ {converges} to $v_1(x)$ as $t\to\infty$, where $v_1(x)$ is the unique solution of \eqref{v-eqq} with $u(\cdot)$ being replaced by $u_1(\cdot)$. In Figure \ref{chi-5} as well as all  other figures, except Figure \ref{chi-20-1},
we then only present the limit profile and evolution of $u(t,x;u_0)$ as $t$ changes.
    The same phenomenon can be seen when  $u_0$ is replaced by $\tilde u_0=1-0.5\cos(\pi x)$.
In fact, we have $\tilde u_0(x)=u_0(1-x)$ and then $(u(t,x;\tilde u_0), v(t,x;\tilde u_0))=  (u(t,1-x;u_0),v(t,1-x;u_0))$ for any $t>0$ and $x\in [0,1]$.
Hence we do not present the pictures for this initial function.


\begin{figure}[!ht]
\begin{center}
\subfigure[]{
\resizebox*{0.36\linewidth}{!}{\includegraphics{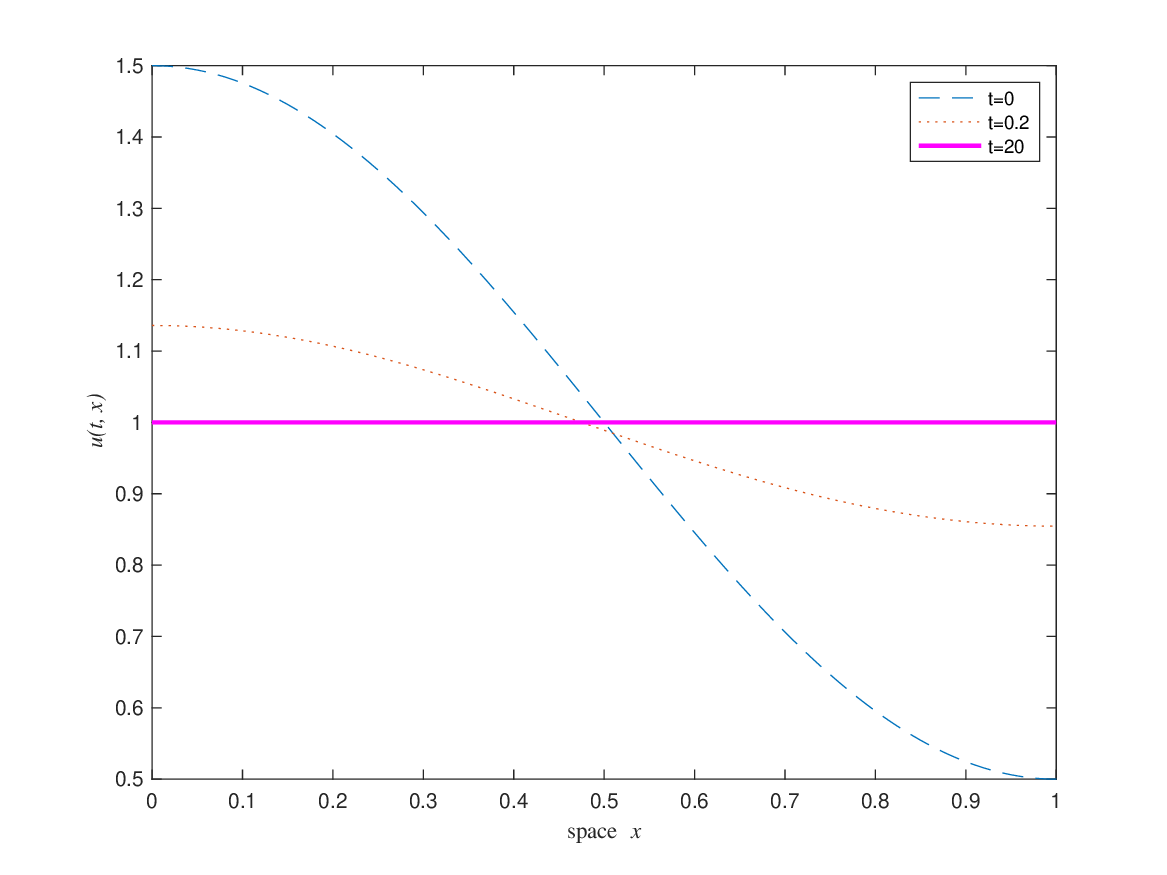}}
}
\subfigure[]{\resizebox*{0.36\linewidth}{!}{\includegraphics{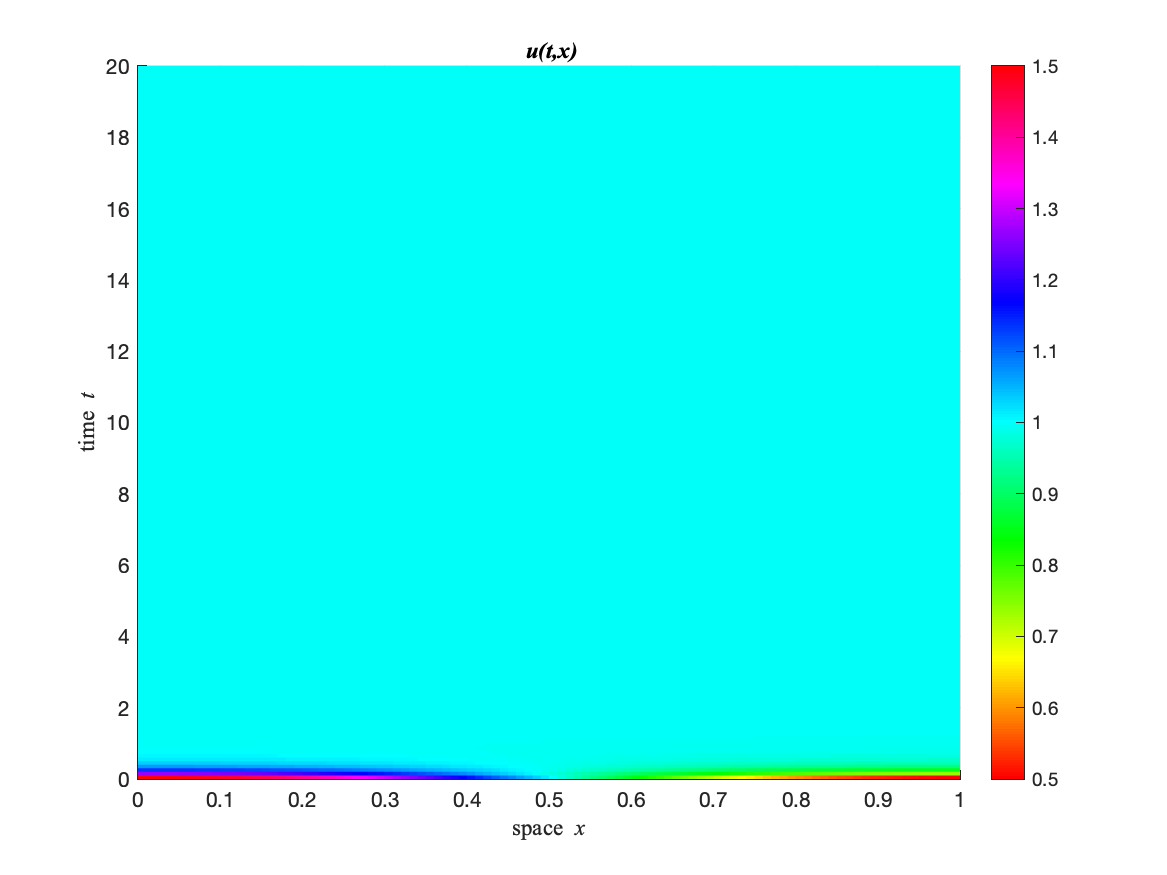}}\label{para-elli-chi-5-u0-1+0.5cos-u}}
\caption{(a) limit profile, (b)  evolution of  $u(t,x;u_0)$  with $\chi=5$, $a=b=\mu=\nu=L=1$ and initial function $u_0=1+0.5\cos(\pi x)$ 
}
\label{chi-5}
\end{center}
\end{figure}

Next, we  take 
$\chi=11.96$, which {is} very close to $\chi^{*}$. The same phenomenon is observed as well
(see 
Figure \ref{chi-11.96}). To see numerically whether the positive constant solution $(1,1)$ is globally stable, we also choose initial functions $u_0=1\pm 0.5\cos(\pi x)\pm 0.1\cos(2\pi x)$. All the other parameters remain the same. We observe that the numerical solution eventually converges to $(1,1)$ for each initial function.

\begin{figure}[!ht]
\begin{center}
\subfigure[]{
\resizebox*{0.36\linewidth}{!}{\includegraphics{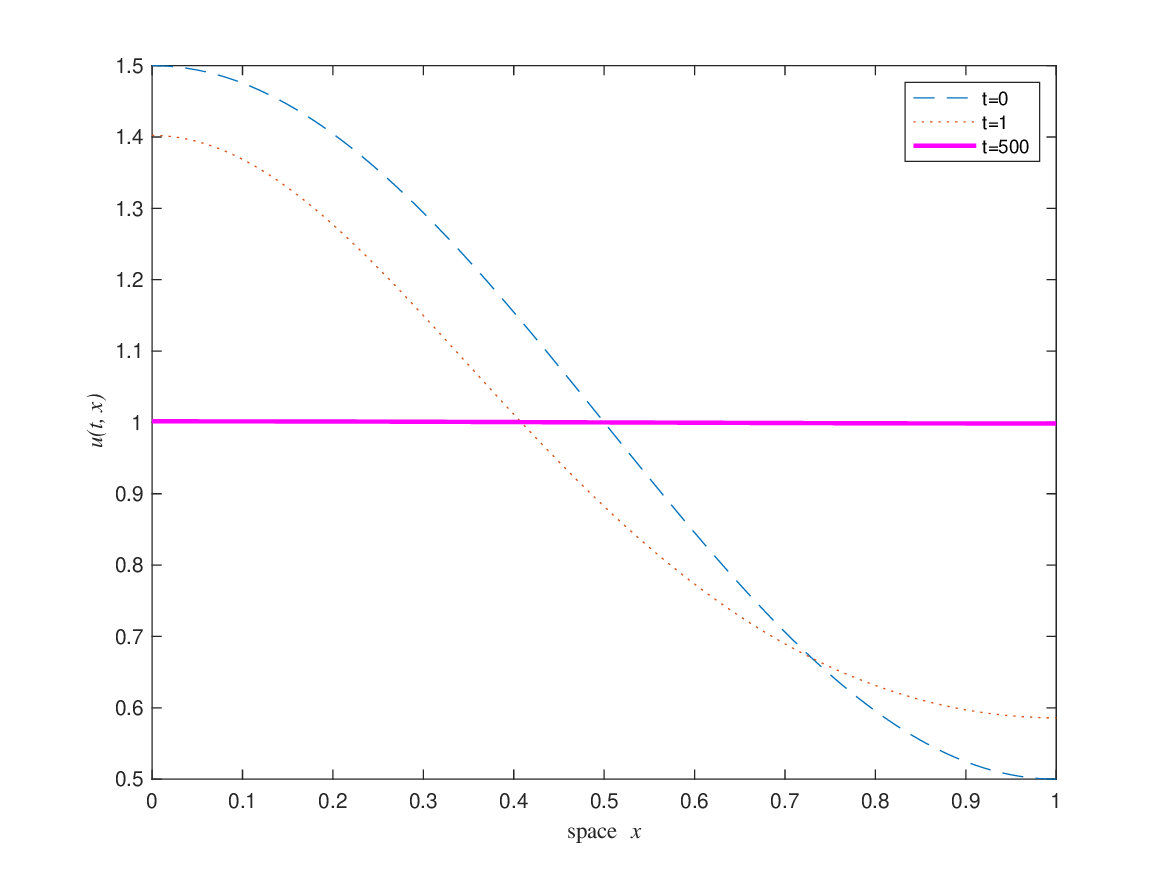}}
}
\subfigure[]{\resizebox*{0.36\linewidth}{!}{\includegraphics{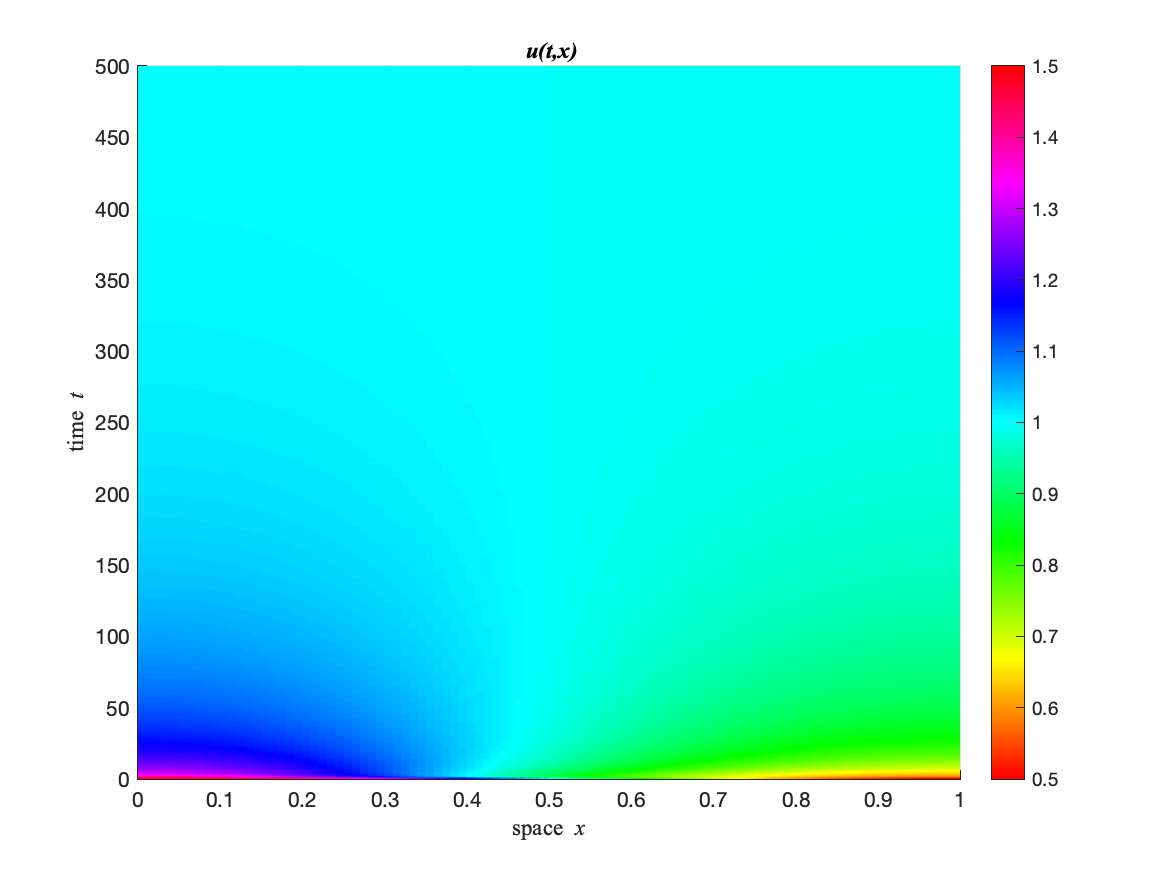}}\label{para-elli-chi-11.96-u0-1+0.5cos-u}}
\caption{(a) limit profile, (b) evolution of $u(t,x;u_0)$  with $\chi=11.96$, $a=b=\mu=\nu=L=1$ and initial function $u_0=1+0.5\cos(\pi x)$ 
}
\label{chi-11.96}
\end{center}
\end{figure}

\noindent {\bf Observations from Experiment 1.} It is known that the constant solution $(1,1)$ is locally stable when $0<\chi<\chi^*$, is unstable when $\chi>\chi^*$,  and super-critical pitchfork bifurcations occurs when $\chi$ passes through $\chi^*$.   It is observed from  the experiment 1 that the constant solution $(1,1)$ is also  stable with respect to large perturbations when $0<\chi<\chi^*$. We conjecture that
the constant solution $(1,1)$ is globally stable when $0<\chi<\chi^*$.

\smallskip

\smallskip

\noindent{\bf Numerical Experiment 2.} In this numerical experiment,  we explore  the global  bifurcation of the constant solution $(1,1)$.  First of all, we choose $\chi=11.98$,  which is slightly larger than the first bifurcation value. Let $u_0=1+0.5\cos(\pi x)$. We observe that
 the numerical solution of $(u(t,x;u_0), v(t,x;u_0))$ changes very little when $t$ large enough   and converges  to a  nonconstant stationary solution $(u_1(x),v_1(x))$, which is close to the constant solution $(1,1)$  and    corresponds to the analytical nonconstant stationary solution described in the first formula of \eqref{upper-branch-eq} (see Figure \ref{chi-11.98} for  the profile of the numerical solution
 at $t=360$, which is close to the profile of the nonconstant stationary solution).  This  implies that, when $\chi=11.98$,  there exist two nonconstant stationary solutions $(u_1(x), v_1(x))$ and $(u_2(x), v_2(x))$, where  $u_2(x)=u_1(1-x)$, $v_2(x)=v_1(1-x)$,
 a fact that can be seen from the formulas \eqref{upper-branch-eq} and \eqref{lower-branch-eq}.

\begin{figure}[!ht]
\begin{center}
\subfigure[]{
\resizebox*{0.36\linewidth}{!}{\includegraphics{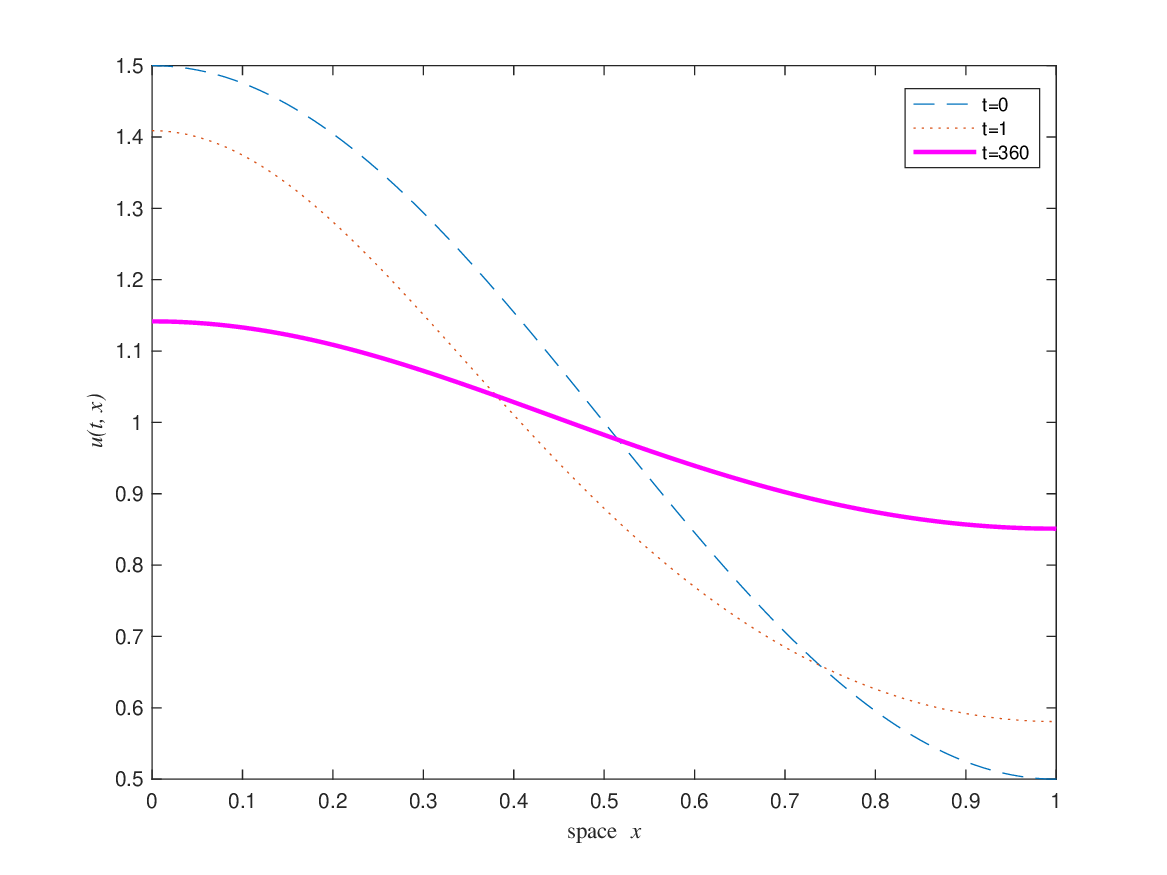}}
}
\subfigure[]{\resizebox*{0.36\linewidth}{!}{\includegraphics{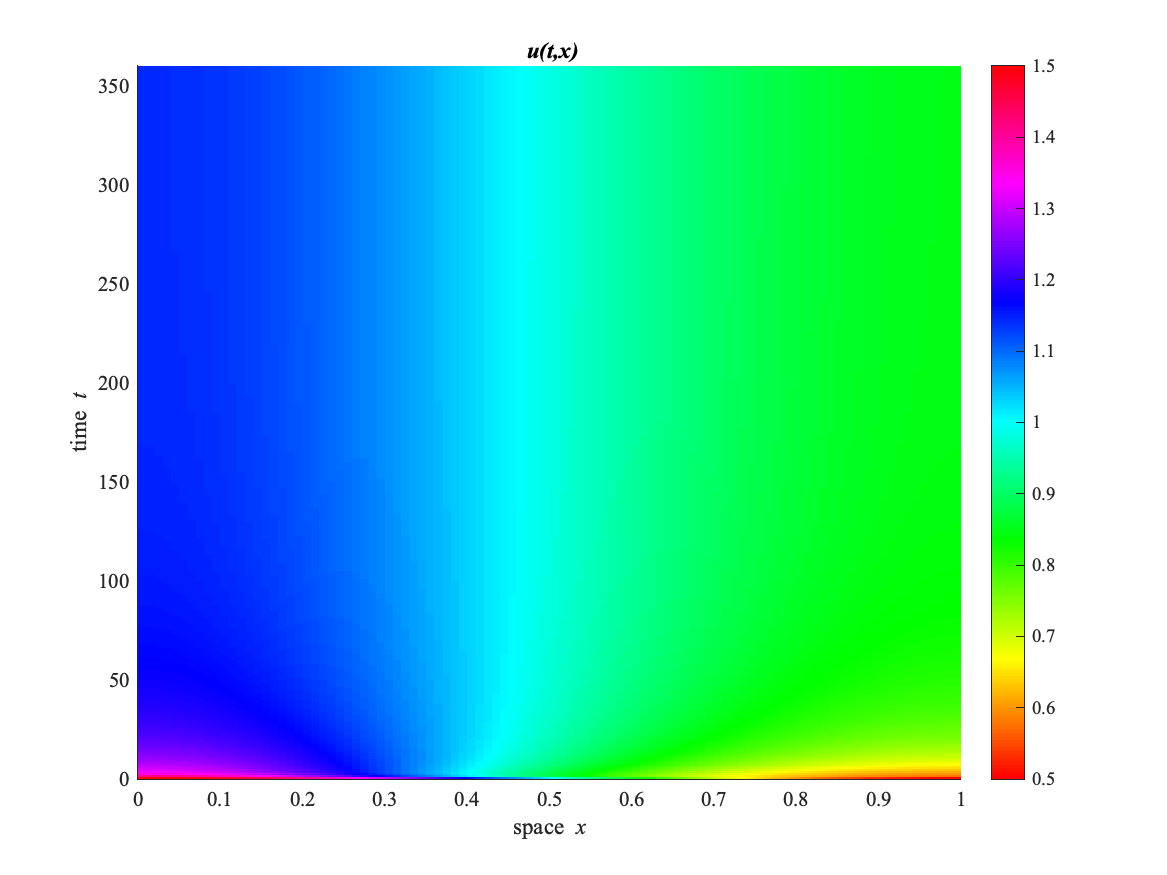}}
}
\caption{(a) limit profile, (b)  evolution of $u(t,x;u_0)$  with $\chi=11.98$, $a=b=\mu=\nu=L=1$ and initial function $u_0=1+0.5\cos(\pi x)$ 
}
\label{chi-11.98}
\end{center}
\end{figure}

Next, we increase the value of $\chi$. Let $\chi=20$, which is not close to the bifurcation value $\chi^*\approx 11.9709$.  Let $u_0=1+0.5\cos(\pi x)$. All the other parameters remain the same. We observe that as time evolves, the numerical solution of {$(u(t,x;u_0), v(t,x;u_0))$} converges to a {nonconstant} stationary solution, the $u$-component of which  has a spike near the boundary $x=0$  (see Figure \ref{chi-20}), but the $v$-component does not develop spikes (see Figure \ref{chi-20-1}).

\begin{figure}[!ht]
\begin{center}
\subfigure[]{
\resizebox*{0.36\linewidth}{!}{\includegraphics{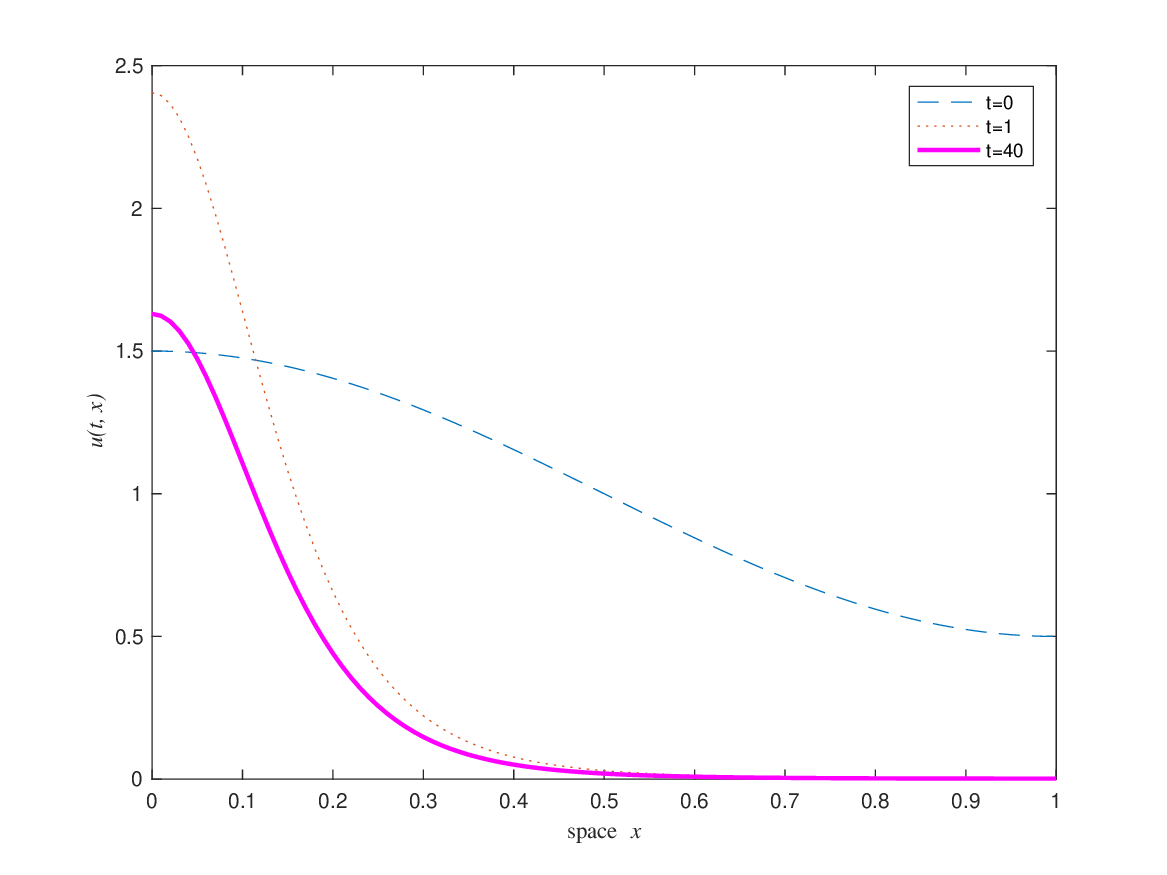}}
}
\subfigure[]{
\resizebox*{0.36\linewidth}{!}{\includegraphics{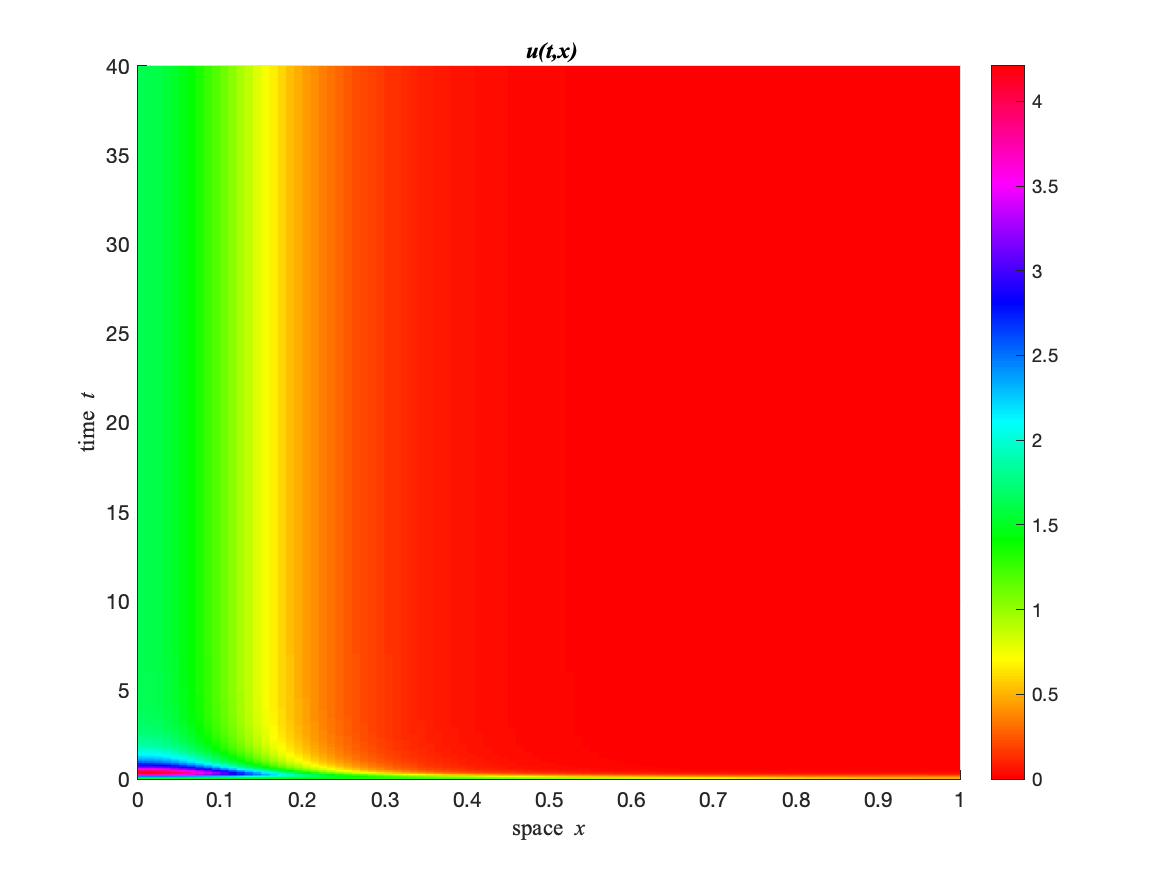}}\label{para-elli-chi-20-u0-1+0.5cospix-u}}
\caption{(a) limit profile, (b)  evolution of $u(t,x;u_0)$  with $\chi=20$, $a=b=\mu=\nu=L=1$ and initial function $u_0=1+0.5\cos(\pi x)$ 
}
\label{chi-20}
\end{center}
\end{figure}

\begin{figure}[!ht]
\begin{center}
\subfigure[]{
\resizebox*{0.36\linewidth}{!}{\includegraphics{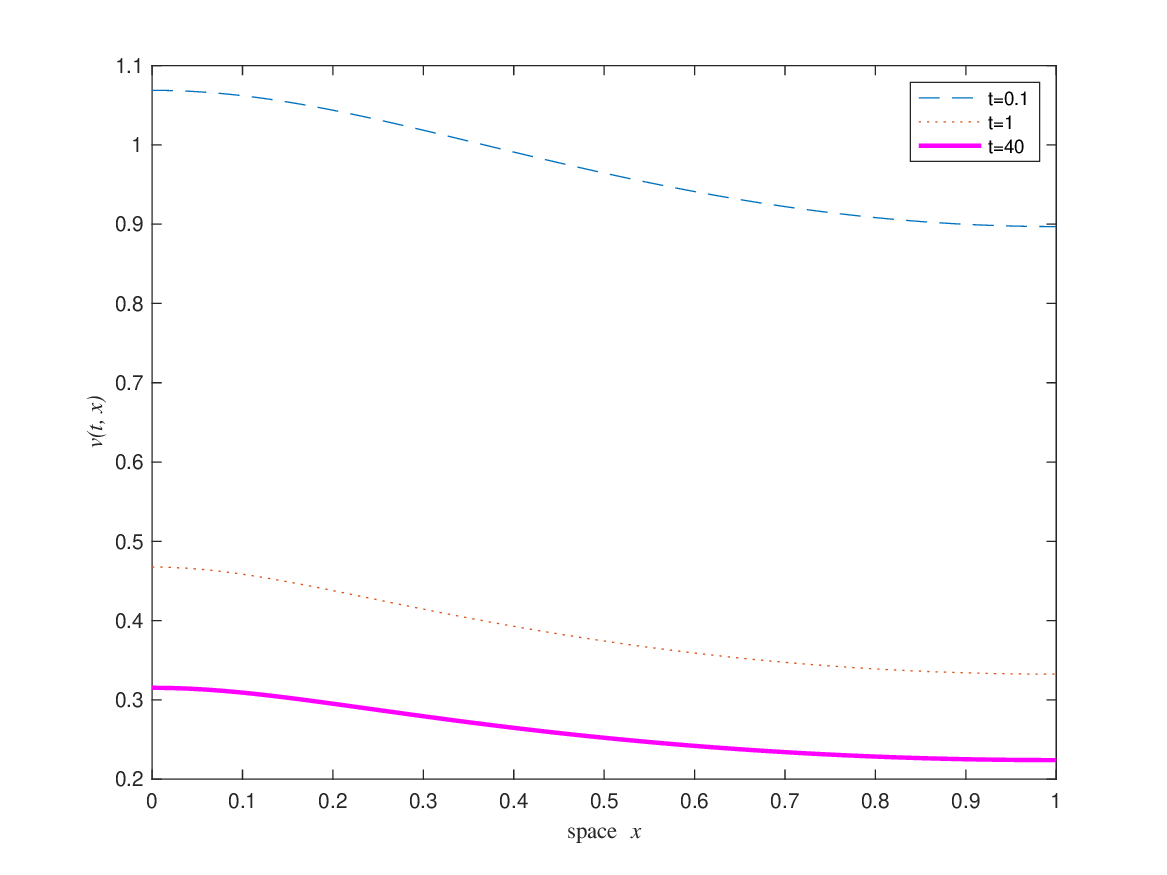}}
}
\subfigure[]{
\resizebox*{0.36\linewidth}{!}{\includegraphics{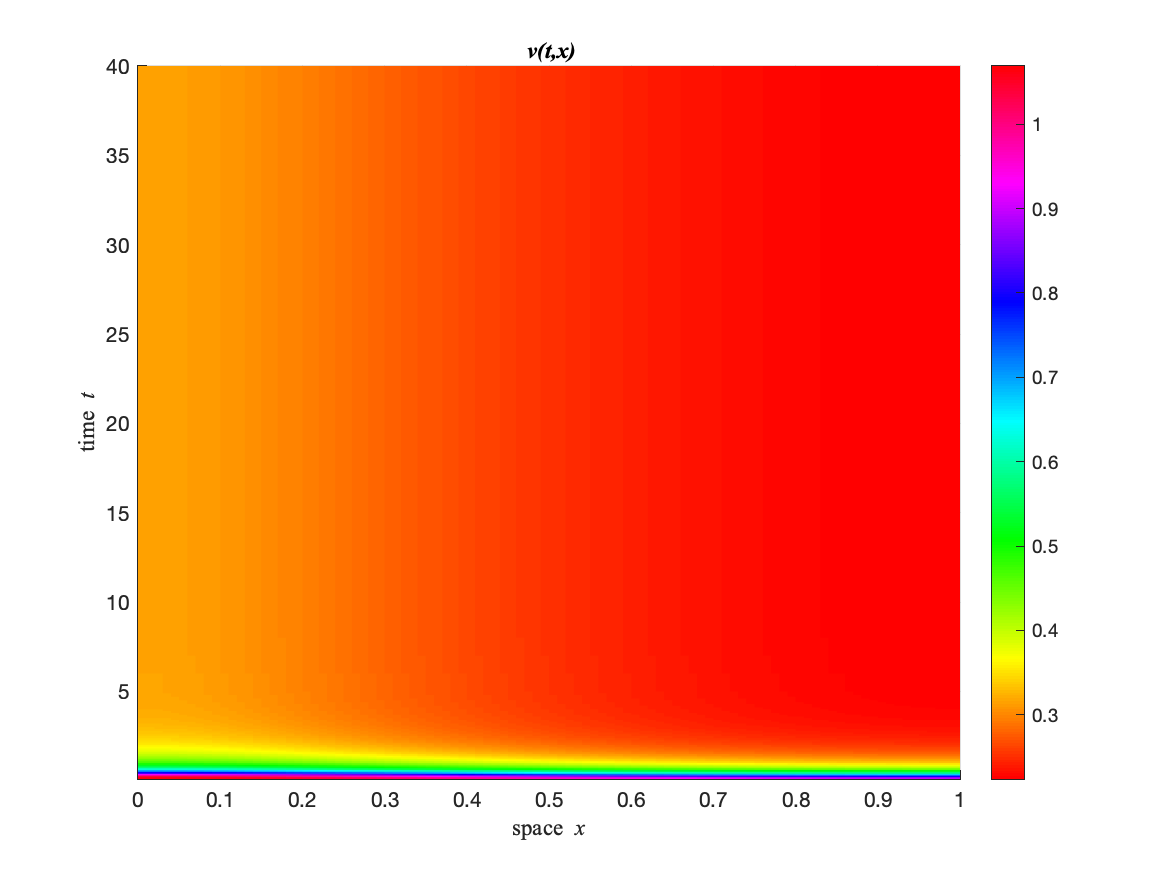}}\label{para-elli-chi-20-u0-1+0.5cospix-v}}
\caption{(a) limit profile, (b)  evolution  of $v(t,x;u_0)$  with $\chi=20$, $a=b=\mu=\nu=L=1$ and initial function $u_0=1+0.5\cos(\pi x)$ 
}
\label{chi-20-1}
\end{center}
\end{figure}

We further increase the value of $\chi$.  Let $\chi=40$ and take $u_0=1+ 0.5\cos(\pi x)$. The same phenomenon is observed (see Figure \ref{chi-40}). In particular, we observe that when $\chi$ becomes larger, the $u$-component of the numerical nonconstant stationary solution concentrates more towards the boundary $x=0$ (see Figure \ref{chi-40}),  which is consistent with Theorem {\ref{spiky-solu-thm} (3)}. By symmetry, the solution of \eqref{main-eq0} with initial function $u_0=1-0.5\cos(\pi x)$ converges to a nonconstant stationary solution,
which has a {spike} near the boundary $x=1$ (see Figure \ref{chi-40-1}).

\begin{figure}[!ht]
\begin{center}
\subfigure[]{
\resizebox*{0.36\linewidth}{!}{\includegraphics{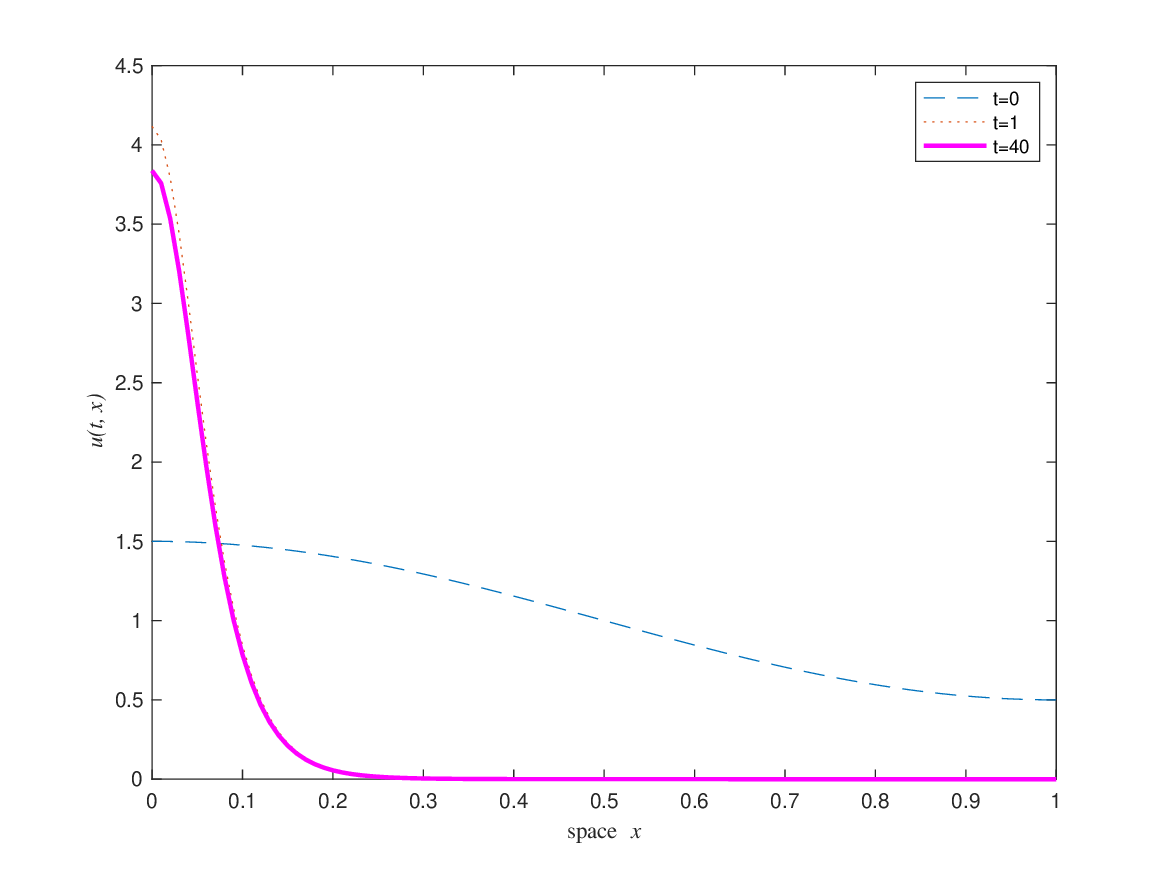}}
}\subfigure[]{
\resizebox*{0.36\linewidth}{!}{\includegraphics{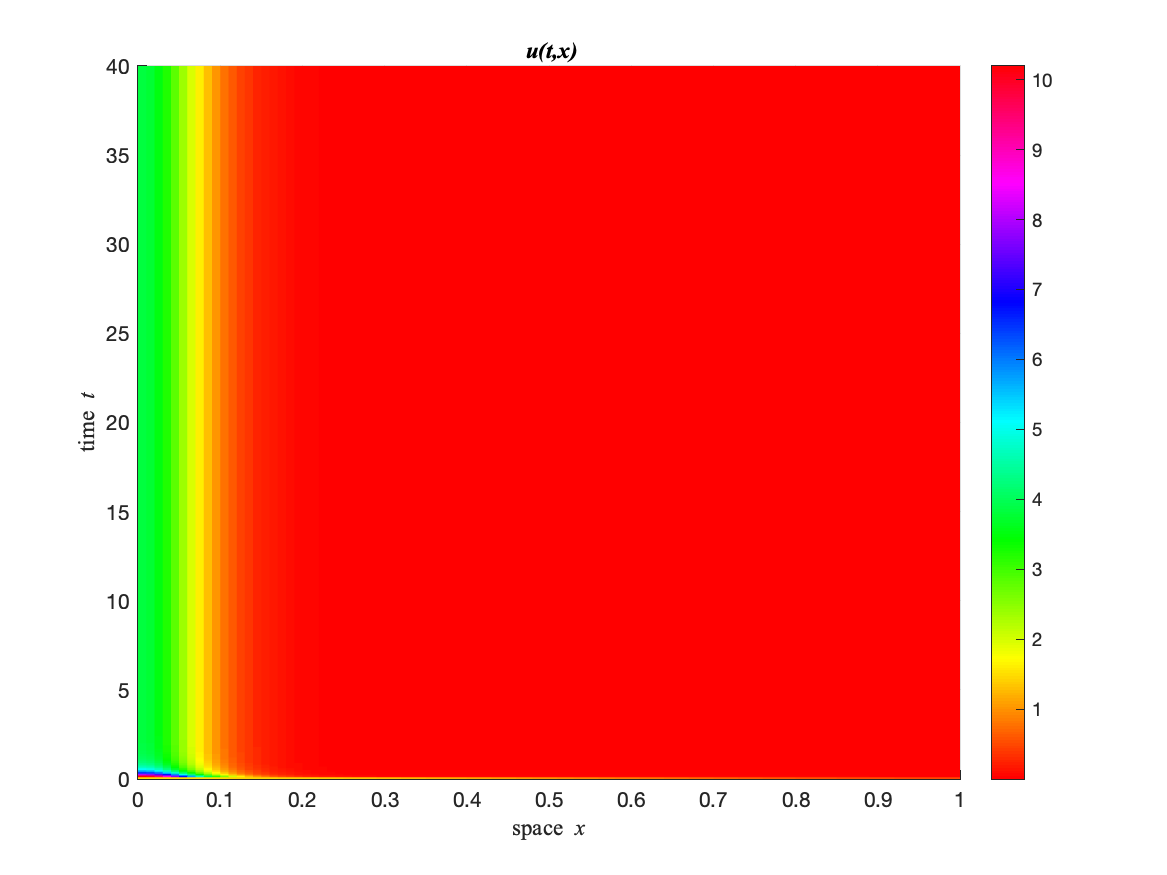}}}
\caption{(a) limit profile, (b)  evolution of $u(t,x;u_0)$  with $\chi=40$, $a=b=\mu=\nu=L=1$ and initial function $u_0=1+0.5\cos(\pi x)$}
\label{chi-40}
\end{center}
\end{figure}

\begin{figure}[!ht]
\begin{center}
\subfigure[]{
\resizebox*{0.36\linewidth}{!}{\includegraphics{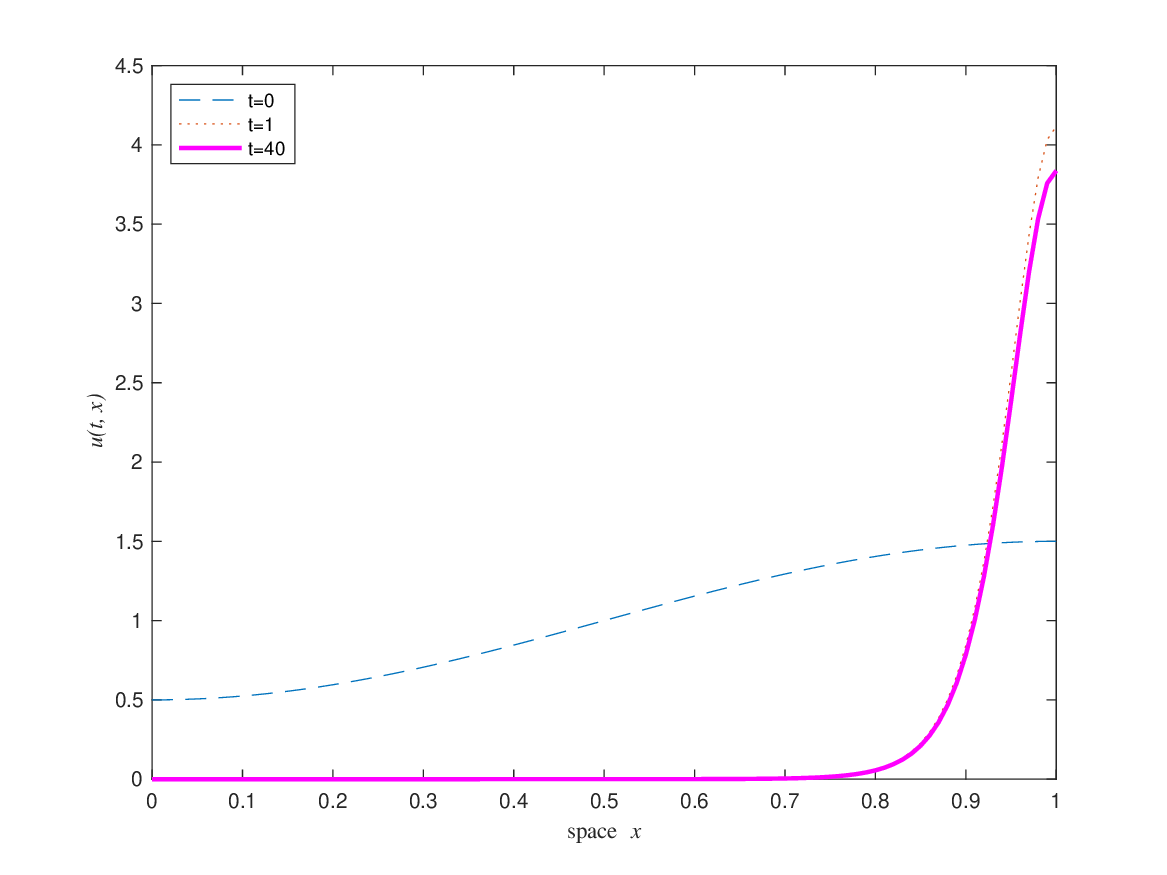}}
}\subfigure[]{
\resizebox*{0.36\linewidth}{!}{\includegraphics{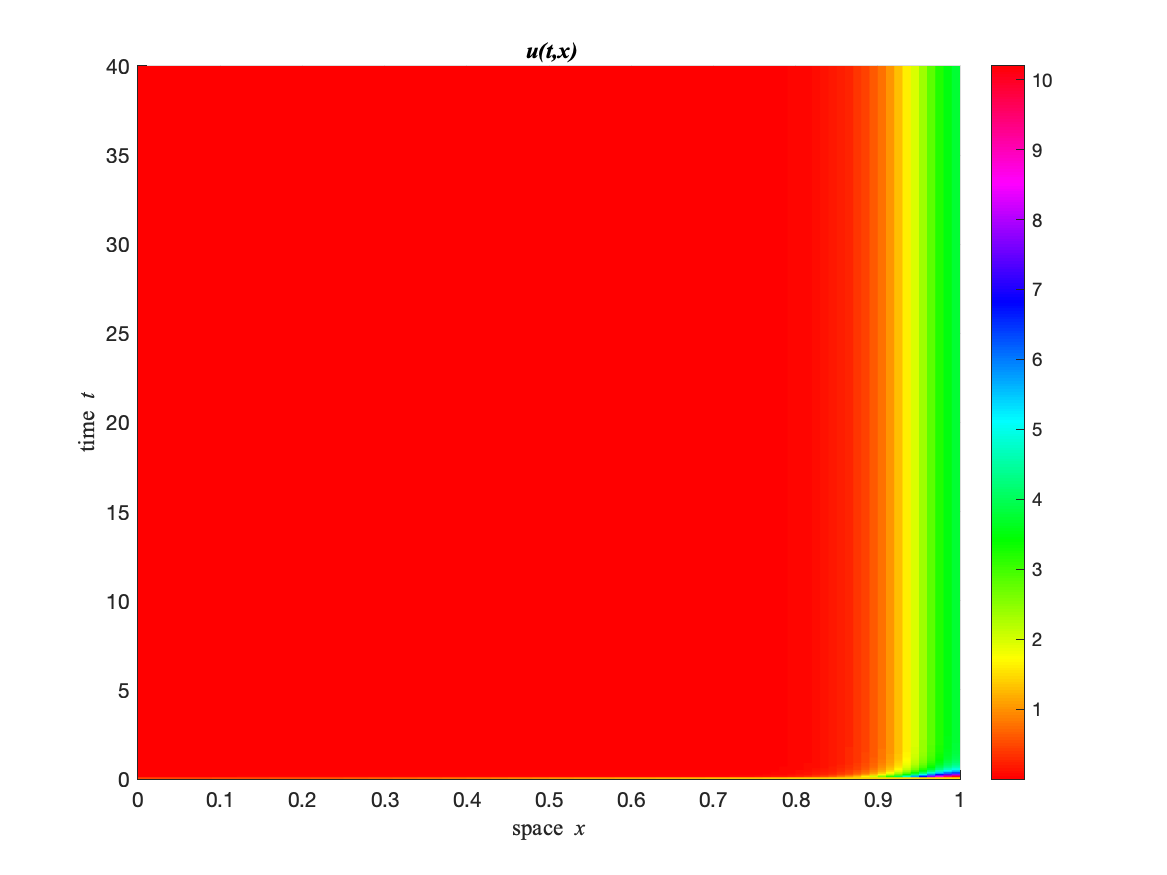}}}
\caption{(a) limit profile, (b)  evolution of $u(t,x;u_0)$   with $\chi=40$, $a=b=\mu=\nu=L=1$ and  initial function $u_0=1-0.5\cos(\pi x)$}
\label{chi-40-1}
\end{center}
\end{figure}

\smallskip

\noindent {\bf Observations from Experiment 2.}
For each $\chi$ in the experiment 2,
the same phenomenon is observed for the numerical solution with  the initial function $u_0=1+0.5\cos(\pi x)\pm 0.1\cos(2\pi x)$,
 that is,  the numerical  solutions  with initial functions
$u_0(x)=1+0.5\cos(\pi x)$ and $u_0(x)=1+0.5\cos(\pi x)\pm 0.1 \cos(2\pi x)$ converge
to the same nonconstant stationary solution.
 This indicates the nonconstant stationary solution is stable. According to the theoretical results, super-critical pitchfork  bifurcation occurs when $\chi$ passes through $\chi^{*}$ and  the  bifurcation solution is then locally stable for $\chi$ near $\chi^*$. Our numerical simulation confirms this result.
The numerical simulation also indicates that  this local bifurcation branch extends to $\chi=\infty$ and the bifurcation solutions are locally stable. Moreover, as $\chi$ increases,
the $u$-component of the bifurcation solution develops a spike near the boundary $x=0$ or $x=1$, but the $v$-component does not develop any spikes. In fact,
it is proved in {Theorem \ref{property-bif-solu-thm}} that for any stationary solutions of \eqref{main-eq0},
$v$ and $v_x$ stay bounded as $\chi$ increases, which implies that the $v$-component of the bifurcation solution does not develop spikes.

\subsection{Numerical simulations for the case $a=b=\mu=\nu=1$ and $L=6$}
\label{l-6-sec}

In this subsection, we discuss the numerical simulations for the case
$a=b=\mu=\nu=1$ and $L=6$. In this case, it is known that $\chi^*=\chi_2^*\approx 4.0085$;
$(\frac{a}{b},\frac{\nu}{\mu}\frac{a}{b})=(1,1)$ is locally  asymptotically stable when
$0<\chi<\chi^*$; and when $\chi$ passes through $\chi^*$, sub-critical pitchfork bifurcation occurs.
Throughout this subsection, $a=b=\mu=\nu=1$ and $L=6$.

\smallskip

\noindent{\bf Numerical Experiment 3.} In this numerical experiment, we   investigate the global  stability of the constant solution $(1,1)$.
 First of all, let $\chi=2$ and initial function $u_0=1\pm 0.5\cos(\frac{\pi x}{3})$. We observe that as time goes by, the numerical solution of $(u(t,x;u_0), v(t,x;u_0))$ converges to $(1,1)$ (see Figures \ref{chi-2} and \ref{chi-2-1}).

\begin{figure}[!ht]
\begin{center}
\subfigure[]{
\resizebox*{0.36\linewidth}{!}{\includegraphics{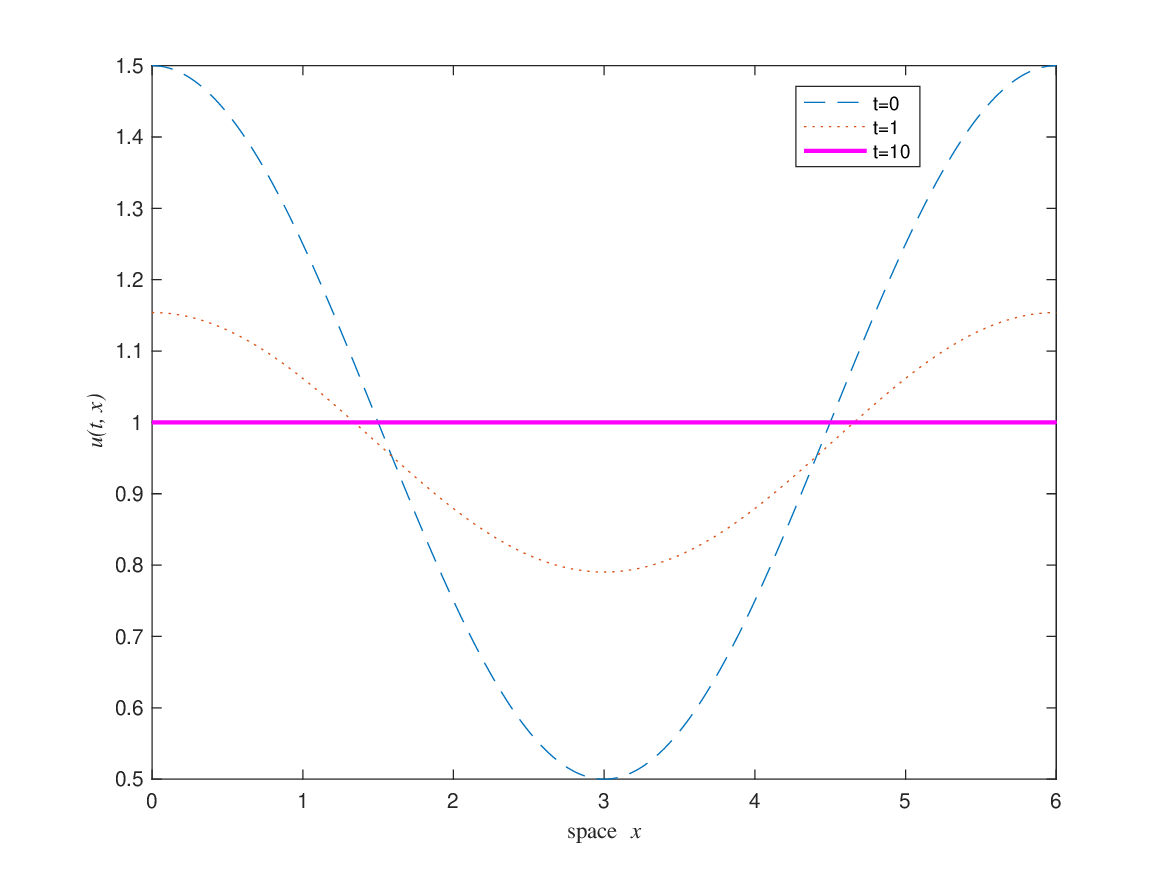}}
}
\subfigure[]{
\resizebox*{0.36\linewidth}{!}{\includegraphics{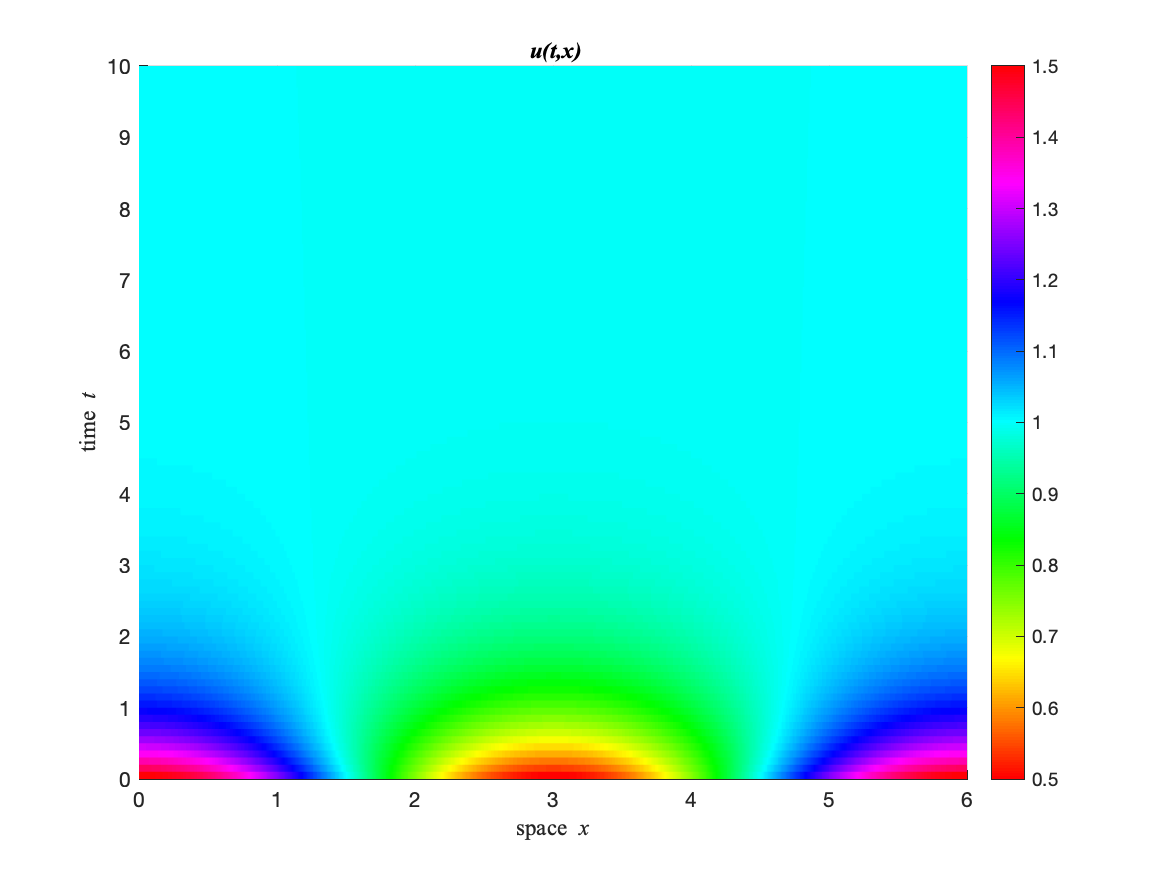}}\label{chi-2-L-6-u0-1+0.5cos0.3pix-u}}
\caption{(a) limit profile, (b)  evolution of $u(t,x;u_0)$  with $\chi=2$, $a=b=\mu=\nu=1$, $L=6$ and initial function $u_0=1+0.5\cos(\frac{\pi x}{3})$}
\label{chi-2}
\end{center}
\end{figure}

\begin{figure}[!ht]
\begin{center}
\subfigure[]{
\resizebox*{0.36\linewidth}{!}{\includegraphics{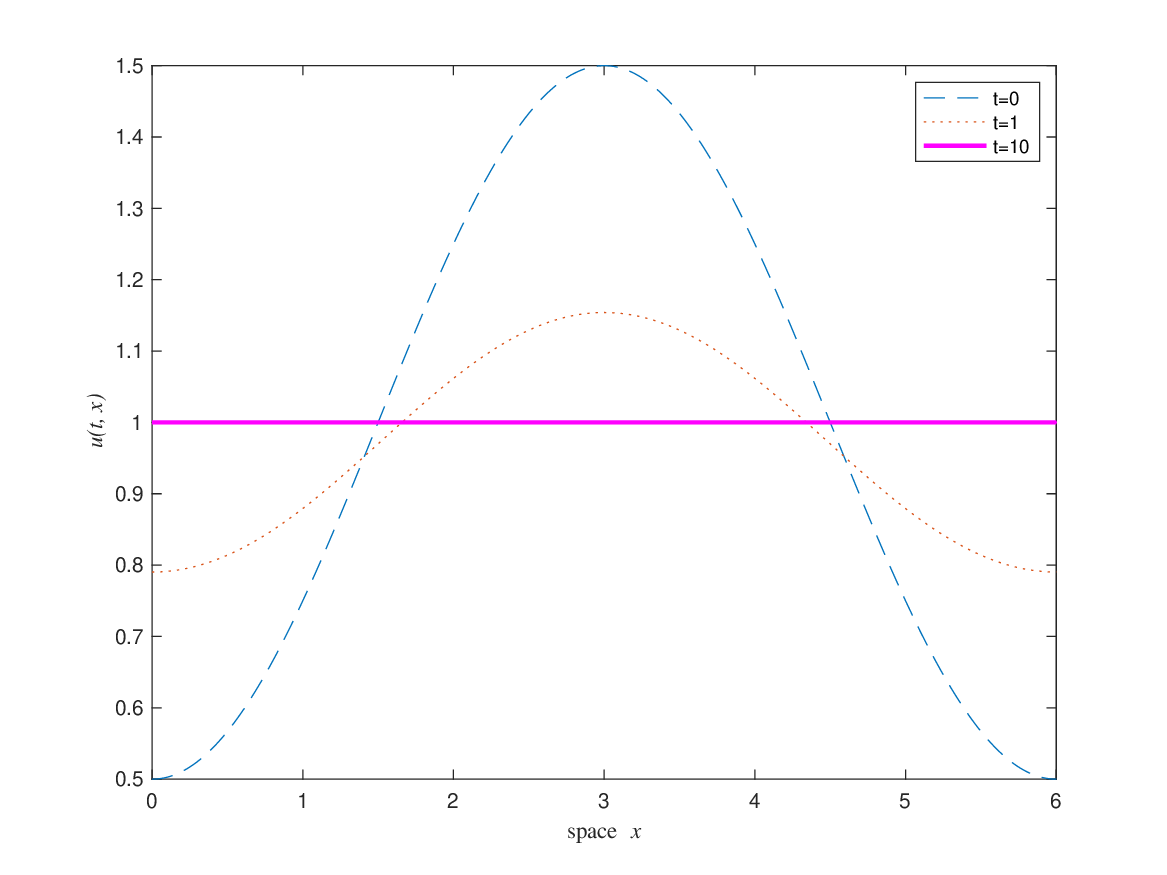}}
}
\subfigure[]{\resizebox*{0.36\linewidth}{!}{\includegraphics{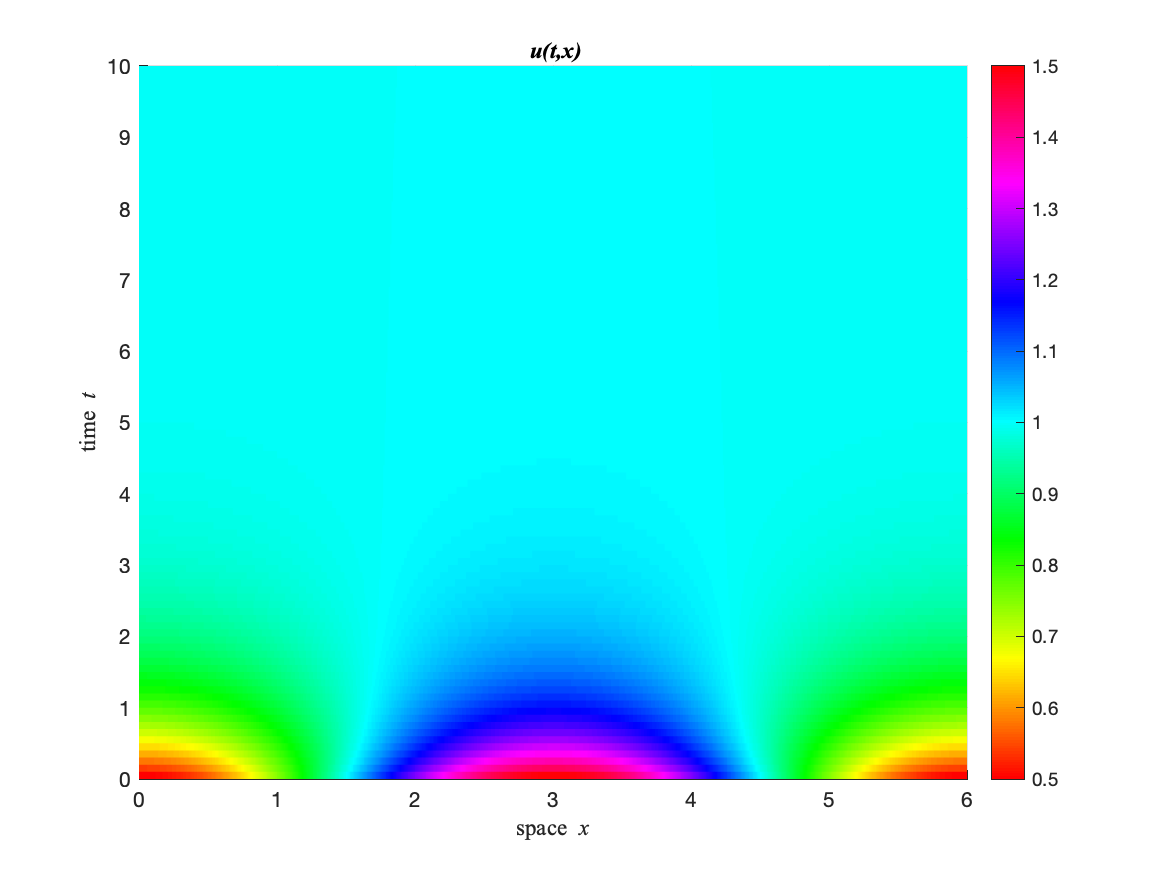}}\label{chi-2-L-6-u0-1-0.5cos0.3pix-u}}
\caption{(a) limit profile, (b) evolution of $u(t,x;u_0)$  with $\chi=2$, $a=b=\mu=\nu=1$, $L=6$ and  initial function $u_0=1-0.5\cos(\frac{\pi x}{3})$}
\label{chi-2-1}
\end{center}
\end{figure}

{Next, let $\chi=3.95$, which is slightly smaller than $\chi^*\approx 4.0085$.
Let $u_0=1\pm 0.1\cos(\frac{\pi x}{3})$. We observe that the numerical solution of $(u(t,x;u_0), v(t,x;u_0))$ changes very little when time is large enough and {converges} to the constant stationary solution $(1,1)$ (see Figures \ref{chi-3.95-1} and \ref{chi-3.95-2}).


\begin{figure}[!ht]
\begin{center}
\subfigure[]{
\resizebox*{0.36\linewidth}{!}{\includegraphics{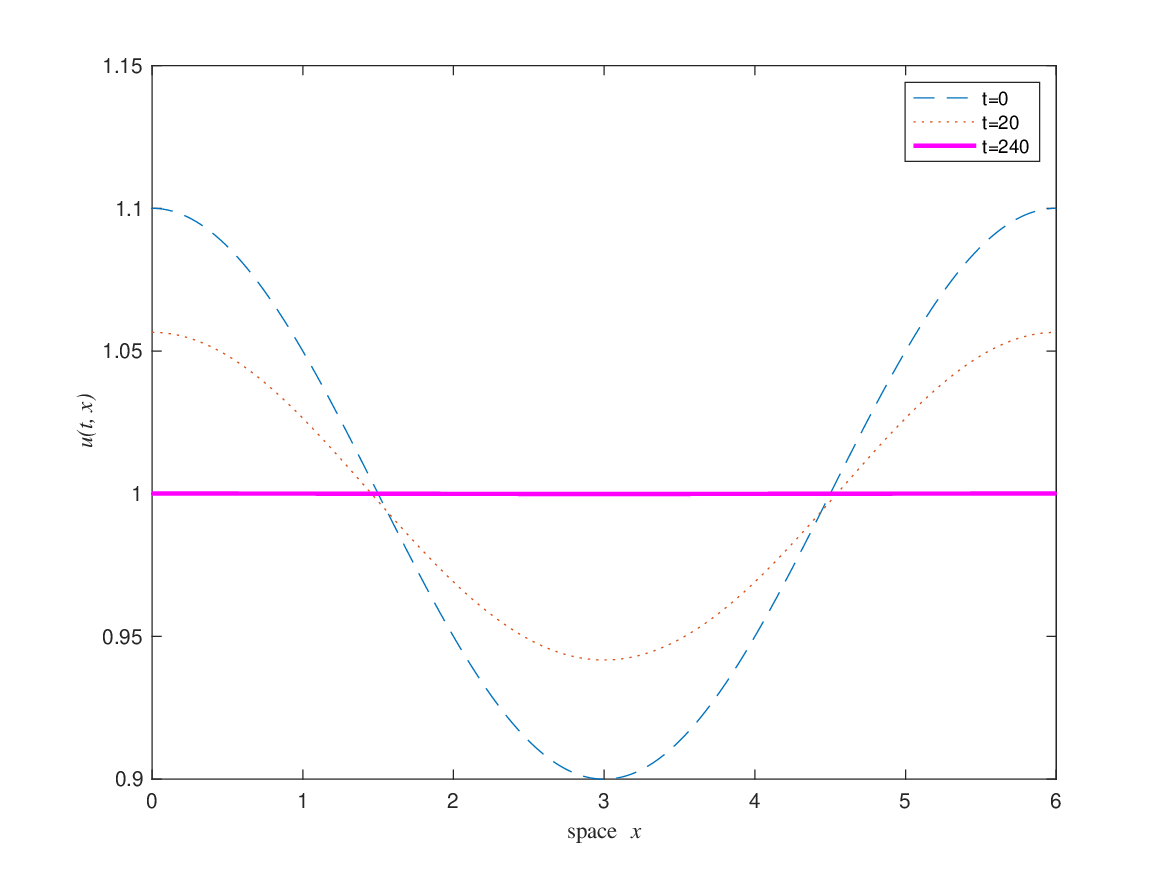}}
}
\subfigure[]{
\resizebox*{0.36\linewidth}{!}{\includegraphics{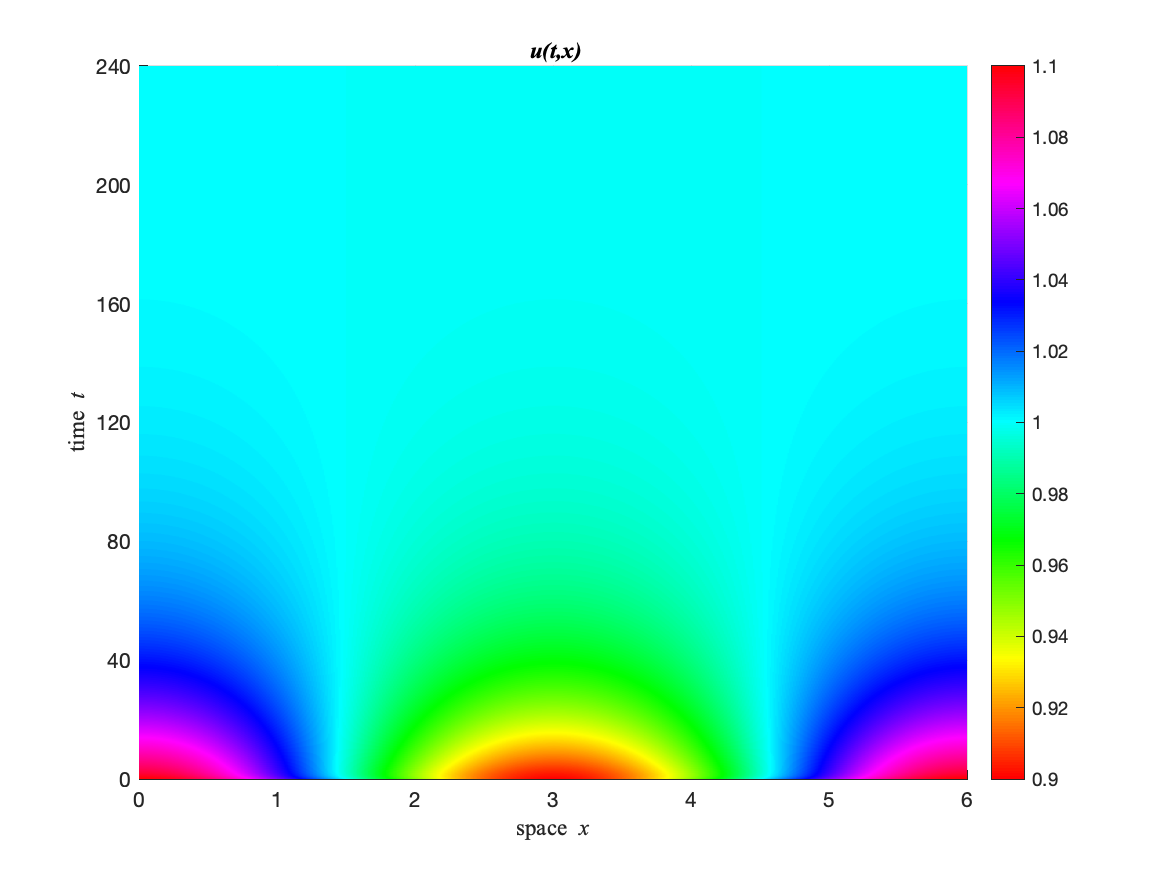}}
\label{chi-3.95-L-6-u0-1+0.1cos0.3pix-u}}
\caption{(a) limit profile, (b)  evolution of $u(t,x;u_0)$  with $\chi=3.95$,  $a=b=\mu=\nu=1$, $L=6$ and initial function $u_0=1+0.1\cos(\frac{\pi x}{3})$}
\label{chi-3.95-1}
\end{center}
\end{figure}

\begin{figure}[!ht]
\begin{center}
\subfigure[]{
\resizebox*{0.36\linewidth}{!}{\includegraphics{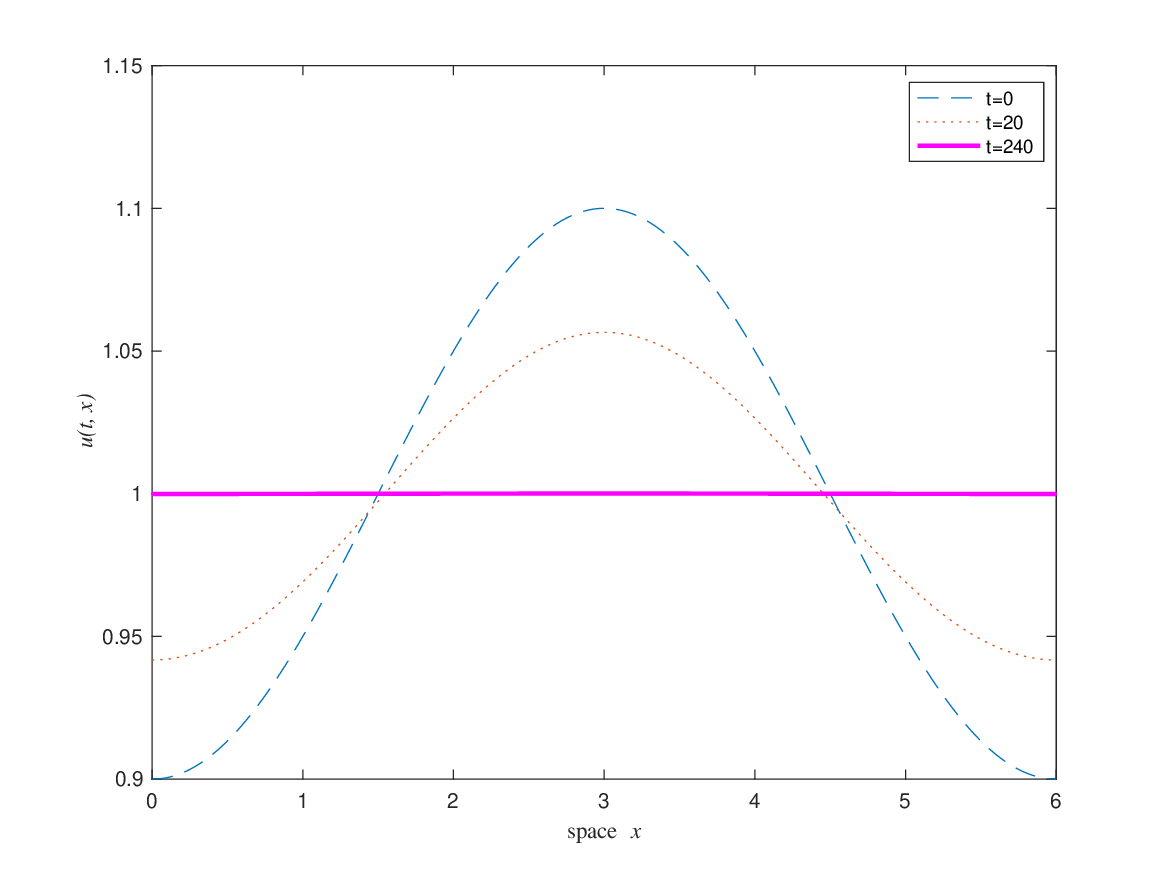}}
}
\subfigure[]{\resizebox*{0.36\linewidth}{!}{\includegraphics{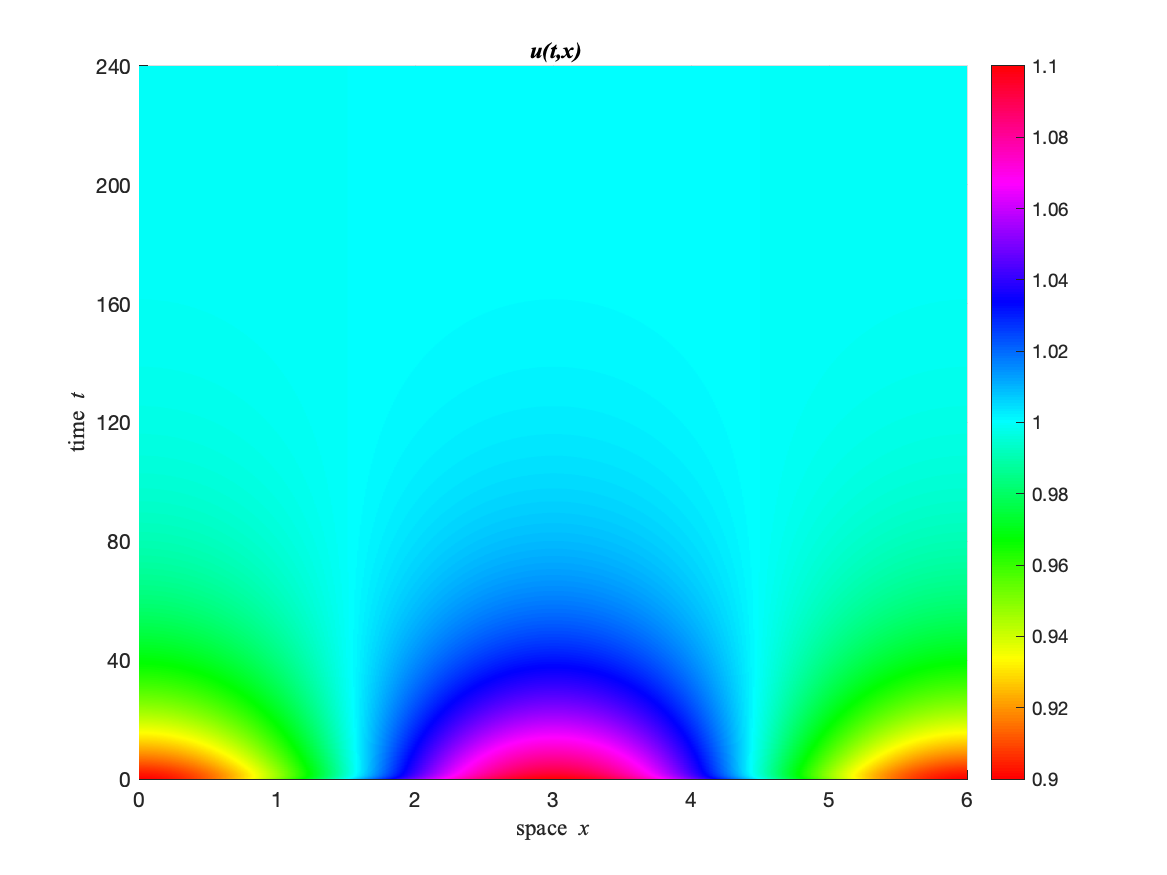}}\label{chi-3.95-L-6-u0-1-0.1cos0.3pix-u}}
\caption{(a) limit profile, (b)  evolution of $u(t,x;u_0)$  with $\chi=3.95$,  $a=b=\mu=\nu=1$, $L=6$ and initial function $u_0=1-0.1\cos(\frac{\pi x}{3})$}
\label{chi-3.95-2}
\end{center}
\end{figure}

For $\chi=3.95$, let  $u_0=1\pm 0.5\cos(\frac{\pi x}{3})$.  We observe that the numerical solution of $(u(t,x;u_0)$, $v(t,x;u_0))$ changes very little when time is large enough and {converges} to nonconstant stationary solutions, which are not so close to the constant stationary solution $(1,1)$ (see Figures \ref{chi-3.95-3} and \ref{chi-3.95-4}).

\begin{figure}[!ht]
\begin{center}
\subfigure[]{
\resizebox*{0.36\linewidth}{!}{\includegraphics{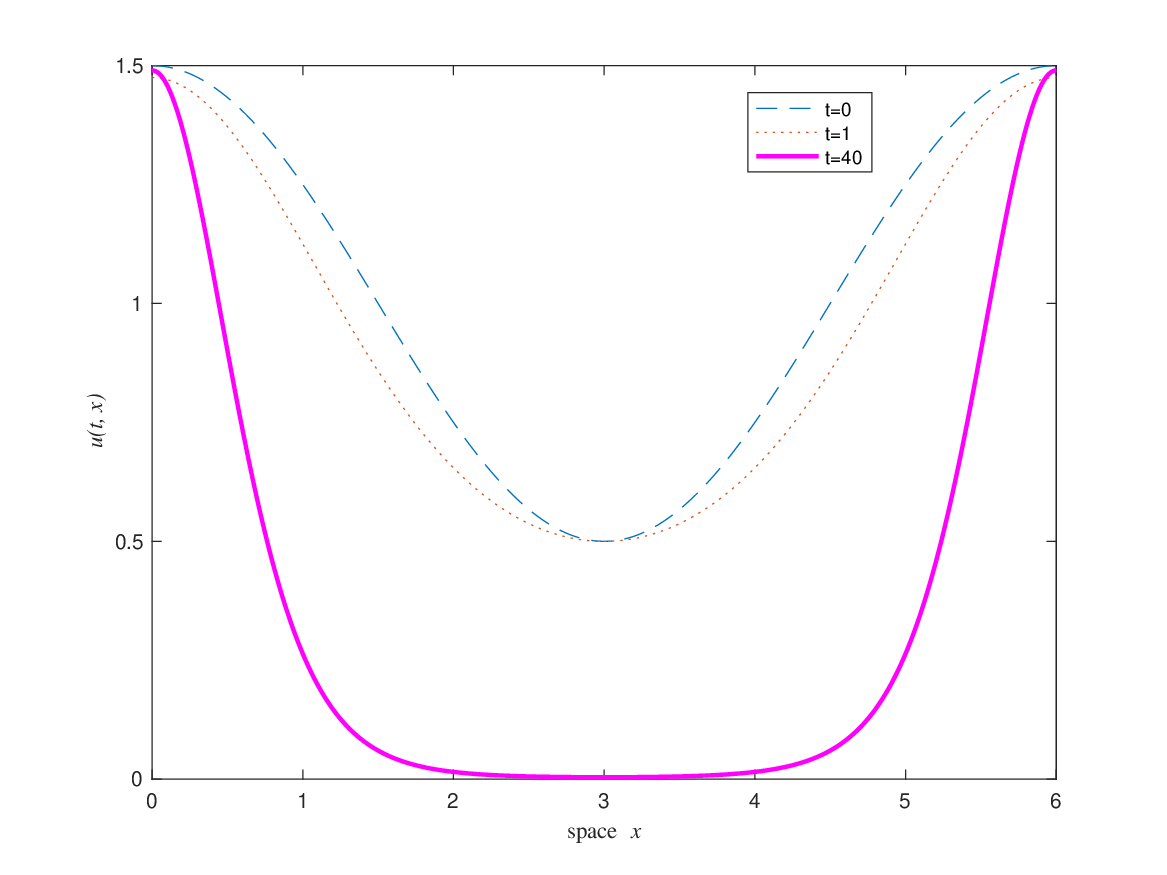}}
}
\subfigure[]{
\resizebox*{0.36\linewidth}{!}{\includegraphics{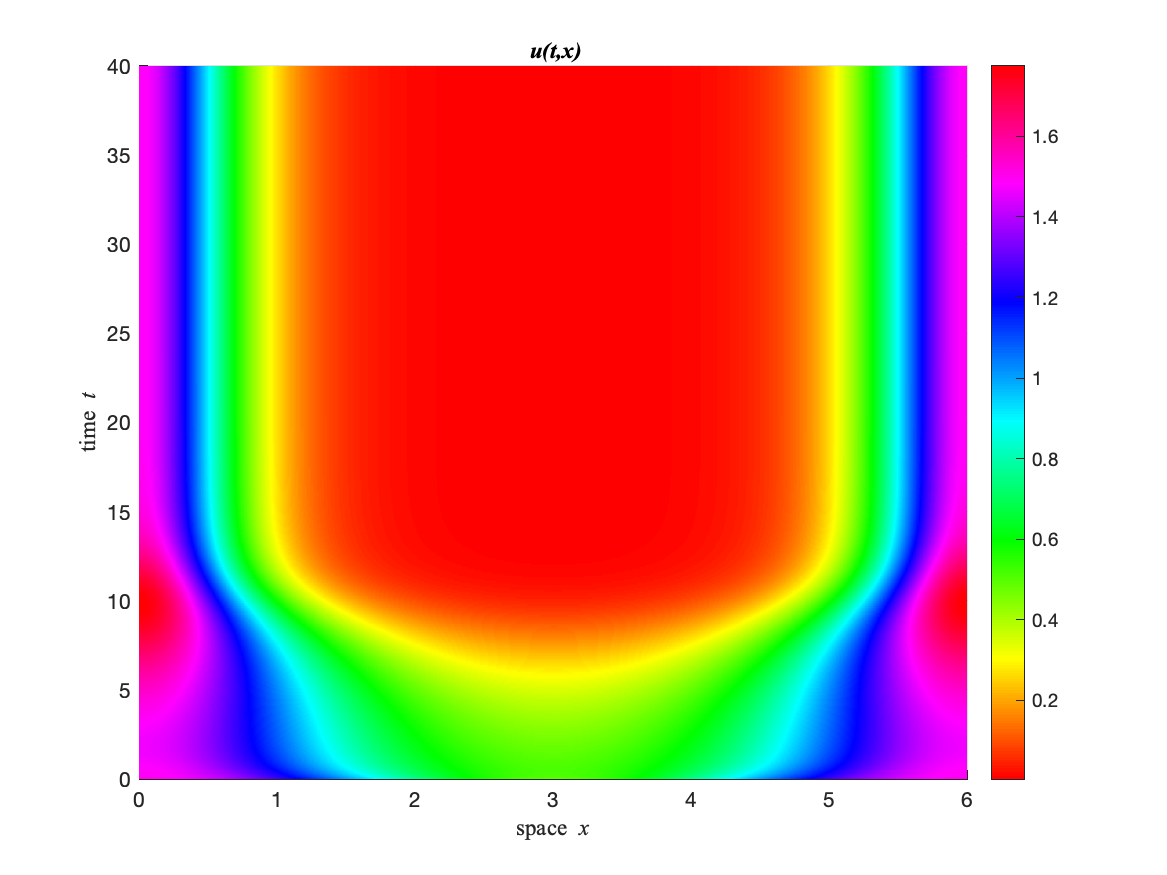}}\label{chi-3.95-L-6-u0-1+0.5cos0.3pix-u}}
\caption{(a) limit profile, (b) evolution of $u(t,x;u_0)$  with $\chi=3.95$, $a=b=\mu=\nu=1$,  $L=6$ and initial function $u_0=1+0.5\cos(\frac{\pi x}{3})$}
\label{chi-3.95-3}
\end{center}
\end{figure}

\begin{figure}[!ht]
\begin{center}
\subfigure[]{
\resizebox*{0.36\linewidth}{!}{\includegraphics{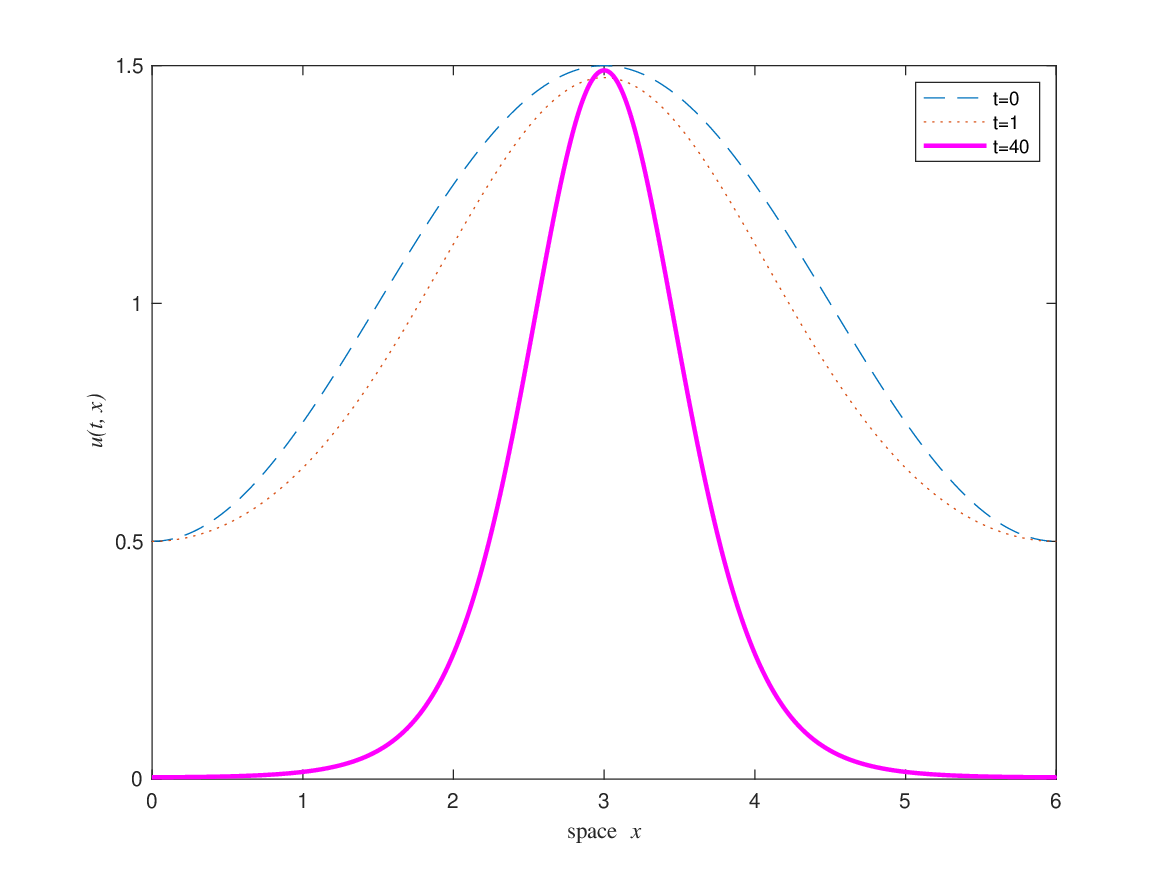}}
}
\subfigure[]{\resizebox*{0.36\linewidth}{!}{\includegraphics{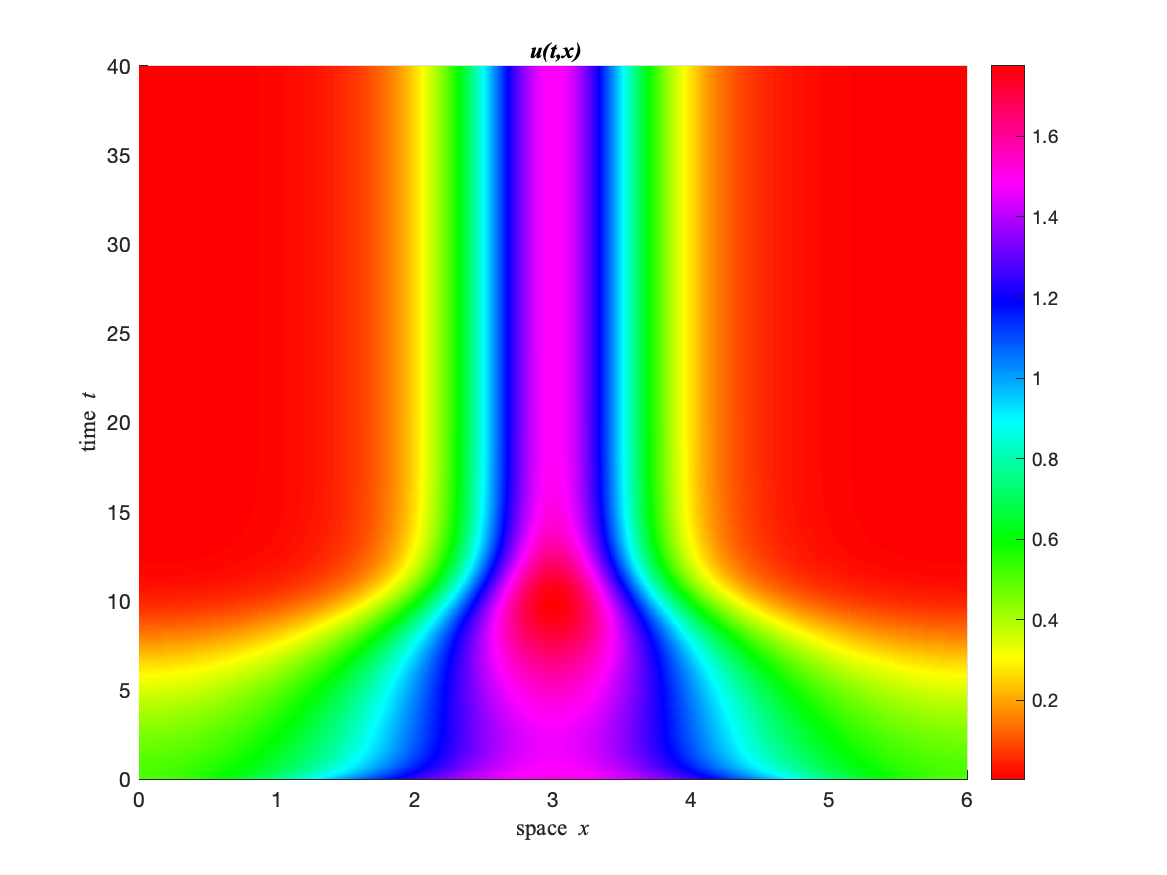}}\label{chi-3.95-L-6-u0-1-0.5cos0.3pix-u}}
\caption{(a) limit profile, (b)  evolution of $u(t,x;u_0)$  with $\chi=3.95$, $a=b=\mu=\nu=1$,  $L=6$ and initial function $u_0=1-0.5\cos(\frac{\pi x}{3})$}
\label{chi-3.95-4}
\end{center}
\end{figure}

\smallskip

\noindent{\bf  Observations from  Experiment 3.} It is known that when $0<\chi<\chi^*$, the constant solution $(1,1)$ is locally stable;
when $\chi>\chi^*$, $(1,1)$ is unstable, and sub-critical pitchfork bifurcation occurs when
$\chi$ passes through $\chi^*$,  that is, there are two unstable nonconstant stationary solutions bifurcating from $(1,1)$ for $\chi<\chi^*$ and $\chi$ is near $\chi^*$. It is observed from the experiment 3 that when $\chi<\chi^*$ and $\chi$ is near $\chi^*$,
the constant stationary solution $(1,1)$ is not globally stable and
there are other locally stable nonconstant stationary solutions, which are on the extension of
the local pitchfork bifurcation branch (see more in the observation of experiment 4 in the following).

\smallskip

\noindent {\bf Numerical Experiment 4.} In this experiment, we investigate the global bifurcation of the constant solution $(1,1)$.
First,  let $\chi=4.01$ which is slightly larger than the first bifurcation value.
Let $u_0=1\pm 0.5\cos(\frac{\pi x}{3})$.  As in the case that $\chi={3.95}$, we observe that the numerical solution of $(u(t,x;u_0)$, $v(t,x;u_0))$ changes very little when time is large enough and converge to nonconstant stationary solutions, which are not so close to the constant stationary solution $(1,1)$
 (see Figures \ref{chi-4.01} and \ref{chi-4.01-1}).

\begin{figure}[!ht]
\begin{center}
\subfigure[]{
\resizebox*{0.36\linewidth}{!}{\includegraphics{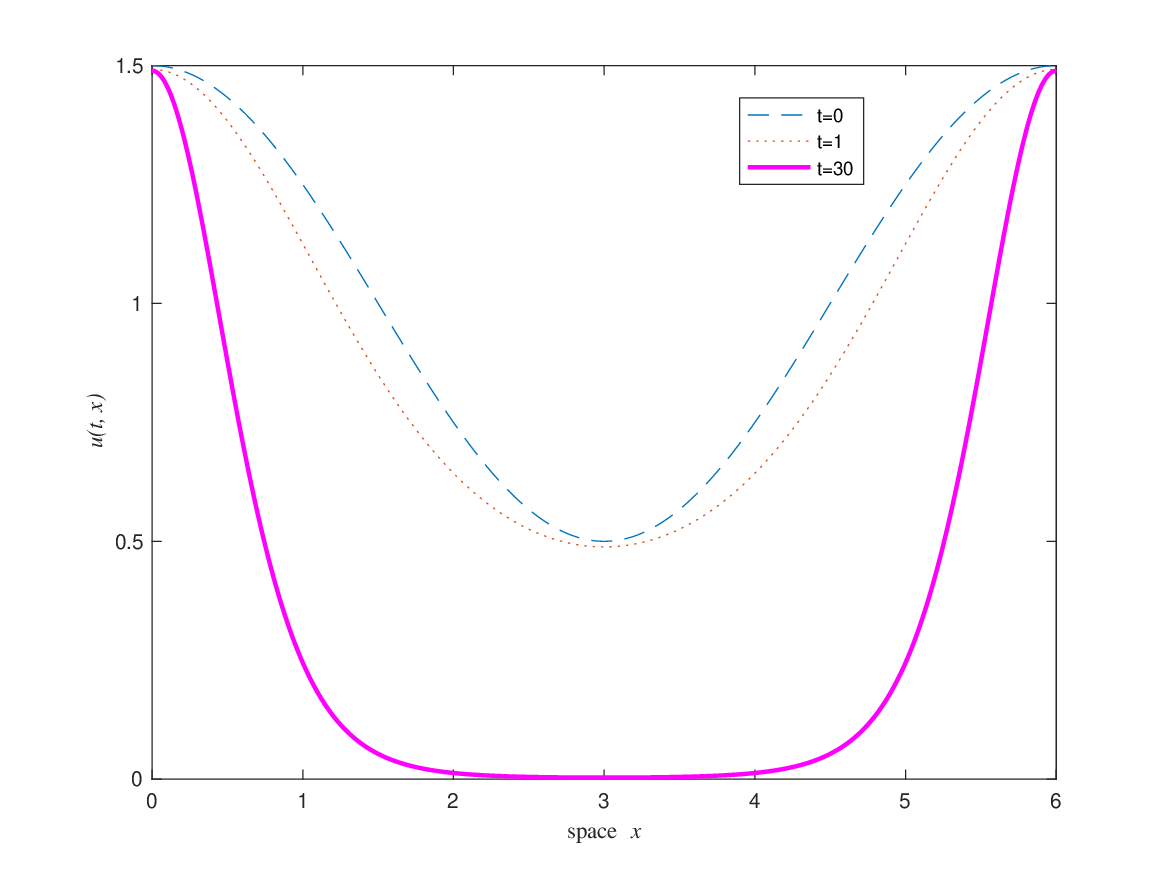}}
}
\subfigure[]{
\resizebox*{0.36\linewidth}{!}{\includegraphics{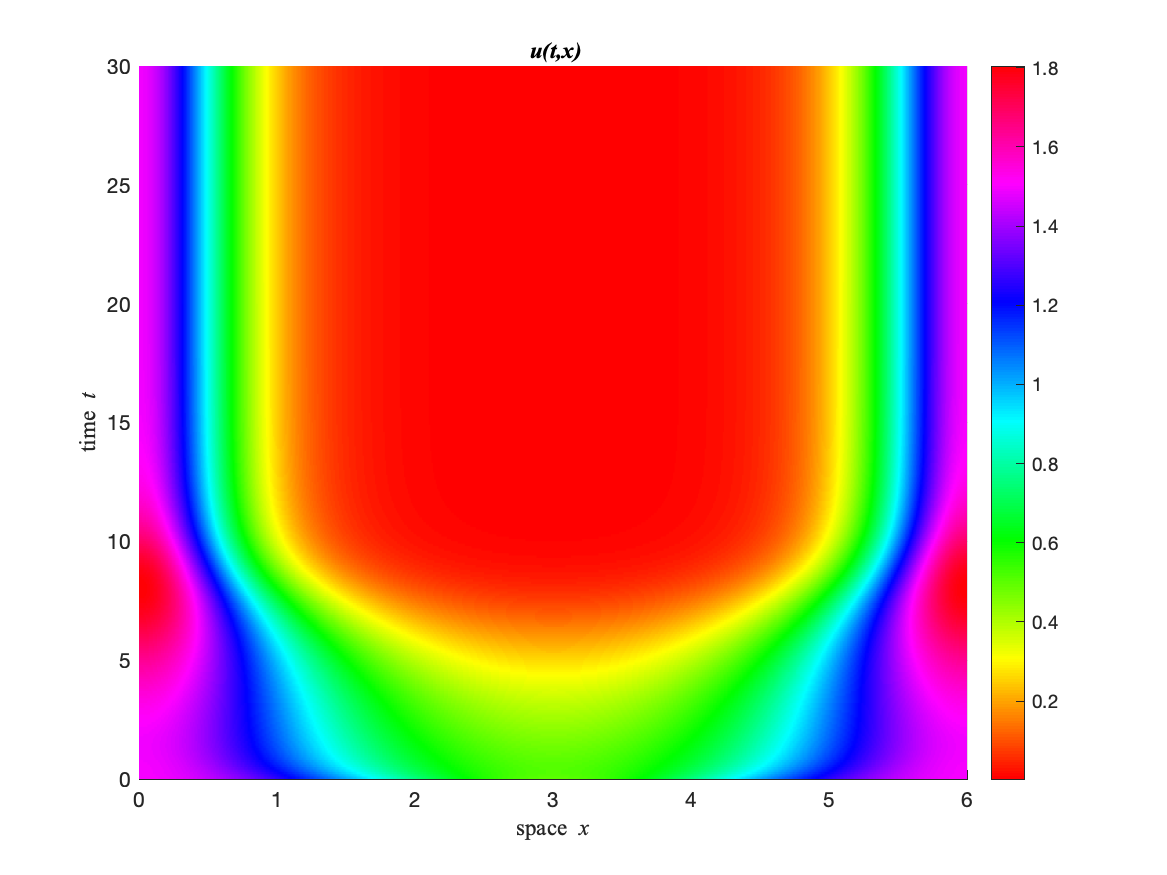}}\label{chi-4.01-L-6-u0-1+0.5cos0.3pix-u}}

\caption{(a) limit profile, (b) evolution of $u(t,x;u_0)$  with $\chi=4.01$, $a=b=\mu=\nu=1$, $L=6$ and initial function $u_0=1+0.5\cos(\frac{\pi x}{3})$
}
\label{chi-4.01}
\end{center}
\end{figure}

\begin{figure}[!ht]
\begin{center}
\subfigure[]{
\resizebox*{0.36\linewidth}{!}{\includegraphics{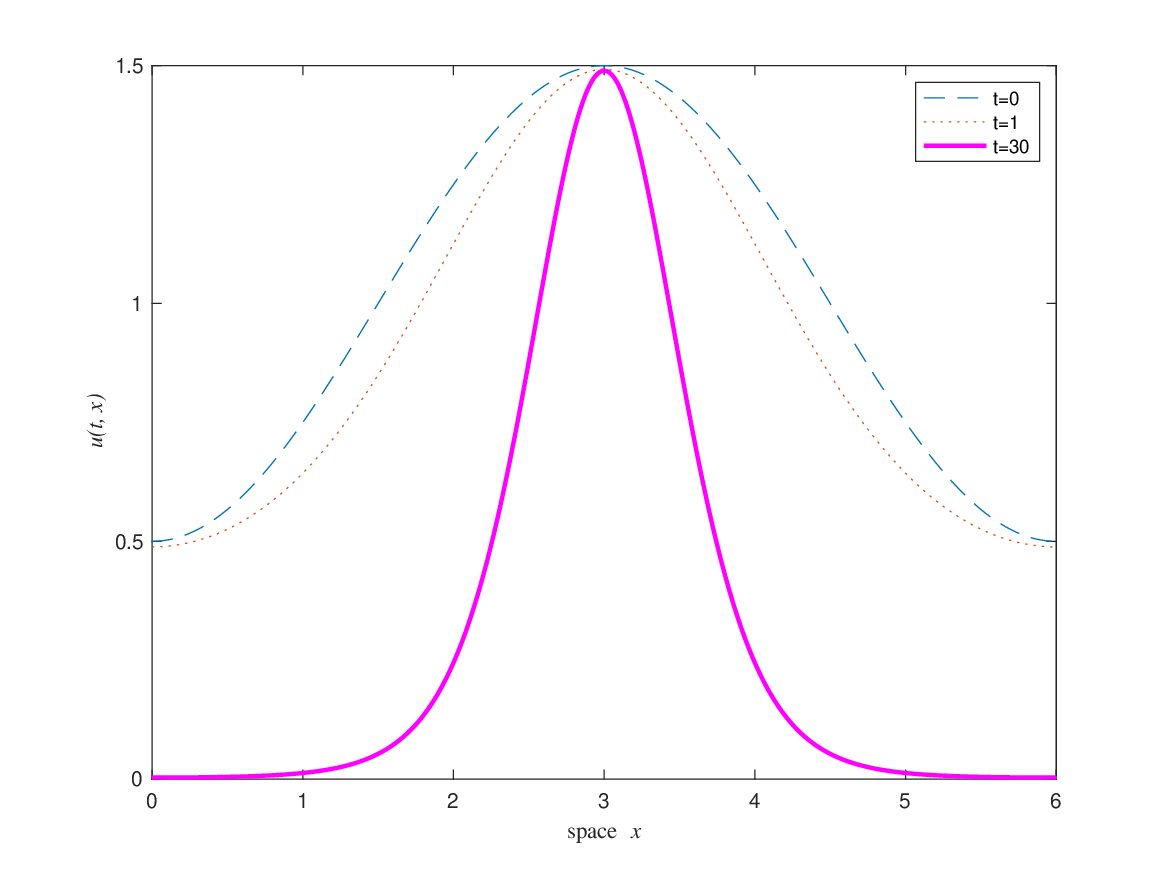}}
}
\subfigure[]{\resizebox*{0.36\linewidth}{!}{\includegraphics{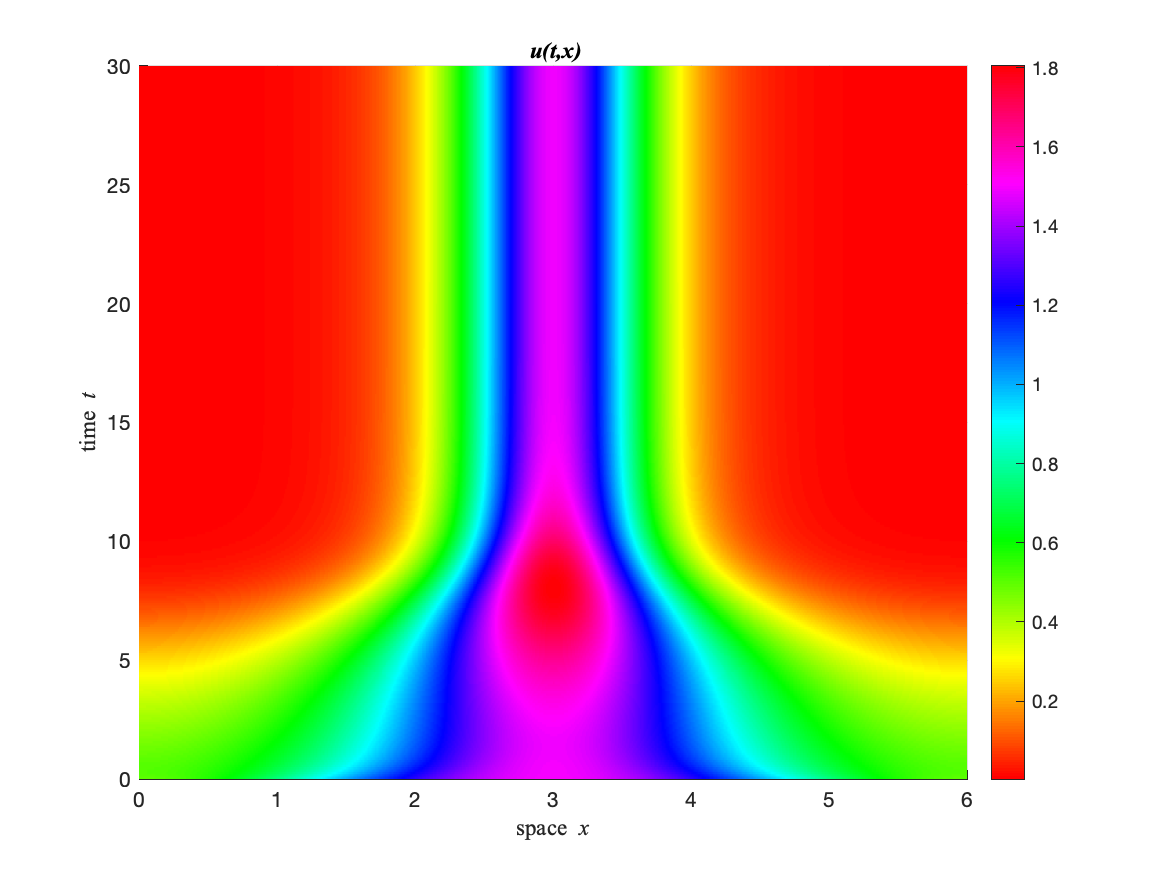}}\label{chi-4.01-L-6-u0-1-0.5cos0.3pix-u}}

\caption{(a) limit profile, (b)  evolution of $u(t,x;u_0)$  with $\chi=4.01$, $a=b=\mu=\nu=1$, $L=6$ and  initial function $u_0=1-0.5\cos(\frac{\pi x}{3})$
}
\label{chi-4.01-1}
\end{center}
\end{figure}


Next, we increase the value of $\chi$. Let $\chi=10$, which is not close to the bifurcation value $\chi^*$.
  Let $u_0=1+0.5\cos(\frac{\pi x}{3})$. All the other parameters remain the same. We observe that as time evolves, the numerical solution of $u(t,x;u_0)$ converges to a double boundary spike solution (see Figure \ref{chi-10-L-6}). The profile of the numerical solution at time $t=10$ can be viewed as the profile of the numerical double boundary spike solution. Let $u_0=1-0.5\cos(\frac{\pi x}{3})$. We observe that the numerical solution of $u(t,x;u_0)$ converges to a single interior spike solution (see Figure \ref{chi-10-L-6-1}).

\begin{figure}[!ht]
\begin{center}
\subfigure[]{
\resizebox*{0.36\linewidth}{!}{\includegraphics{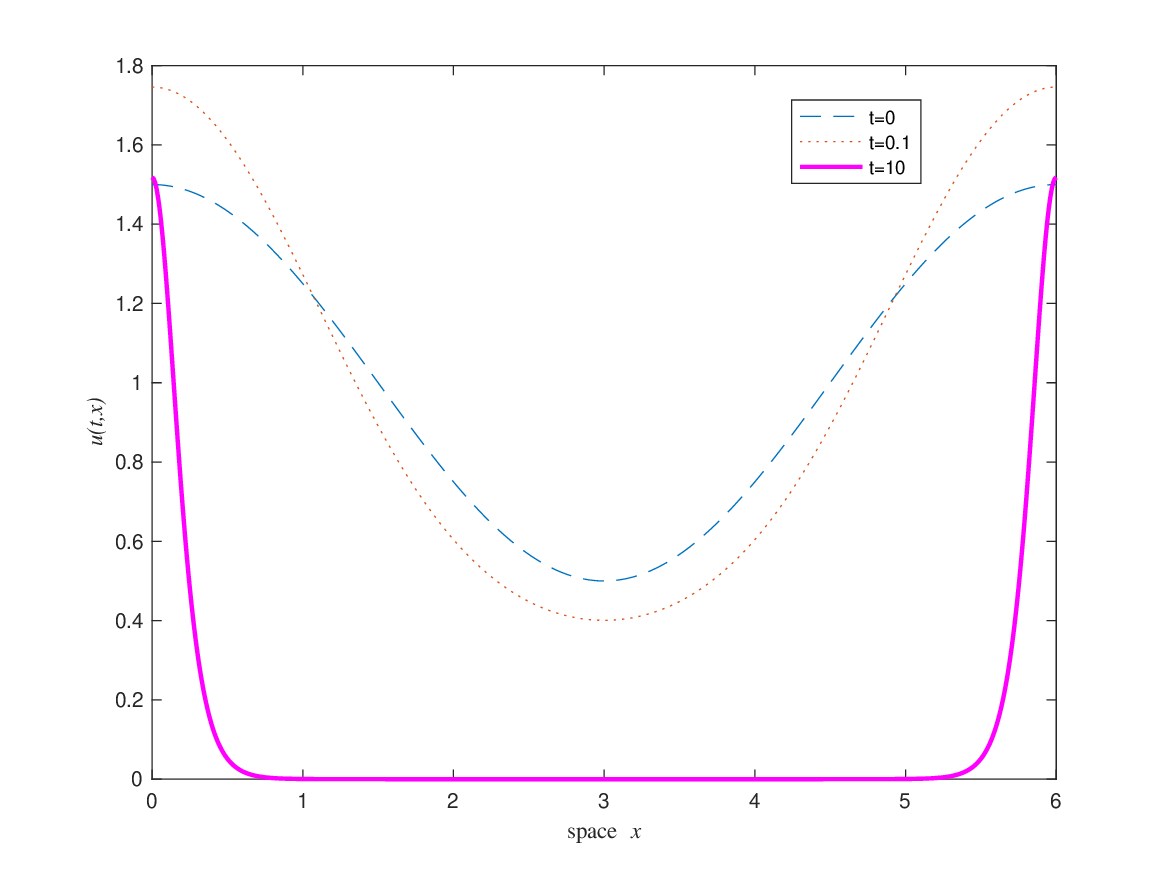}}
}
\subfigure[]{
\resizebox*{0.36\linewidth}{!}{\includegraphics{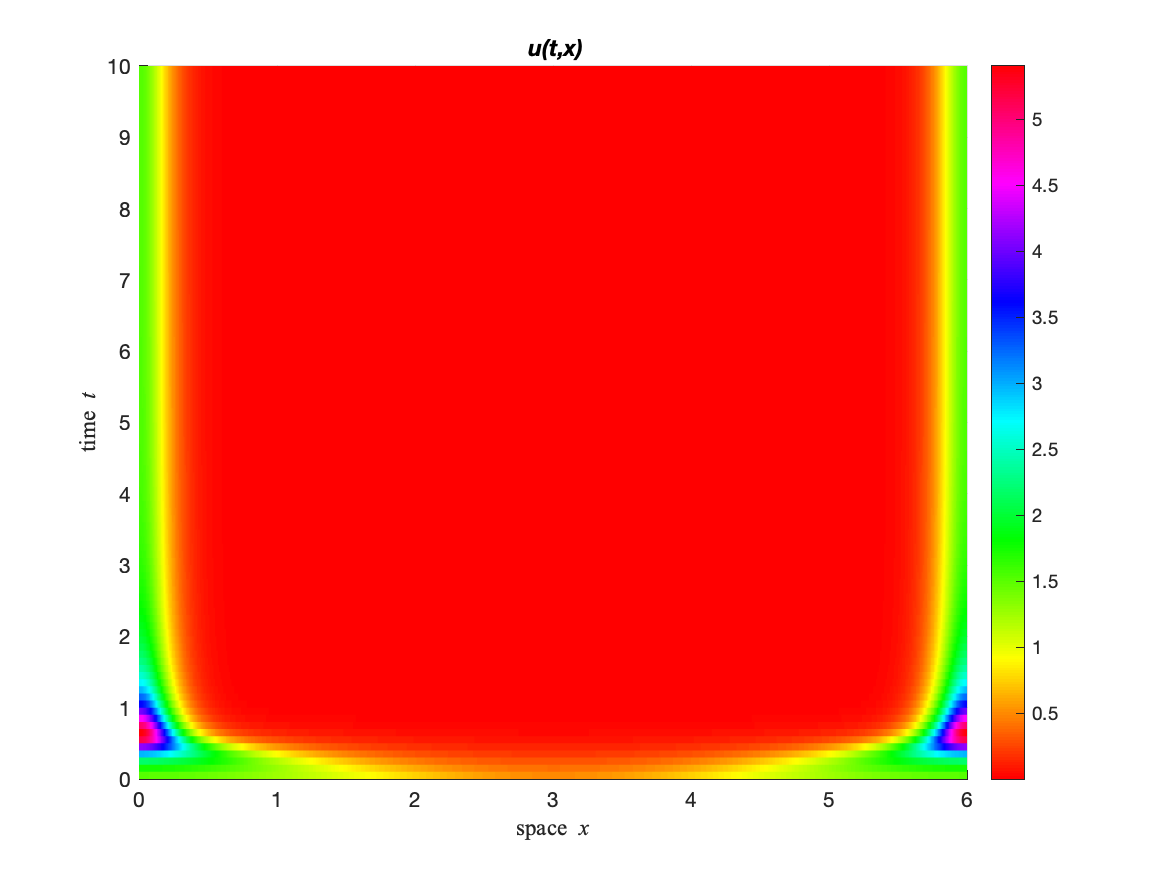}}\label{chi-10-L-6-u0-1+0.5cos0.3pix-u}}
\caption{(a) limit profile, (b)  evolution of $u(t,x;u_0)$  with $\chi=10$, $a=b=\mu=\nu=1$, $L=6$ and initial function $u_0=1+0.5\cos(\frac{\pi x}{3})$}
\label{chi-10-L-6}
\end{center}
\end{figure}

\begin{figure}[!ht]
\begin{center}
\subfigure[]{
\resizebox*{0.36\linewidth}{!}{\includegraphics{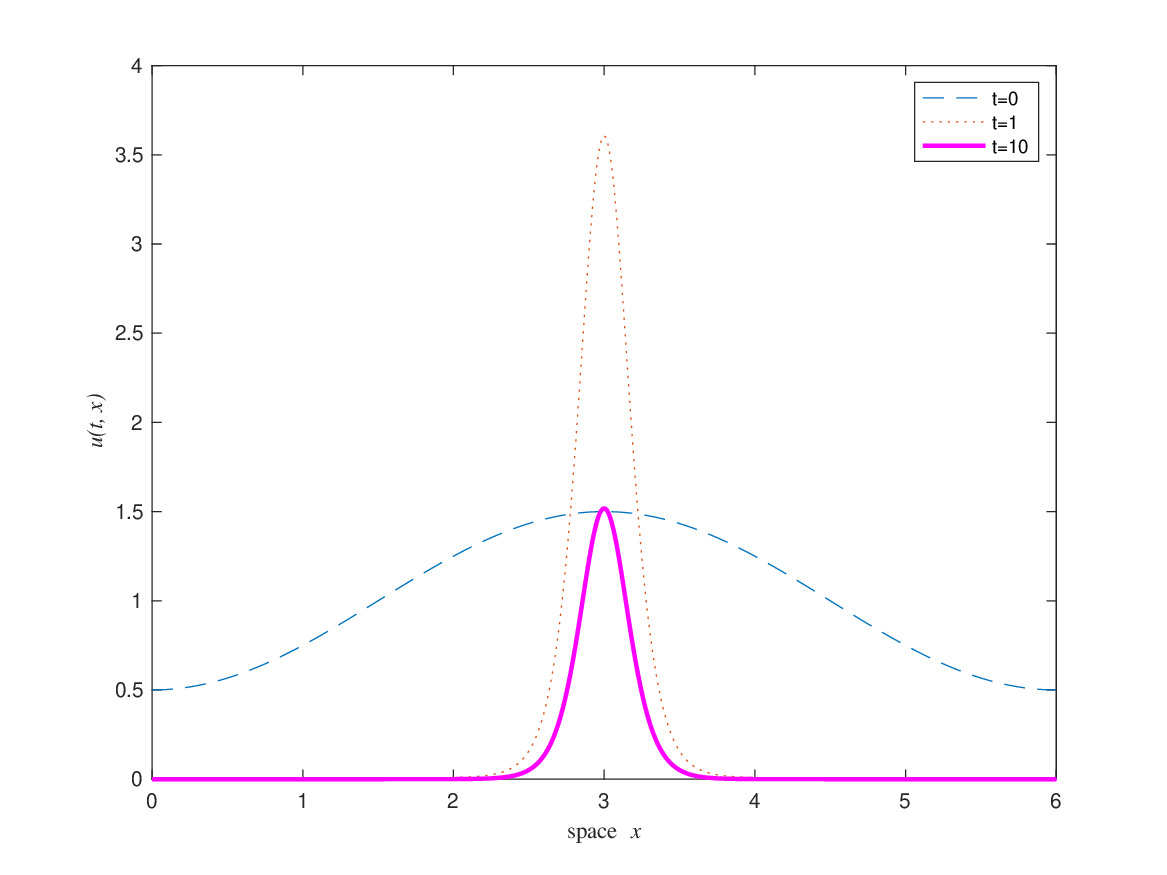}}
}
\subfigure[]{
\resizebox*{0.36\linewidth}{!}{\includegraphics{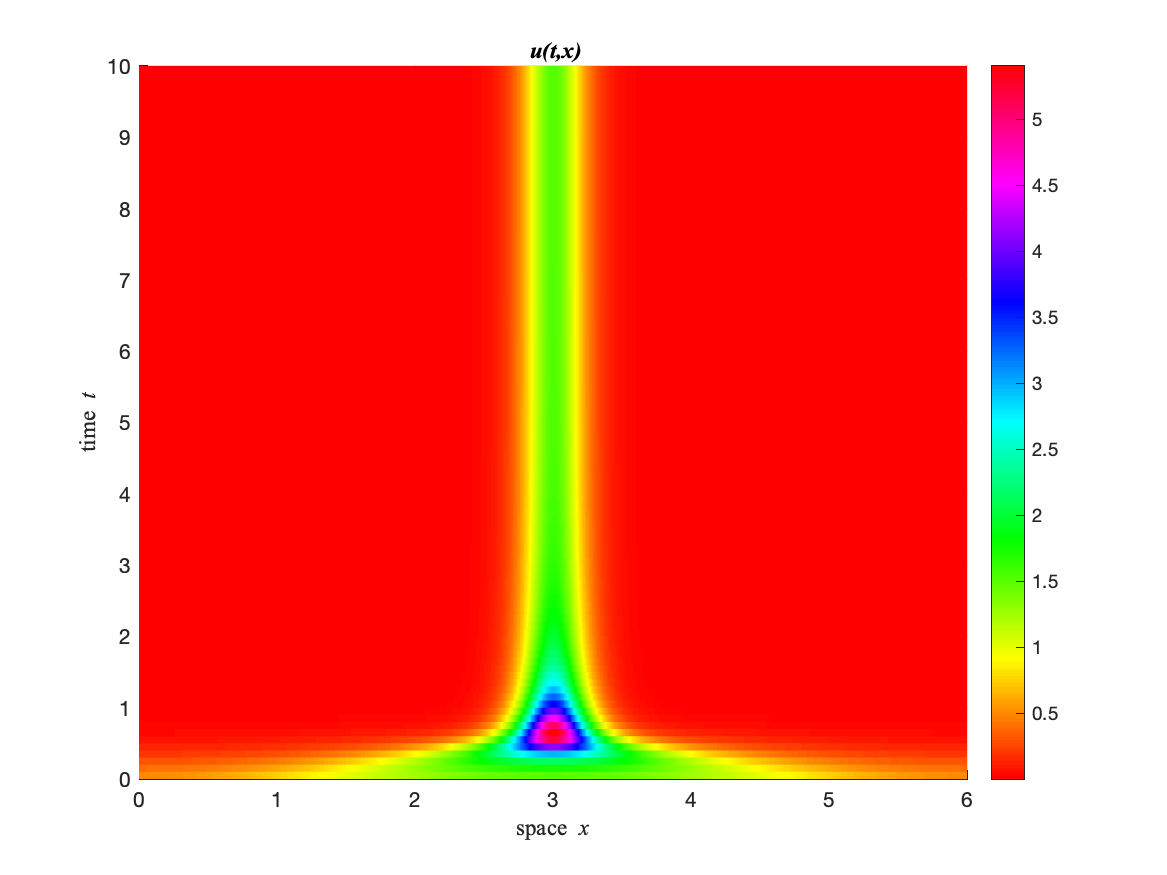}}\label{chi-10-L-6-u0-1-0.5cos0.3pix-u}}
\caption{(a) limit profile, (b) evolution of $u(t,x;u_0)$   with $\chi=10$, $a=b=\mu=\nu=1$, $L=6$ and initial function  $u_0=1-0.5\cos(\frac{\pi x}{3})$}
\label{chi-10-L-6-1}
\end{center}
\end{figure}

We further increase the value of $\chi$.  Let  $\chi=20$ and take $u_0=1\pm 0.5\cos(\frac{\pi x}{3})$. The same phenomenon as in the case $\chi=10$ is observed (see Figures \ref{chi-20-L-6} and \ref{chi-20-L-6-1}). In particular, we observe that when $\chi$ becomes larger, the numerical nonconstant stationary solution concentrates more towards both boundaries when $u_0=1+0.5\cos(\frac{\pi x}{3})$ and towards the middle of the location when $u_0=1-0.5\cos(\frac{\pi x}{3})$.

\begin{figure}[!ht]
\begin{center}
\subfigure[]{
\resizebox*{0.36\linewidth}{!}{\includegraphics{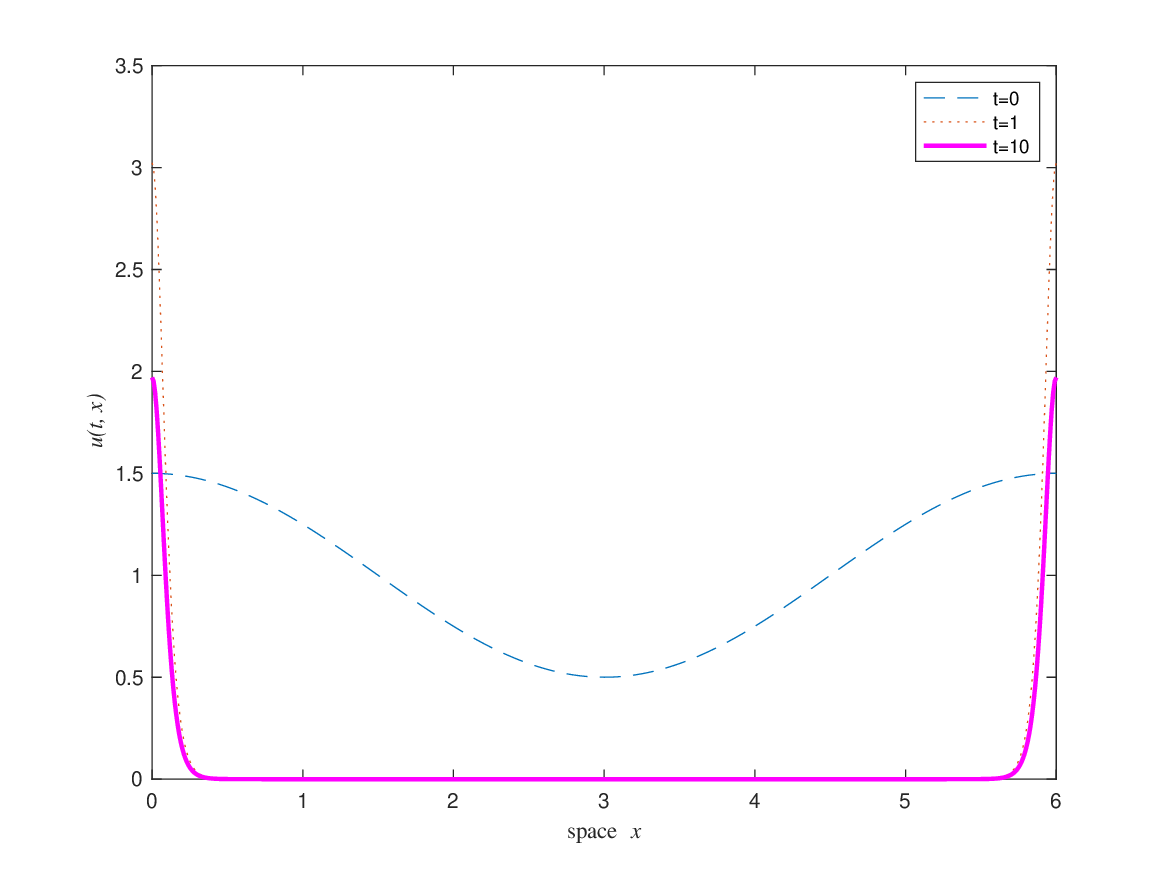}}
}
\subfigure[]{
\resizebox*{0.36\linewidth}{!}{\includegraphics{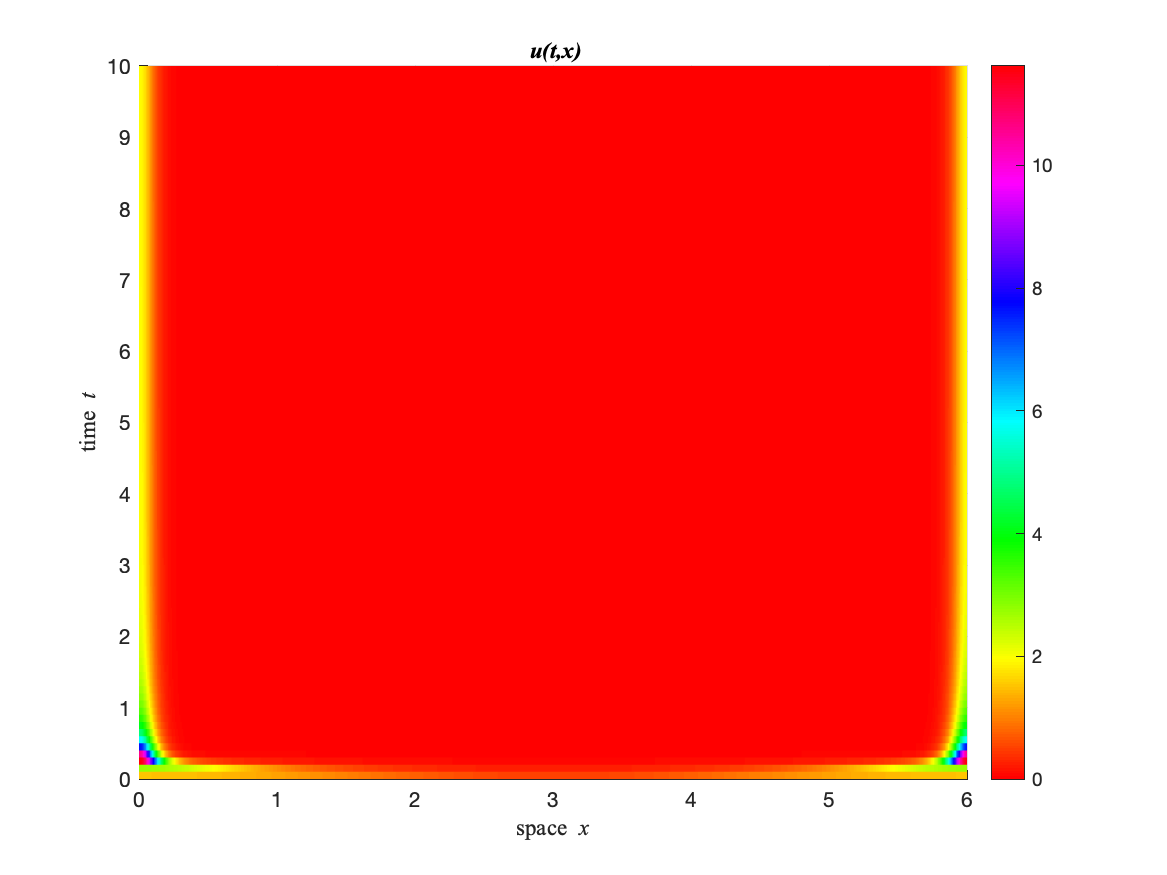}}\label{chi-20-L-6-u0-1+0.5cos0.3pix-u}}
\caption{(a) limit profile, (b)  evolution of $u(t,x;u_0)$   with $\chi=20$, $a=b=\mu=\nu=1$, $L=6$ and initial function $u_0=1+0.5\cos(\frac{\pi x}{3})$}
\label{chi-20-L-6}
\end{center}
\end{figure}

\begin{figure}[!ht]
\begin{center}
\subfigure[]{
\resizebox*{0.36\linewidth}{!}{\includegraphics{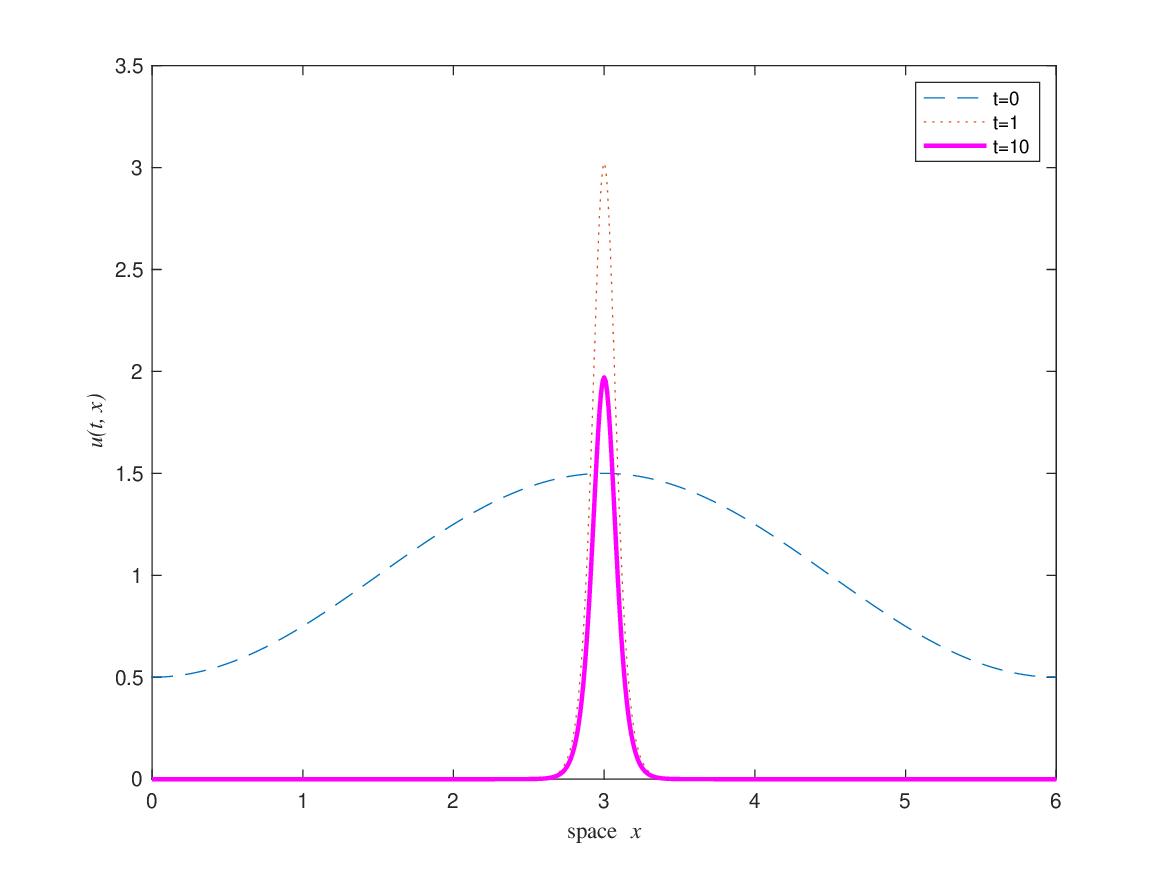}}
}
\subfigure[]{
\resizebox*{0.36\linewidth}{!}{\includegraphics{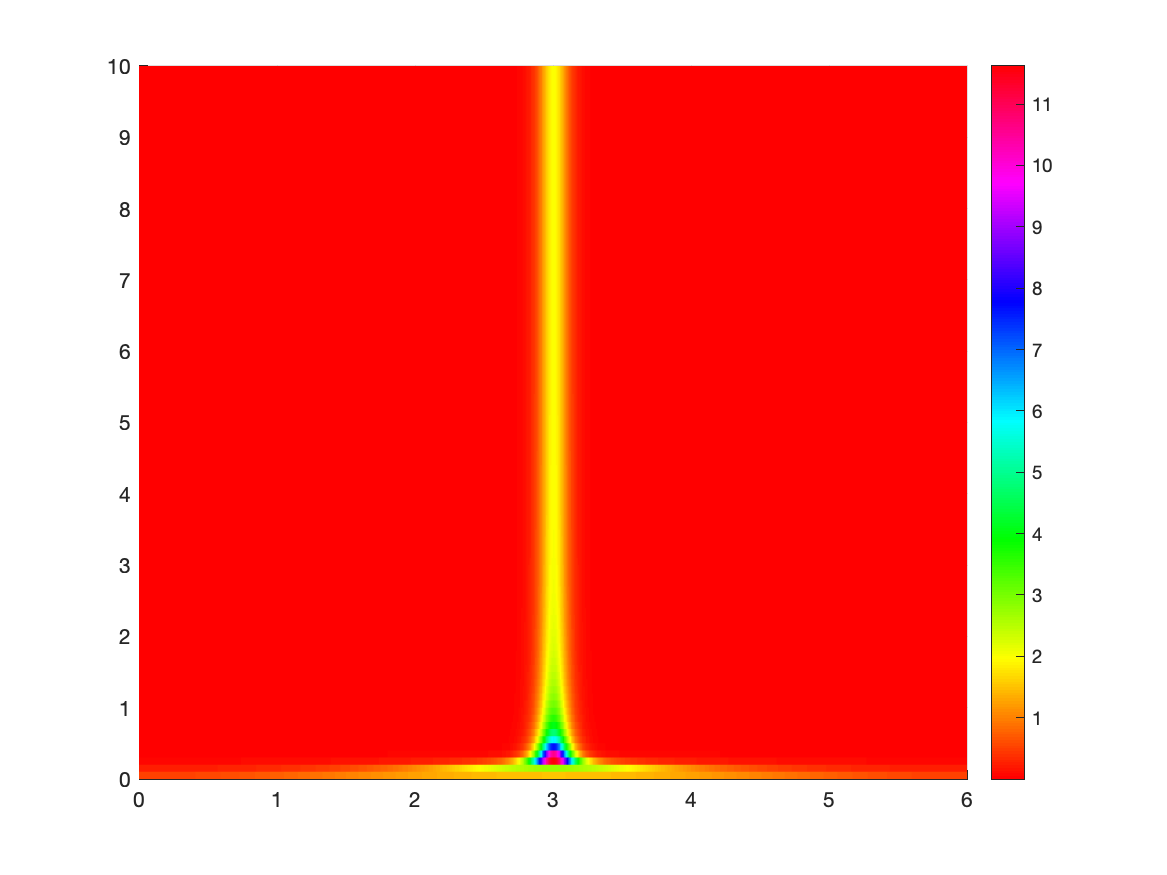}}\label{chi-20-L-6-u0-1-0.5cos0.3pix-u}}
\caption{(a) limit profile, (b) evolution of $u(t,x;u_0)$  with $\chi=20$, $a=b=\mu=\nu=1$, $L=6$ and initial function  $u_0=1-0.5\cos(\frac{\pi x}{3})$}
\label{chi-20-L-6-1}
\end{center}
\end{figure}

\noindent {\bf Observation from experiment 4.} As it is mentioned in the observation of experiment 3,   sub-critical pitchfork bifurcation occurs when
$\chi$ passes through $\chi^*$,  that is, there are two unstable nonconstant stationary solutions bifurcating from $(1,1)$ for $\chi<\chi^*$ and $\chi$ is near $\chi^*$. It is observed from the experiment 4  that
the local pitchfork bifurcation branch extends to $\chi=\infty$ and the bifurcation solutions
on the extended branch are locally stable when $\chi>\chi^*$. Moreover, when $\chi$ increases, the $u$-component of the bifurcation solutions either develops spikes
near both boundary points $x=0$ and $x=1$, or develops an interior spike.

\subsection{Numerical simulations for the case $a=2$, $b=\nu=1$, $\mu=3$, $L=6$}
\label{a-2-l-6-sec}


In this subsection, we discuss the numerical simulations we carried out for the case that
$a=2$, $b=\nu=1$, $\mu=3$,  and $L=6$. In this case, it is known that $\chi^*=\chi_3^*\approx 3.2997$;
$(\frac{a}{b},\frac{\nu}{\mu}\frac{a}{b})=(2, \frac{2}{3})$ is locally  asymptotically stable when
$0<\chi<\chi^*$; and when $\chi$ passes through $\chi^*$,
subcritical pitchfork bifurcation occurs.
When $\chi\approx \chi^*$, similar dynamical scenarios  as in the case $a=b=\mu=\nu=1$ and $L=6$ are observed numerically. To be more precise,
 it is observed numerically  that when $\chi<\chi^*$ and $\chi$ is near $\chi^*$,
the constant stationary solution $(2,\frac{2}{3})$ is not globally stable and
there are other locally stable nonconstant stationary solutions, which are on the extension of
the local pitchfork bifurcation branch.  Moreover, in this case, it is  observed numerically that, when $\chi$ increases, the $u$-component of the bifurcation solution develops both boundary and interior spikes.

For example,   let $\chi=3.3$ which is slightly larger than {the bifurcation value $\chi^*$}.
Let $u_0=2\pm 0.5\cos(\frac{\pi x}{2})$.  We observe that for each initial function, the numerical solution of $(u(t,x;u_0)$, $v(t,x;u_0))$ changes very little when time is large enough and converges to nonconstant stationary solution, which is not so close to the constant stationary solution $(2,\frac{2}{3})$
 (see Figure \ref{chi-3.3-L-6-u0-2+0.5cos0.5pix}).

\begin{figure}[!ht]
\begin{center}
\subfigure[]{
\resizebox*{0.36\linewidth}{!}{\includegraphics{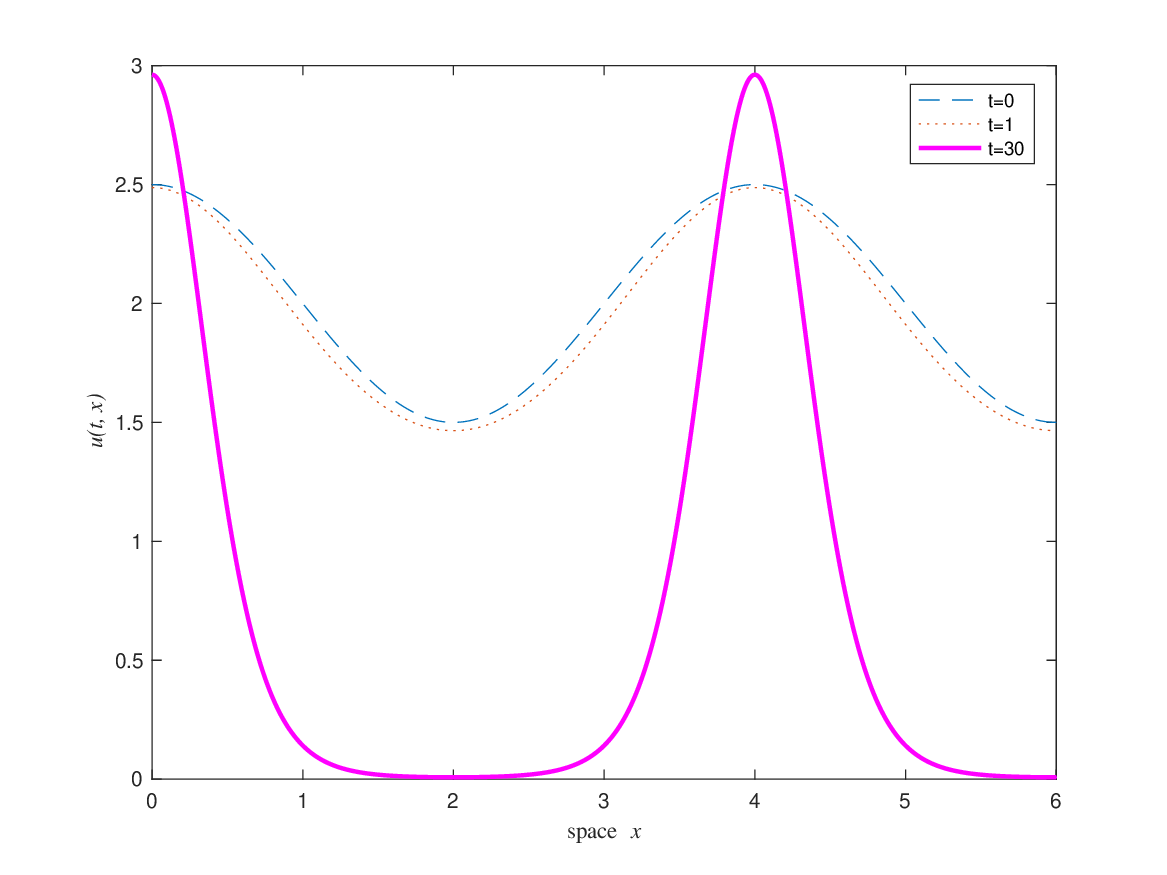}}
}
\subfigure[]{
\resizebox*{0.36\linewidth}{!}{\includegraphics{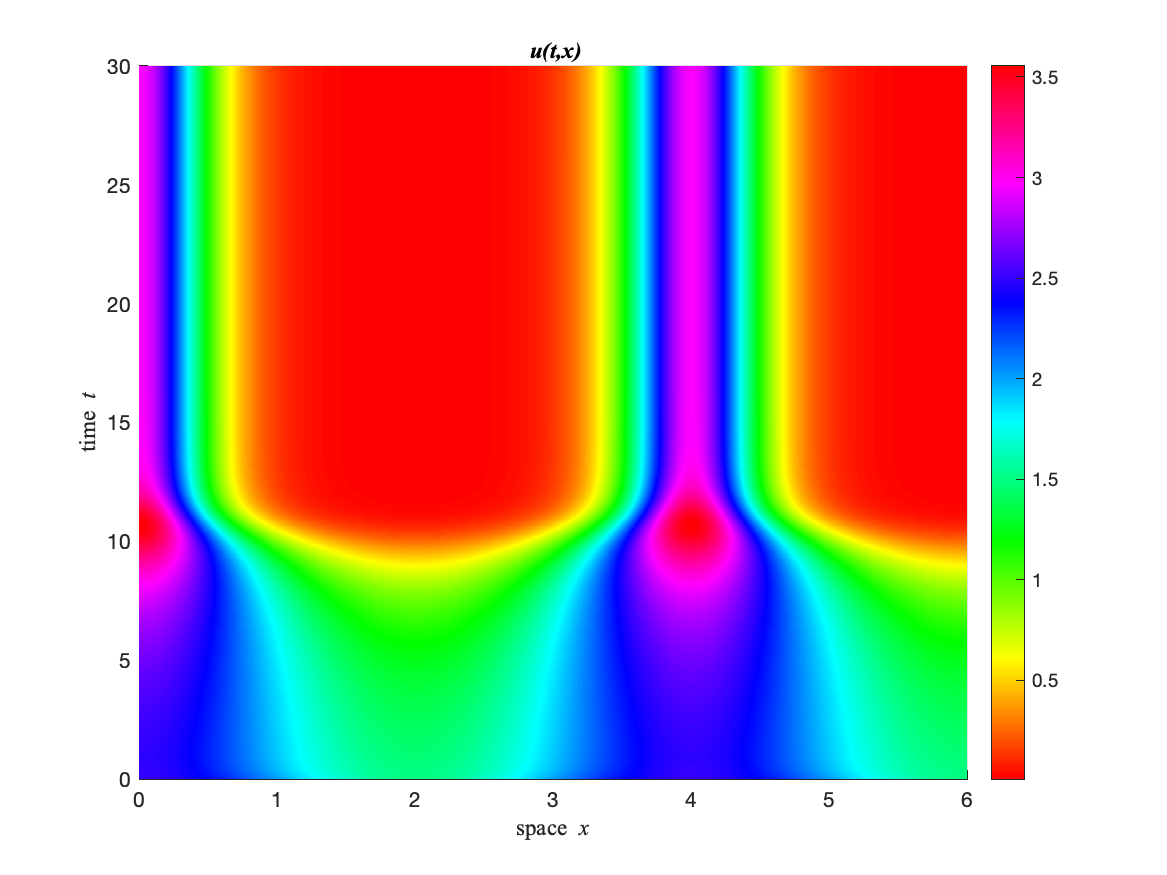}}\label{chi-10-L-6-u0-2+0.5cos0.5pix-u}}
\caption{(a) limit profile, (b) evolution of $u(t,x;u_0)$ with $\chi=3.3$,  $a=2$, $b=\nu=1$, $\mu=3$, $L=6$ and initial function $u_0=2+0.5\cos(\frac{\pi x}{2})$}
\label{chi-3.3-L-6-u0-2+0.5cos0.5pix}
\end{center}
\end{figure}

 Let $\chi=10$, which is not close to the bifurcation value $\chi^*$.
  Let $u_0=2+0.5\cos(\frac{\pi x}{2})$. All the other parameters remain the same. We observe that as time evolves, the numerical solution of $u(t,x;u_0)$ converges to a spike solution which has boundary spike at $x=0$ and interior spike at $x=4$ (see
Figure \ref{chi-10-L-6-u0-2+0.5cos0.5pix}).

\begin{figure}[!ht]
\begin{center}
\subfigure[]{
\resizebox*{0.36\linewidth}{!}{\includegraphics{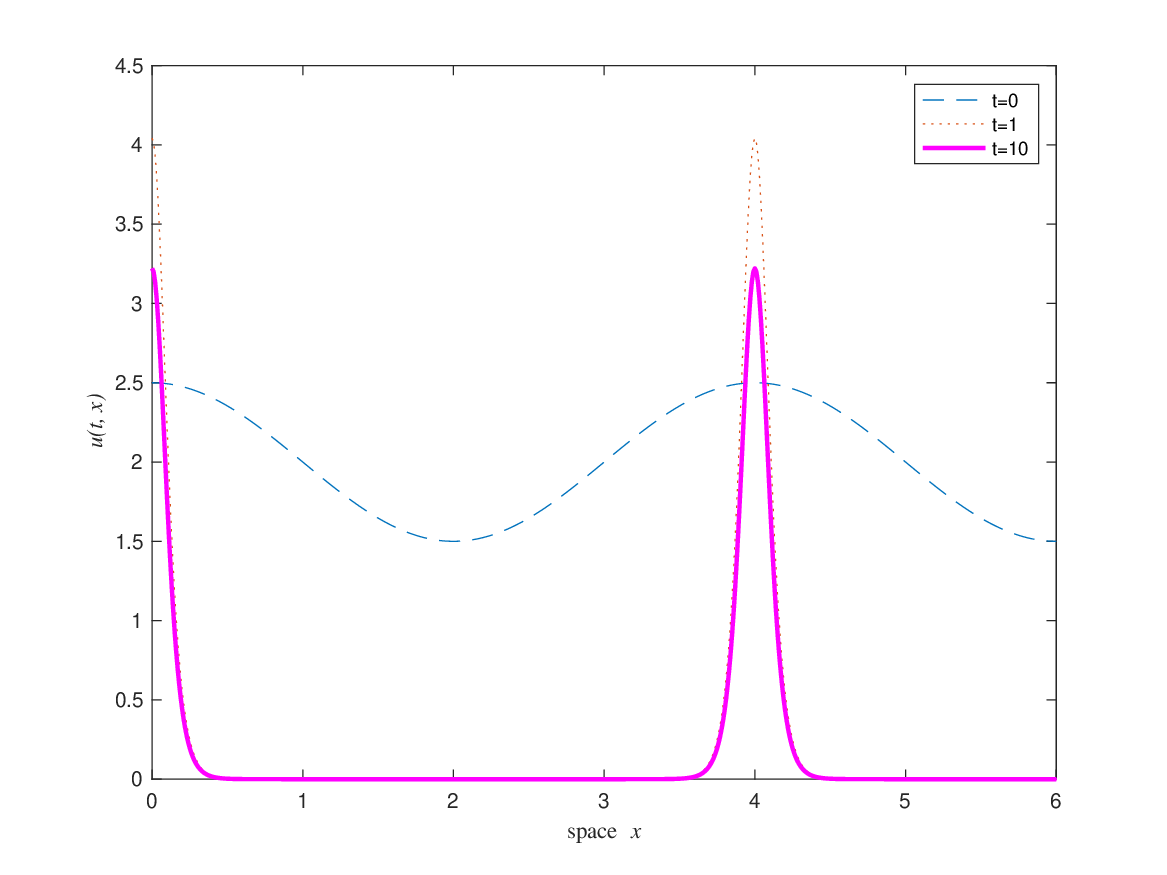}}
}
\subfigure[]{
\resizebox*{0.36\linewidth}{!}{\includegraphics{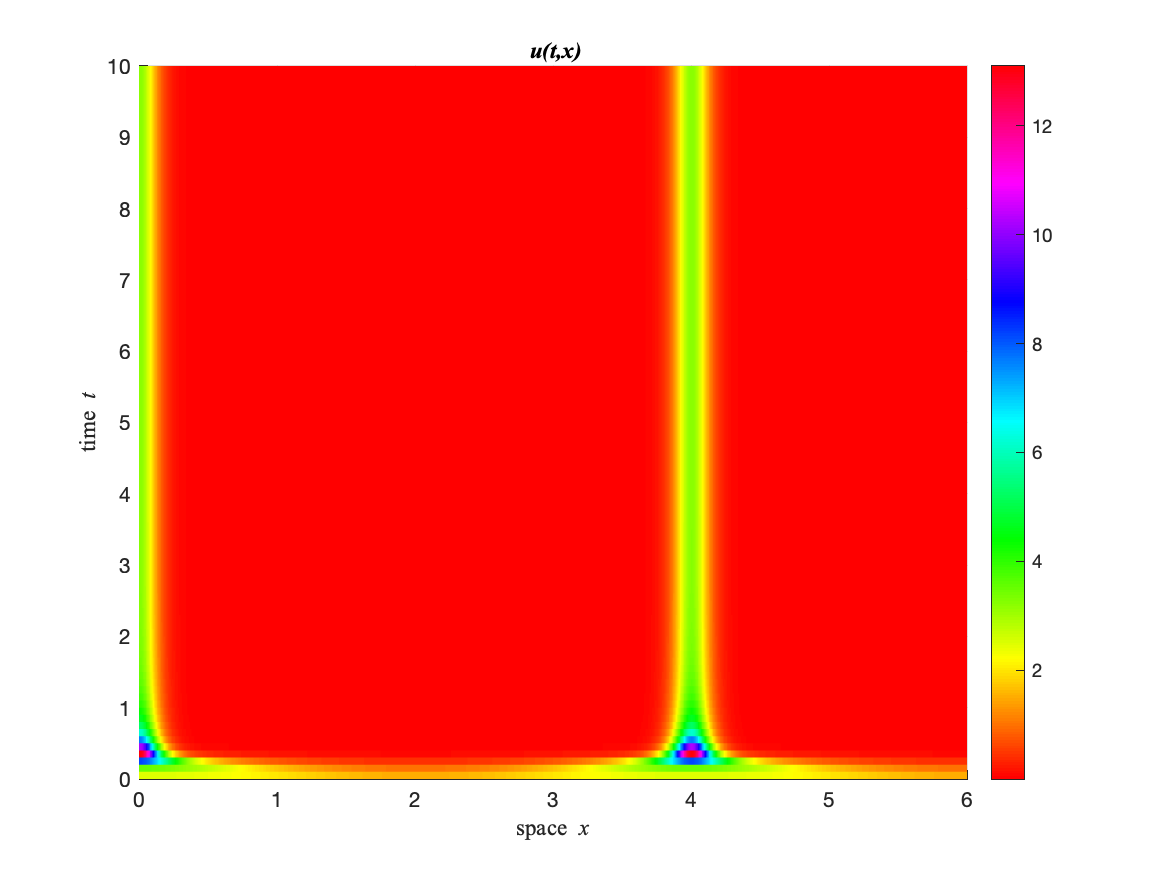}}\label{chi-10-L-6-u0-2+0.5cos0.5pix-u}}
\caption{(a) limit profile, (b) evolution of $u(t,x;u_0)$ with $\chi=10$,  $a=2$, $b=\nu=1$, $\mu=3$,   $L=6$ and initial function $u_0=2+0.5\cos(\frac{\pi x}{2})$. 
}
\label{chi-10-L-6-u0-2+0.5cos0.5pix}
\end{center}
\end{figure}

  Let  $\chi=20$ and take $u_0=2+ 0.5\cos(\frac{\pi x}{2})$. The same phenomenon as in the case $\chi=10$ is observed  (see Figure \ref{chi-20-L-6-00}).
  In particular, we observe that when $\chi$ becomes larger, the numerical nonconstant stationary solution concentrates more towards the boundary at $x=0$ and the interior at $x=4$.

\begin{figure}[!ht]
\begin{center}
\subfigure[]{
\resizebox*{0.36\linewidth}{!}{\includegraphics{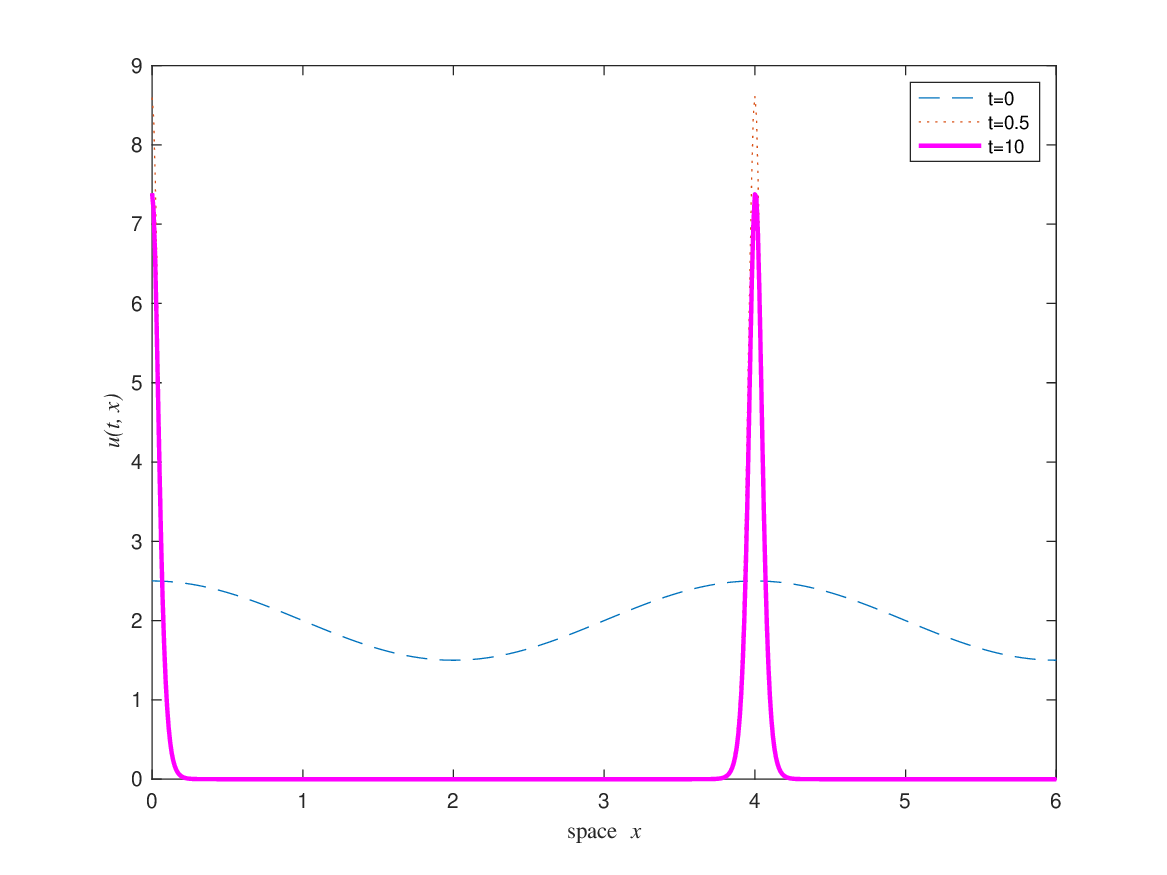}}
}
\subfigure[]{
\resizebox*{0.36\linewidth}{!}{\includegraphics{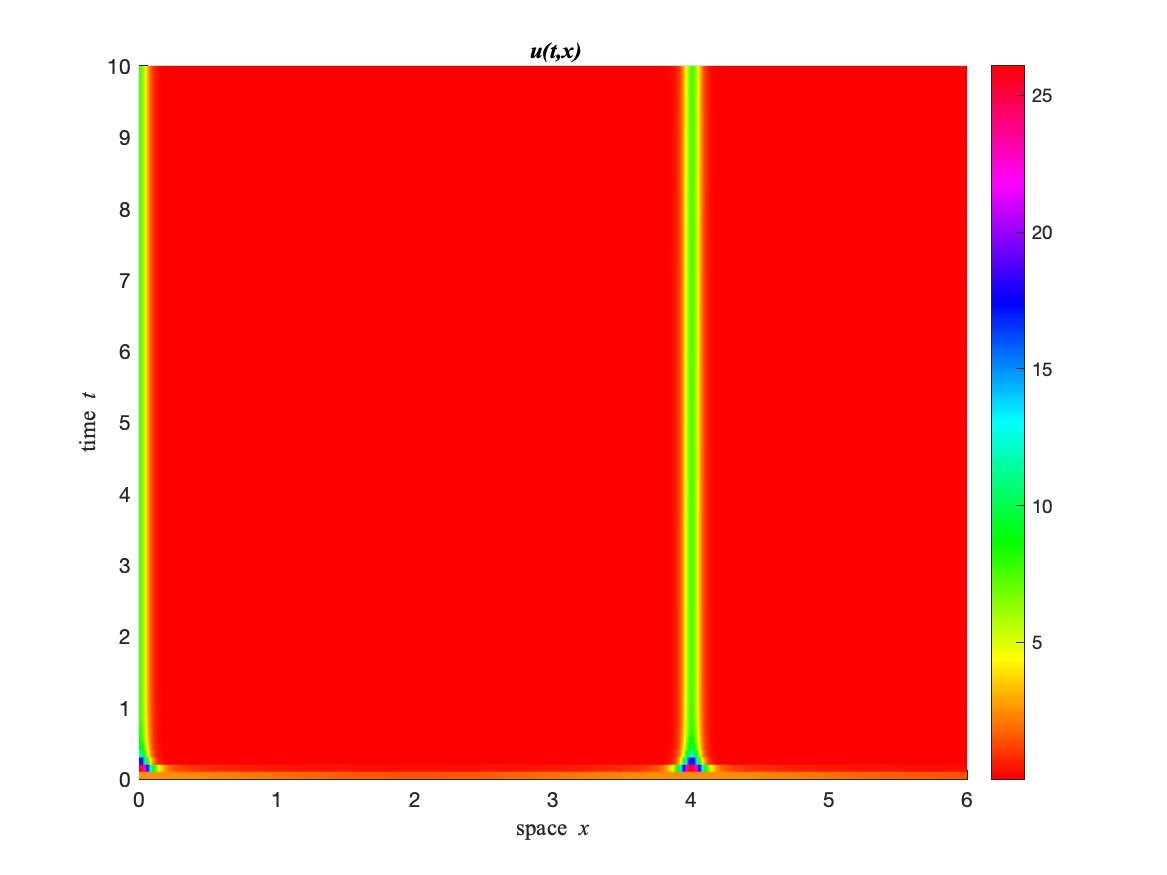}}
}
\caption{(a) limit profile, (b) evolution of $u(t,x;u_0)$   with $\chi=20$,  $a=2$, $b=\nu=1$, $\mu=3$,  $L=6$ and initial function $u_0=2+0.5\cos(\frac{\pi x}{2})$.}
\label{chi-20-L-6-00}
\end{center}
\end{figure}

\end{document}